\documentclass[article]{siamart1116}

\nonstopmode

\providecommand{\emph}[1]{{\it #1}}
\renewcommand{\emph}[1]{{\it #1}}

\usepackage{color}

\usepackage{latexsym,epsfig,bm}

\usepackage{bm}
\usepackage{upgreek}

\usepackage{mathrsfs}
\usepackage{times}
\usepackage{amssymb}

\definecolor{MyDarkBlue}{rgb}{0,0.08,0.45}
\usepackage{hyperref}
\hypersetup{pdfborder={0 0 0},
  colorlinks,
  urlcolor={MyDarkBlue},
  linkcolor={MyDarkBlue},
  citecolor={MyDarkBlue},
  breaklinks=true}
\providecommand{\url}[1]{\small\textcolor{blue}{#1}}
\usepackage{doi}
\providecommand{\eprint}[1]{}
\renewcommand{\eprint}[1]{arXiv:\href{http://arxiv.org/abs/#1}{#1}}

\usepackage{amsmath}

\providecommand{\eqref}[1]{{\rm (\ref{#1})}}
\providecommand{\itref}[1]{{\it (\ref{#1})}}

\DeclareSymbolFont{AMSb}{U}{msb}{m}{n}
\DeclareSymbolFontAlphabet{\mathbb}{AMSb}

\DeclareSymbolFont{EUR}{U}{eur}{m}{n}
\SetSymbolFont{EUR}{bold}{U}{eur}{b}{n}
\DeclareSymbolFontAlphabet{\eur}{EUR}

\DeclareSymbolFont{EUB}{U}{eur}{b}{n}
\SetSymbolFont{EUB}{bold}{U}{eur}{b}{n}
\DeclareSymbolFontAlphabet{\eub}{EUB}

\newcommand{\jj}{\mathrm{i}}
\newcommand{\End}{\,{\rm End}\,}

\newcommand\scrV{\mathscr{V}}
\newcommand\scrU{\mathscr{U}}

\newcommand{\calK}{\mathcal{K}}

\newcommand{\eurl}{\eur{l}}

\newcommand{\bmuprho}{\bm\uprho}

\newcommand{\notyet}[1]{{}}

\newcommand{\sgn}{\mathop{\rm sgn}}
\newcommand{\tr}{\mathop{\rm Tr}}

\newcommand{\p}{\partial}
\newcommand{\at}[1]{\vert\sb{\sb{#1}}}

\def\R{\mathbb{R}}
\providecommand{\C}{\mathbb{C}}
\renewcommand{\C}{\mathbb{C}}

\newcommand{\N}{\mathbb{N}}
\newcommand{\Abs}[1]{\left\vert#1\right\vert}
\newcommand{\abs}[1]{\vert #1 \vert}
\newcommand{\Norm}[1]{\Big\Vert #1 \Big\Vert}

\newcommand{\norm}[1]{\Vert #1 \Vert}

\newcommand{\sothat}{\,\,{\rm ;}\ \,}

\DeclareMathSymbol{\varGamma}{\mathord}{letters}{"00}
\DeclareMathSymbol{\varDelta}{\mathord}{letters}{"01}
\DeclareMathSymbol{\varTheta}{\mathord}{letters}{"02}
\DeclareMathSymbol{\varLambda}{\mathord}{letters}{"03}
\DeclareMathSymbol{\varXi}{\mathord}{letters}{"04}
\DeclareMathSymbol{\varPi}{\mathord}{letters}{"05}
\DeclareMathSymbol{\varSigma}{\mathord}{letters}{"06}
\DeclareMathSymbol{\varUpsilon}{\mathord}{letters}{"07}
\DeclareMathSymbol{\varPhi}{\mathord}{letters}{"08}
\DeclareMathSymbol{\varPsi}{\mathord}{letters}{"09}
\DeclareMathSymbol{\varOmega}{\mathord}{letters}{"0A}

\newtheorem{remark}{Remark}[section]
\newcounter{step} 
\providecommand{\qedhere}{}

\renewcommand{\theequation}{\thesection.\arabic{equation}}

\makeatletter\@addtoreset{equation}{section}

\makeatother

\renewcommand{\Re}{\mathop{\rm{R\hskip -1pt e}}\nolimits}

\usepackage{lipsum}
\usepackage{amsfonts}
\usepackage{graphicx}
\usepackage{epstopdf}
\usepackage{algorithmic}
\ifpdf
  \DeclareGraphicsExtensions{.eps,.pdf,.png,.jpg}
\else
  \DeclareGraphicsExtensions{.eps}
\fi

\numberwithin{theorem}{section}

\newcommand{\TheTitle}{Nonrelativistic asymptotics of solitary waves in the Dirac equation with Soler-type nonlinearity}

\newcommand{\TheAuthors}{Nabile Boussa\"id and Andrew Comech}

\headers{Nonrelativistic asymptotics of solitary waves in the Dirac equation}{\TheAuthors}

\title{{\TheTitle}
}

\author{
  Nabile Boussa\"id\thanks{Universit\'e Bourgogne Franche-Comt\'e,
                                               25030 Besan\c{c}on CEDEX, France
    (\email{nabile.boussaid@univ-fcomte.fr}).}
  \and
  Andrew Comech\thanks{
  Texas A\&M University, College Station, Texas 77843, USA;
  St.\,Petersburg State University, St.\,Petersburg 199178, Russia;
  IITP, Moscow 127051, Russia
  (\email{comech@math.tamu.edu}).}
}

\usepackage{amsopn}

\ifpdf
\hypersetup{
  pdftitle={\TheTitle},
  pdfauthor={\TheAuthors}
}
\fi

\begin{document}

\maketitle

\begin{abstract}
We use the perturbation theory to build solitary wave solutions $\phi_\omega(x)e^{-i\omega t}$ to the nonlinear Dirac equation in $\mathbb{R}^n$, $n\ge 1$, with the Soler-type nonlinear term $f(\psi^*\beta\psi)\beta\psi$, with $f(\tau)=|\tau|^k+o(|\tau|^k)$, $k>0$, which is continuous but not necessarily differentiable.  We obtain the asymptotics of solitary waves in the nonrelativistic limit $\omega\lesssim m$; these asymptotics are important for the linear stability analysis of solitary wave solutions.  We also show that in the case when the power of the nonlinearity is Schr\"odinger charge-critical ($k=2/n$), then one has $Q'(\omega)<0$ for $\omega\lesssim m$, with $Q(\omega)$ being the charge of the corresponding solitary wave; this implies the absence of the degeneracy of zero eigenvalue of the linearization at this solitary wave.
\end{abstract}




\section{Introduction}
Construction of solitary wave solutions
in Dirac-type systems has a long history.
In the three-dimensional nonlinear Dirac equation,
the solitary waves were numerically constructed
by Soler \cite{PhysRevD.1.2766}
and then proved to exist in
\cite{MR0456046,MR847126,MR949625,MR1344729}.
In the Dirac--Maxwell system,
solitary waves were obtained
numerically~\cite{MR0190520,wakano-1966,MR1364144}
and then analytically~\cite{MR1386737} (for $\omega\in(-m,0)$)
and~\cite{MR1618672} (for $\omega\in(-m,m)$);
for an overview of these results, see~\cite{MR1897689}.
A perturbation method
for the construction of solitary waves
in the nonlinear Dirac equation
was used in \cite{MR1750047}.
This work was later followed
in \cite{2008arXiv0812.2273G,MR3208458}
and also generalized
to the Einstein-Dirac and Einstein-Dirac-Maxwell systems
\cite{MR2593110,MR2647868,MR2671162}
and to the Dirac--Maxwell system \cite{dm-existence}.
Our aim here
is to make the perturbative approach
of the seminal work \cite{MR1750047} rigorous
for the important case of lower order nonlinearities.
The usefulness of such an approach
is that it gives the asymptotic behaviour
of solitary waves
which is needed for the study of their
stability properties.
The bifurcation approach
(in the nonrelativistic limit $\omega\gtrsim -m$)
to obtain Dirac--Maxwell solitary waves
as perturbations of solitary waves of the Choquard equation
was developed in~\cite{dm-existence}.

In the present analysis, we use the bifurcation approach
to construct solitary wave solutions
to the nonlinear Dirac equation
with scalar-type self-interaction,
known as the Soler model
\cite{jetp.8.260,PhysRevD.1.2766}:
\begin{eqnarray}\label{nld-0}
\jj\p\sb t\psi=D_m\psi-f(\bar\psi\psi)\beta\psi,
\qquad
\psi(t,x)\in\C\sp N,
\quad x\in\R^n,
\quad
n\ge 1,
\end{eqnarray}
where
$D\sb m=
-\jj \bm\alpha\cdot\nabla+\beta m
$
is the free Dirac operator,
with $\bm\alpha=(\alpha\sp j)\sb{1\le j\le n}$,
$\alpha\sp j$ and $\beta$
being the self-adjoint $N\times N$ Dirac matrices
(see Remark~\ref{remark-N} below for a possible choice
of such matrices);
$m>0$ is the mass.
We use the standard Physics notation
$\bar\psi:=\psi\sp\ast\beta$.
The real-valued function
\begin{eqnarray}\label{f-such}
f\in C(\R),
\qquad
f(\tau)=\abs{\tau}^k+o(\abs{\tau}^k),
\qquad
\tau\in\R,
\quad
k>0,
\end{eqnarray}
describes the nonlinearity.
We obtain solitary wave solutions
to \eqref{nld-0}
in the nonrelativistic limit,
\[
\phi\sb\omega(x)e^{-\jj\omega t},
\qquad
\phi\sb\omega\in H^1(\R^n),
\qquad
\omega\lesssim m,
\]
building them as bifurcations from solitary waves
of nonlinear Schr\"odinger equation;
the construction provides description of solitary waves
which we will need for the analysis of their spectral stability 
(presence or absence of eigenvalues with positive real part 
in the spectrum of the linearization at a solitary wave),
continuing the program
started in \cite{MR3530581}.
 We refer to that work for more details
and the background on the subject.

Most common models considered by physicists and chemists (e.g.~\cite{ranada1983classical}) are pure powers $f(\tau)=\abs{\tau}^k$,
usually cubic ($k=1$) and quintic ($k=2$).
As we already mentioned, there have been several implementations
of constructing solitary waves via the bifurcation method for such models,
but
these approaches did not allow one to handle
the low regularity case, such as
$
f(\tau)=\abs{\tau}^k,
$
with $k\in(0,1)$,
when $f(\tau)$ is no longer differentiable at $\tau=0$,
so that
the derivative of $f$ would contribute a singularity
if the Lorentz scalar
$\bar\phi\sb\omega\phi\sb\omega
:=\phi\sb\omega\sp\ast\beta\phi\sb\omega$
vanished.
On the other hand, this low regularity case also corresponds to the interesting 
``Schr\"odinger charge-subcritical'' case, 
when $k\in(0,2/n)$ (with $n\geq 2$),
so that the ``groundstate'' solitary waves
for NLS are stable
(groundstate is understood in the sense of~\cite{MR695535}:
it is a strictly positive, spherically symmetric,
decaying solution to the stationary NLS).
With these values of $k$, one can compare stability properties in both models,  
pushing further the discussion from~\cite{PhysRevLett.116.214101}.
We overcome the difficulties
resulting from the low regularity of $f$
in the nonrelativistic limit $\omega\lesssim m$,
constructing solitary waves for arbitrary $f$
from \eqref{f-such}.
The main points are to base the construction
on the Schauder fixed point theorem
(instead of the contraction mapping principle
which is not available to us when $f(\tau)$
is not Lipschitz)
and to prove that
$\phi\sb\omega(x)\sp\ast\beta\phi\sb\omega(x)$
is bounded from below by $c
\phi\sb\omega(x)\sp\ast
\phi\sb\omega(x)
$
with some $c\in(0,1)$,
for $\omega$ sufficiently close to $m$.
In the case when $f$ is differentiable away from the origin,
we will additionally prove uniqueness of $\phi\sb\omega$
(up to the symmetry transformations)
and also its differentiability with respect to $\omega$.

We note that quintic nonlinear Schr\"odinger equation
in (1+1)D and the cubic one in (2+1)D
are ``charge critical'',  in the sense that the equation has the same scaling as the charge and as a consequence all groundstate solitary waves have the same charge.
As a consequence, by~\cite{VaKo}, 
the linearization at any solitary wave
has a $4\times 4$ Jordan block at $\lambda=0$.
We mention that
there is also a blow-up phenomenon in the charge-critical
as well as in the charge-supercritical cases;
see in particular
\cite{1971ZhPmR..14..564Z,ZhETF.68.465,MR0460850,MR691044,MR1048692}.
In the present work, we will show that,
on the contrary, for the nonlinear Dirac
with the ``Schr\"odinger charge-critical'' power 
$f(\tau)=\abs{\tau}^{k}$, with $k=2/n$ (in any dimension $n\ge 1$)
the charge of solitary waves is no longer the same,
satisfying $\p\sb\omega Q(\phi\sb\omega)<0$
for $\omega\lesssim m$,
where $Q(\phi\sb\omega)=\int\sb{\R^n}
\phi\sb\omega(x)\sp\ast\phi\sb\omega(x)
\,dx$
is the corresponding charge.
This reduces the degeneracy of the zero eigenvalue
of the linearization at the corresponding solitary wave
(see e.g. \cite{MR3311594}).
In formal agreement
with the Vakhitov--Kolokolov stability criterion \cite{VaKo},
one expects that the solitary wave solutions
to the nonlinear Dirac equation
in the nonrelativistic limit $\omega\lesssim m$
are spectrally stable;
indeed,
this has been verified numerically in one-
and in two-dimensional cases
\cite{gn-stability,PhysRevLett.116.214101}.

Let us make a few more remarks on the relation to the
Vakhitov--Kolokolov stability criterion \cite{VaKo}.
In the case of self-interacting 
classical spinor fields,
although the relation of the sign of the quantity
$\p\sb\omega Q(\phi\sb\omega)$
entering the Vakhitov--Kolokolov stability criterion
and the presence or absence of positive eigenvalues
in the spectrum is no longer clear,
the vanishing of $\p\sb\omega Q(\phi\sb\omega)$,
together with the energy vanishing,
indicate the collision of point eigenvalues at the origin; for more details, see~\cite{MR3311594}.
Moreover, we point out that,
unlike in  the Schr\"odinger equation,
in the Dirac context eigenvalues with nonzero real parts can emerge
not only from the collision of purely imaginary eigenvalues at the origin,
but also from collision of purely imaginary eigenvalues
away from the origin \cite{PhysRevLett.116.214101}
and directly from the essential spectrum~\cite{PhysRevLett.80.5117}.
Thus the Vakhitov-Kolokolov criterion
is insufficient
for the characterization
of the spectral stability.

\medskip

Here is the plan of the present analysis.
The main results stated in Section~\ref{sect-results}
are the existence of solitary waves for
the case of a continuous nonlinearity
(Theorem~\ref{theorem-solitary-waves})
and the improvement for the case
of nonlinearity differentiable
everywhere except perhaps at zero
(Theorem~\ref{theorem-existence-c1}).
Theorem~\ref{theorem-solitary-waves}
is proved in Sections~\ref{sect-waves-nrl}
(the Schauder fixed point theorem),
Section~\ref{sect-shooting} (positivity of
$\bar\phi\sb\omega\phi\sb\omega
:=\phi\sb\omega\sp\ast\beta\phi\sb\omega$),
and Section~\ref{sect-improvement} (accurate estimates
on the error terms).
Theorem~\ref{theorem-existence-c1}
is proved in
Sections~\ref{sect-waves-nrl-1}
(regularity of mapping $\omega\mapsto\phi\sb\omega$)
and~\ref{sect-critical} (Vakhitov--Kolokolov condition).

The regularity of NLS solitary waves
is addressed in Appendix~\ref{sect-nls-smooth}.

\subsection*{Notations}

We denote the free Dirac operator by
\begin{eqnarray}\label{def-dm}
D\sb m
=
D\sb 0
+\beta m
=-\jj \bm\alpha\cdot\nabla+\beta m,
\qquad
m>0,
\end{eqnarray}
where
$
D_0=-\jj \bm\alpha\cdot\nabla
=
-\jj\sum\sb{j=1}\sp{n}\alpha\sp j\frac{\p}{\p x\sp j},
$
with
$\alpha\sp j$ and $\beta$
being self-adjoint $N\times N$ Dirac matrices
which satisfy
\[
(\alpha\sp j)\sp 2=\beta\sp 2=I\sb N,
\qquad
\alpha\sp j\alpha\sp k
+\alpha\sp k\alpha\sp j=2\delta\sb{j k}I\sb N,
\qquad
\alpha\sp j\beta+\beta\alpha\sp j=0,
\qquad
1\le j,k\le n.
\]
$I\sb N$ is the $N\times N$ identity matrix.
The anticommutation relations lead to e.g.
$
\tr\alpha\sp j
=\tr\beta^{-1}\alpha\sp j\beta=-\tr\alpha\sp j=0,
$
$1\le j\le n$,
and similarly $\tr\beta=0$;
together with $\sigma(\alpha\sp j)=\sigma(\beta)=\{\pm 1\}$,
this yields the conclusion that $N$ is even.

For $\psi\in\C^N$, one denotes
\[
\bar\psi=\psi\sp\ast\beta,
\]
where $\psi\sp\ast$ is the hermitian conjugate of $\psi$.

\begin{remark}\label{remark-N}
One can use the Clifford algebra representation theory
(see e.g. \cite[Chapter 1, \S5.3]{MR1401125})
to show that there is a relation
\[
N\in 2^{[\frac{n+1}{2}]}M,\qquad
M\in\N.
\]
\end{remark}

Without loss of generality,
we may assume that
the matrix $\beta$ has the following form:
\[
\beta
=
\begin{bmatrix}
I\sb{N/2}&0\\0&-I\sb{N/2}
\end{bmatrix}.
\]
Then the anticommutation relations
$\{\alpha\sp j,\beta\}=0$
show that
the matrices $(\alpha\sp j)\sb{1\leq j\leq n}$ are block-antidiagonal,
\[
\alpha\sp j
=
\begin{bmatrix}0&\upsigma\sb j\sp\ast\\\upsigma\sb j&0\end{bmatrix},
\qquad
1\le j\le n,
\]
where the matrices
$(\upsigma\sb j)\sb{1\leq j\leq n}$
satisfy
\begin{eqnarray}\label{sigma-sigma}
\upsigma\sb j\sp\ast\upsigma\sb k
+\upsigma\sb k\sp\ast\upsigma\sb j
=2\delta\sb{j k},
\qquad
\upsigma\sb j\upsigma\sb k\sp\ast
+\upsigma\sb k\upsigma\sb j\sp\ast
=2\delta\sb{j k},
\qquad
1\le j,\,k\le n.
\end{eqnarray}

\begin{remark}
The first relation in \eqref{sigma-sigma}
implies the second one (and vice versa).
Indeed,
it was pointed out to us by A. Sukhtayev
that
the identity
$\upsigma\sb j\sp\ast\upsigma\sb j
=\upsigma\sb j\upsigma\sb j\sp\ast
=I_{N/2}$
allows us to turn
the former relation in \eqref{sigma-sigma} into the latter
multiplying it by $\upsigma\sb j$ from the left
and by $\upsigma\sb j\sp\ast$ from the right.
\end{remark}

\begin{remark}
It is well-known
(see e.g. \cite{MR1821885})
how to build the larger size Dirac matrices by induction;
once we have $n+1$ self-adjoint Dirac matrices
$\alpha\sp j$, $1\le j\le n$, and $\alpha\sb{n+1}:=\beta$ in $\C^N$,
then in $\C^{2N}$ one has $n+3$ self-adjoint Dirac matrices
of the form
\[
\begin{bmatrix}0&\alpha\sb j\\\alpha\sb j&0\end{bmatrix},
\quad
1\le j\le n+1,
\qquad
\begin{bmatrix}0&-\jj I_{N}\\\jj I_{N}&0\end{bmatrix},
\qquad
\begin{bmatrix}I_{N}&0\\0&-I_{N}\end{bmatrix}.
\]
This provides the possibility to choose $N=2^{[\frac{n+1}{2}]}$.
\end{remark}


\medskip

We denote $r=\abs{x}$
for $x\in\R^n$, 
and, abusing notations,
we will also denote the operator of multiplication
with $\abs{x}$ and $\langle x\rangle=(1+\abs{x}^2)^{1/2}$
by $r$ and $\langle r\rangle$, respectively.

\medskip

The charge functional,
(formally) conserved due to the $\mathbf{U}(1)$-invariance
of \eqref{nld},
is denoted by $Q$:
\[
Q(\psi)
=\int\sb{\R^n}\psi\sp\ast(t,x)\psi(t,x)\,dx.
\]
We denote the standard $L^2$-based
Sobolev spaces of $\C\sp N$-valued functions by
$H\sp k(\R^n,\C\sp N)$.
For $s,\,k\in\R$,
we define the weighted Sobolev spaces
\[
H\sp k\sb{s} (\R\sp n,\C\sp N)=\left\{ u\in \mathscr{S}'(\R\sp n,\C\sp N),\,
\norm{u}\sb{H\sp k\sb s}
<\infty\right\},
\qquad
\norm{u}\sb{ H\sp k\sb s}=
\norm{\langle r\rangle\sp s\langle -\jj\nabla\rangle\sp k u}\sb{L\sp 2}.
\]
We write
$L\sp{2}\sb{s}(\R\sp n,\C\sp N)$
for~$H\sp{0}\sb{s}(\R\sp n,\C\sp N)$.
For $u\in L\sp 2(\R\sp n,\C\sp N)$,
we denote $\norm{u}=\norm{u}\sb{L\sp 2}$.

\bigskip

We will construct the solitary waves
in the following Banach spaces:
\begin{eqnarray}\label{def-space-x}
&&
X=L^2(\R,\abs{t}^{n-1}\mathrm{d}t;\,\C)\cap L^\infty(\R;\,\C),
\\
\nonumber
&&
\mbox{with}\quad
\norm{\cdot}\sb X
=
c\left(
\norm{\cdot}\sb{L^2(\R,\abs{t}^{n-1}\mathrm{d}t;\,\C)}
+\norm{\cdot}\sb{L^\infty(\R;\,\C)}
\right),
\end{eqnarray}
\begin{eqnarray}\label{def-space-x1}
X^1
=H^1(\R,\langle t\rangle^{n-1}\mathrm{d}t;\,\C)
=H^1\sb{(n-1)/2}(\R;\,\C)
\subset X.
\end{eqnarray}
The space $X^1$ is equipped
with the standard norm of $H^1_s(\R)$,
$s=(n-1)/2$,
while the constant $c>0$ in \eqref{def-space-x}
is chosen so that
\begin{eqnarray}\label{x-x1}
\norm{\xi}\sb{X}\le\norm{\xi}\sb{X^1},
\qquad
\forall\xi\in X^1.
\end{eqnarray}
We note that both $X$ and $X^1$
are algebras:
there is $C<\infty$ such that
\begin{eqnarray}\label{x-algebra}
&&
\norm{\xi\eta}\sb{X}
\le C\norm{\xi}\sb{X}\norm{\eta}\sb{X},
\qquad
\forall\xi,\,\eta\in X;
\\[1ex]
\label{x1-algebra}
&&
\norm{\xi\eta}\sb{X^1}
\le C\norm{\xi}\sb{X^1}\norm{\eta}\sb{X^1},
\qquad
\forall\xi,\,\eta\in X^1.
\end{eqnarray}
Abusing notations,
for $\psi=\begin{bmatrix}\psi_1\\\psi_2\end{bmatrix}$
with $\psi_1,\,\psi_2\in X$,
we also denote
\[
\norm{\psi}\sb{X}=\sqrt{\norm{\psi_1}\sb{X}^2+\norm{\psi_2}\sb{X}^2},
\]
and similarly in the case of $X^1$ instead of $X$.

The space
\[
H\sp 1\sb{e,o}(\R,\abs{t}^{n-1}\mathrm{d}t;\,\C^2)
:=
H\sb{\mathrm{even}}^1(\R,\abs{t}^{n-1}\mathrm{d}t;\C)
\times
H\sb{\mathrm{odd}}^1(\R,\abs{t}^{n-1}\mathrm{d}t;\,\C)
\]
denotes the subspace of $\C^2$-valued functions on $\R$
such that the first component is even
and the second is odd. We also denote
\[
X\sb{e,o}:=
L^2\sb{e,o}(\R, |t|^{n-1}\mathrm{d}t;\,\C\sp 2) \cap L^\infty(\R;\,\C\sp 2),
\qquad
X^1\sb{e,o}:=
H^1\sb{e,o}(\R, \langle t\rangle^{n-1}\mathrm{d}t;\,\C\sp 2).
\]



\subsection*{Acknowledgments}
 Support from the grant
ANR-10-BLAN-0101 of the French Ministry of Research is gratefully
acknowledged by the first author.

The research of Andrew Comech was carried out
at the Institute for Information Transmission Problems
of the Russian Academy of Sciences
at the expense of the Russian Foundation
for Sciences (project 14-50-00150). 
He was also partially supported by
Universit\'e Bourgogne Franche-Comt\'e.

We are grateful to the referees for their valuable remarks.

\section{Main results}
\label{sect-results}

We consider the
nonlinear Dirac equation \eqref{nld-0},
\begin{equation}\label{nld}\tag{2.1; NLDE}
 \jj \p\sb t\psi=D\sb m\psi-f(\psi\sp\ast\beta\psi)\beta\psi,
\qquad
\psi(t,x)\in\C\sp N,
\quad
x\in\R\sp n,
\end{equation}
\setcounter{equation}{1}where $D_m$ is the Dirac operator
(cf. \eqref{def-dm})
and $f\in C(\R)$ with $f(0)=0$.
The structure of the nonlinearity is such that
the equation is both $\mathbf{U}(1)$-invariant and hamiltonian,
with the hamiltonian density given by
\[
\mathscr{H}(\psi)
=
\psi\sp\ast D\sb m\psi
-F(\psi\sp\ast\beta\psi),
\]
with $F(\tau)=\int_0^\tau f(t)\,dt$,
$\tau\in\R$.







If $\phi\sb\omega(x)e^{-\jj\omega t}$ is a solitary wave solution
to \eqref{nld},
then the profile $\phi\sb\omega$ satisfies the stationary equation
\begin{eqnarray}\label{nld-stationary}
\omega\phi\sb\omega
=D\sb m\phi\sb\omega-f(\phi\sb\omega\sp\ast\beta\phi\sb\omega)\beta\phi\sb\omega.
\end{eqnarray}
In the nonrelativistic limit $\omega\lesssim m$,
the solitary waves to nonlinear Dirac equation
could be obtained as bifurcations
of the solitary wave solutions
$\varphi\sb\omega(x)e^{-\jj\omega t}$
to the nonlinear Schr\"odinger equation
\begin{eqnarray}\label{nls-k}
\jj \dot\psi=-\frac{1}{2m}\Delta\psi-\abs{\psi}\sp{2k}\psi,
\qquad
\psi(t,x)\in\C,
\quad
x\in\R\sp n.
\end{eqnarray}
By \cite{MR0454365,MR695535}
and \cite{MR734575} (for the two-dimensional case),
the stationary nonlinear Schr\"odinger equation
\begin{eqnarray}\label{def-uk}
-\frac{1}{2m}u
=-\frac{1}{2m}\Delta u
-\abs{u}\sp{2k}u,
\qquad
u(x)\in\R,
\quad
x\in\R^n,
\quad
n\ge 1
\end{eqnarray}
has a strictly positive spherically symmetric
exponentially decaying solution
$u_k\in C^2(\R^n)\cap H^1(\R^n)$
(called the groundstate)
if and only if
$0<k<2/(n-2)$
(any $k>0$ if $n\le 2$).
The linearization at the solitary wave solution
$u_k(x)e\sp{-\jj \omega t}$ with $\omega=-\frac{1}{2m}$
is given by
$
\p\sb t\bmuprho
=\begin{bmatrix}
0&\eurl\sb{-}\\-\eurl\sb{+}&0\end{bmatrix}
\bmuprho,
$
$\bmuprho(t,x)\in\C^2$,
where $\eurl\sb\pm$ are defined by
\begin{eqnarray}\label{def-l-small-pm}
\eurl\sb{-}=\frac{1}{2m}-\frac{\Delta}{2m}-u_k\sp{2k},
\qquad
\eurl\sb{+}=\frac{1}{2m}-\frac{\Delta}{2m}-(1+2k)u_k\sp{2k}.
\end{eqnarray}
By \eqref{def-uk},
the function
$u\sb{k,\lambda}(x)=\lambda\sp{1/k} u_k(\lambda x)$,
$\lambda>0$,
satisfies the identity
$
0=\frac{\lambda\sp{2}}{2m}u\sb{k,\lambda}
-\frac{1}{2m}\Delta u\sb{k,\lambda}
-u\sb{k,\lambda}\sp{1+2k}.
$
Differentiating this identity
with respect to $\lambda$ at $\lambda=1$
yields the following relation
(which we will need in Lemma~\ref{lemma-hat-v-tilde-v-p} below):
\begin{eqnarray}\label{weird}
0
=
\frac{1}{m}u_k
+\eurl\sb{+}(\p\sb\lambda\at{\lambda=1}u\sb{k,\lambda})
=\frac{1}{m}u_k
+\eurl\sb{+}\Big(\frac{1}{k}u_k+x\cdot\nabla u_k\Big).
\end{eqnarray}

We set
\begin{eqnarray}\label{Vhatdef}
\hat V(t) := u_k(\abs{t}),
\qquad
\hat U(t) :=-\frac{1}{2m}\hat V'(t),
\qquad
t\in\R,
\end{eqnarray}
where $u_k$ is considered as a function of $r=\abs{x}$,
$x\in\R^n$.
Note that the inclusion $u_k\in C^2(\R^n)$ implies that
$\hat V\in C^2(\R)$ and $\hat U\in C^1(\R)$.
By~\eqref{def-uk},
the functions $\hat V$
and $\hat U$
(which are even and odd, respectively)
satisfy
\begin{eqnarray}\label{def-hat-phi}
\frac{1}{2m}\hat V
+
\p\sb t \hat U + \frac{n-1}{t}\hat U
=|\hat V|\sp{2k}\hat V,
\qquad
\p\sb t\hat V+2m\hat U =0,
\qquad
t\in\R,
\end{eqnarray}
where $\hat U(t)/t$
at $t=0$ is understood in the limit sense,
$\lim\sb{t\to 0}\hat U(t)/t=\hat U'(0)$.
We will obtain the solitary wave solutions to \eqref{nld}
as bifurcations from $(\hat V,\,\hat U)$.

\begin{theorem}
\label{theorem-solitary-waves}
Let $n\in\N$,
$N=2^{[(n+1)/2]}$,
and assume that $f\in C(\R)$
and that there is $k>0$
such that
\begin{eqnarray}\label{ass-fo-0}
\abs{f(\tau)-\abs{\tau}^{k}}
\le o(\abs{\tau}^{k}),
\qquad
\abs{\tau}\le 1.
\end{eqnarray}
If $n\ge 3$, we additionally assume that $k<2/(n-2)$.
\begin{enumerate}
\item
\label{theorem-solitary-waves-i}
There is
\begin{eqnarray}\label{omega-large}
\omega_0\in\Big(\frac m 2,m\Big)
\end{eqnarray}
such that for all $\omega\in (\omega_0,m)$
there are solitary wave solutions
$\phi\sb\omega(x)e\sp{- \jj \omega t}$
to \eqref{nld}
with $\phi\sb\omega\in H^1(\R^n,\C^N)$,
with
\begin{eqnarray}\label{sol-forms}
\qquad
\phi\sb\omega(x)
=\begin{bmatrix}
v(r,\omega)\bm{n}\\
\jj u(r,\omega)\frac{x}{r}\cdot\bm\upsigma\,\bm{n}
\end{bmatrix},
\qquad
r=\abs{x},
\qquad
\bm{n}\in\C^{N/2},
\quad\abs{\bm{n}}=1,
\end{eqnarray}
\begin{eqnarray}\label{U-zero-zero}
\lim\sb{r\to 0}u(r,\omega)=0.
\end{eqnarray}
Moreover, if we express
\begin{eqnarray}\label{def-V-U}
  v(r,\omega) =\epsilon\sp{\frac 1 k}
  V(\epsilon r,\epsilon), \qquad
  u(r,\omega) =\epsilon\sp{1+\frac 1 k}
  U(\epsilon r,\epsilon),
\\
\nonumber
\epsilon=\sqrt{m^2-\omega^2}>0,
\qquad
r\ge 0,
\end{eqnarray}
decomposing
\begin{eqnarray}\label{def-V-U-hat}
V(t,\epsilon)=\hat V(t)+\tilde V(t,\epsilon),
\qquad
U(t,\epsilon)=\hat U(t)+\tilde U(t,\epsilon),
\\
\nonumber
t\in\R,
\quad
\epsilon>0,
\end{eqnarray}
with $\hat V(t)$, $\hat U(t)$ 
defined in \eqref{def-hat-phi}, then there is $\gamma>0$ 
such that $\tilde V(t,\epsilon)$,
$\tilde U(t,\epsilon)$ satisfy
\begin{eqnarray}\label{v-u-tilde-small-0}
\lim_{\epsilon \to 0+}\left\|
e^{\gamma \langle t\rangle}
\begin{bmatrix}\tilde V(\cdot,\epsilon)
\\
\tilde U(\cdot,\epsilon)\end{bmatrix}
\right\|\sb{H^1(\R,\C^2)}=0.
\end{eqnarray}
\item
\label{theorem-solitary-waves-ii}
There is
$\epsilon_1\in(0,\epsilon_0)$, $\epsilon_0:=\sqrt{m^2-\omega_0^2}>0$, 
such that
\begin{eqnarray}\label{UV1}
\epsilon_1\abs{U(t,\epsilon)}
\le
\frac{1}{2}\abs{V(t,\epsilon)},
\qquad
\forall t\in\R,
\quad
\forall\epsilon\in(0,\epsilon_1),
\end{eqnarray}
\begin{eqnarray}\label{phi-beta-phi-large}
\phi\sb\omega(x)\sp\ast\beta\phi\sb\omega(x)
\ge
\abs{\phi\sb\omega(x)}^2/2,
\qquad
\omega=\sqrt{m^2-\epsilon^2},
\\
\nonumber
\forall x\in\R^n,
\qquad
\forall\epsilon\in(0,\epsilon_1).
\end{eqnarray}
\item
\label{theorem-solitary-waves-iii}
One has
\begin{eqnarray}\label{v-u-tilde-smaller-hat-v}
\abs{\tilde V(t,\epsilon)}+\abs{\tilde U(t,\epsilon)}
\le
o(1)\hat V(t),
\qquad
\forall t\in\R,
\quad
\forall
\epsilon\in(0,\epsilon_1),
\end{eqnarray}
where $o(1)$ is with respect to $\epsilon$
(so that $o(1)\to 0$ as $\epsilon\to 0$) uniformly in $t$,
and
there is $b_0<\infty$ such that
\begin{eqnarray}\label{phi-asymptotics}
\abs{V(t,\epsilon)}
+
\abs{U(t,\epsilon)}
\le
b_0
\langle t\rangle^{-(n-1)/2}
e^{-\abs{t}},
\quad
\forall t\in\R,
\quad
\forall\epsilon\in(0,\epsilon_1).
\end{eqnarray}
\item
\label{theorem-solitary-waves-iv}
The solitary waves satisfy
\begin{eqnarray}\label{l-infty-to-zero}
\ \qquad
\norm{\phi\sb\omega}\sb{L^\infty(\R^n,\C^N)}
=O\big(\epsilon^{\frac{1}{k}}\big),
\ \quad
\norm{\phi\sb\omega}\sb{L^2(\R^n,\C^N)}
=O\big(\epsilon^{\frac{1}{k}-\frac{2}{n}}\big),
\ \quad
\omega\lesssim m.
\end{eqnarray}
\item
\label{theorem-solitary-waves-v}
Assume, moreover, that
there is $K>k$
such that
\begin{eqnarray}\label{ass-f-0}
\abs{f(\tau)-\abs{\tau}^{k}}
=O(\abs{\tau}^{K}),
\qquad
\abs{\tau}\le 1.
\end{eqnarray}
Then there are $ b_1,\,b_2<\infty$
such that $\tilde V(t,\epsilon)$,
$\tilde U(t,\epsilon)$ satisfy
\begin{eqnarray}\label{v-u-tilde-small-better}
\left\|
e^{\gamma \langle t\rangle}
\begin{bmatrix}\tilde V(\cdot,\epsilon)
\\\tilde U(\cdot,\epsilon)\end{bmatrix}
\right\|\sb{H^1(\R,\C^2)}
\le  b_1\epsilon\sp{2\varkappa},
\qquad
\epsilon\in(0,\epsilon_1)
\end{eqnarray}
and
\begin{eqnarray}\label{v-u-tilde-smaller-hat-v-b3}
\abs{\tilde V(t,\epsilon)}+\abs{\tilde U(t,\epsilon)}
\le
 b_2\epsilon\sp{2\varkappa}
\hat V(t),
\qquad
\forall t\in\R,
\quad
\forall
\epsilon\in(0,\epsilon_1),
\end{eqnarray}
with
\begin{eqnarray}\label{def-varkappa-1}
\varkappa=\min\Big(1,\frac{K}{k}-1\Big).
\end{eqnarray}
\end{enumerate}
\end{theorem}

\begin{remark}
We expect that, for
solitary wave solutions
$\phi\sb\omega(x)e\sp{- \jj \omega t}$
to \eqref{nld},
the profiles $\phi\sb\omega\in H^1(\R^n,\C^N)$
are continuous and thus all solitary waves of the form \eqref{sol-forms}
satisfy the condition \eqref{U-zero-zero}
(in the present article, we only prove that \eqref{U-zero-zero}
is satisfied by the family constructed
in Theorem~\ref{theorem-solitary-waves}).
\end{remark}


Theorem~\ref{theorem-solitary-waves}~\itref{theorem-solitary-waves-i}
is proved in Section~\ref{sect-waves-nrl}.
The positivity of $\bar\phi\phi$
(Theorem~\ref{theorem-solitary-waves}~\itref{theorem-solitary-waves-ii})
and the asymptotics of solitary waves
(Theorem~\ref{theorem-solitary-waves}~\itref{theorem-solitary-waves-iii})
are in Section~\ref{sect-shooting}.
The asymptotics stated in
Theorem~\ref{theorem-solitary-waves}~\itref{theorem-solitary-waves-iv}
follow from the estimates in
Theorem~\ref{theorem-solitary-waves}~\itref{theorem-solitary-waves-i}
and~\itref{theorem-solitary-waves-ii}.
The error estimates from
Theorem~\ref{theorem-solitary-waves}~\itref{theorem-solitary-waves-v}
are proved in
Section~\ref{sect-improvement}.

\begin{theorem}
\label{theorem-existence-c1}
Let $n\in\N$,
$N=2^{[(n+1)/2]}$,
and assume that
$f\in C^1(\R\setminus\{0\})\cap C(\R)$
and that there are $k>0$ and $K>k$ such that
\begin{eqnarray}
&&
\abs{f(\tau)-|\tau|^k}=O(\abs{\tau}^K),
\qquad
\quad
\abs{\tau}\le 1;
\label{ass-f-1}
\\[1ex]
&&
\abs{\tau f'(\tau)-k\abs{\tau}^k}=O(\abs{\tau}^K),
\qquad
\abs{\tau}\le 1.
\label{ass-fp-0}
\end{eqnarray}
If $n\ge 3$, we additionally assume that $k<2/(n-2)$.
There is $\epsilon_2\in(0,\epsilon_1)$
small enough
(with $\epsilon_1>0$ from Theorem~\ref{theorem-solitary-waves})
so that
for $\omega=\sqrt{m^2-\epsilon^2}$,
$\epsilon\in(0,\epsilon_2)$,
the functions
$\phi\sb\omega(x)$,
$\tilde V(t,\epsilon)$, and $\tilde U(t,\epsilon)$
from
Theorem~\ref{theorem-solitary-waves}~\itref{theorem-solitary-waves-i}
(cf. \eqref{sol-forms}--\eqref{v-u-tilde-small-0})
are unique and
satisfy the following additional properties.
\begin{enumerate}
\item
\label{theorem-existence-c1-i}
One has
$\phi\sb\omega\in H^2(\R^n,\C^N)$.
The map
\[
\omega\mapsto \phi\sb\omega\in H^1(\R^n,\C^N)
\]
is $C^1$,
with $\p\sb\omega\phi\sb\omega\in H^1(\R^n,\C^N)$.
Moreover,
$\p\sb\epsilon \tilde W(\cdot,\epsilon)\in H^1(\R,\C^2)$,
with
\begin{eqnarray}\label{norm-p-epsilon-w}
\norm{e^{\gamma\langle t\rangle}\p\sb\epsilon\tilde W(\cdot,\epsilon)}
\sb{H^1(\R,\C^2)}
=
O(\epsilon^{2\varkappa-1}),
\qquad
\epsilon\in(0,\epsilon_2),
\end{eqnarray}
where
$\tilde W(t,\epsilon)
=\begin{bmatrix}\tilde V(t,\epsilon)\\\tilde U(t,\epsilon)
\end{bmatrix}$,
and there is $c>0$ such that
\begin{eqnarray}
\label{d-phi-d-omega-bound}
\norm{\p\sb\omega\phi\sb\omega}_{L^2}^2
=c\epsilon^{-n+\frac{2}{k}}
(1+O(\epsilon^{2\varkappa})),
\quad
\omega=\sqrt{m^2-\epsilon^2},
\ \ \epsilon\in(0,\epsilon_2).
\end{eqnarray}
\item
\label{theorem-existence-c1-ii}
Additionally, assume that
$k,\,K$ from~\eqref{ass-f-1} and~\eqref{ass-fp-0} satisfy
either
\begin{eqnarray}\label{subcritical}
k<2/n
\end{eqnarray}
or
\begin{eqnarray}\label{critical}
k=2/n,
\qquad
K>4/n.
\end{eqnarray}
Then there is $\omega\sb\ast<m$ such that
$\p\sb\omega Q(\omega)<0$
for all  $\omega\in(\omega\sb\ast,m)$.

If
\begin{eqnarray}\label{supercritical}
k>2/n,
\end{eqnarray}
then there is
$\omega\sb\ast<m$
such that
$\p\sb\omega Q(\omega)>0$
for all $\omega\in(\omega\sb\ast,m)$.
\end{enumerate}
\end{theorem}

\begin{remark}
The absolute value
in the expansion
$f(\tau)=\abs{\tau}^k+\dots$
is needed in the case when $k>0$ is not an integer.
We note that if $k\in\N$ and is odd, then,
with and without the absolute value,
one arrives at two different models; for example,
in the model
\begin{eqnarray}\label{eq:nld-cubic}
\jj\dot\psi=D_m\psi-\bar\psi\psi\,\beta\psi,
\qquad
\psi(t,x)\in\C^4,
\quad
x\in\R^3,
\end{eqnarray}
where $\bar\psi=\psi\sp\ast\beta$,
the small amplitude limit corresponding to $\omega\to -m$
is a defocusing NLS
(contrary to the small amplitude limit when $\omega\to m$
which is a focusing NLS),
while in the model
\begin{eqnarray}\label{eq:nld-pseudocubic}
\jj\dot\psi=D_m\psi-|\bar\psi\psi|\beta\psi,
\qquad
\psi(t,x)\in\C^4,
\quad
x\in\R^3,
\end{eqnarray}
such a limit is a focusing NLS
(just like the small amplitude limit when $\omega\to m$).
We point out that
both equations \eqref{eq:nld-pseudocubic} and \eqref{eq:nld-cubic}
are Hamiltonian systems
and both are invariant
with respect to the Wigner time reversal
\cite[Chapter 5.4]{MR0187641}:
\begin{eqnarray}\label{time-reversal-1}
\psi(t,x)\mapsto\psi_{T}(t,x)
=\jj\gamma^1\gamma^3 \bm{K}\psi(-t,x)
=
-\begin{bmatrix}
\sigma_2&0\\0&\sigma_2
\end{bmatrix}
\bm{K}\psi(-t,x),
\end{eqnarray}
with $\bm{K}:\C^N\to\C^N$ denoting the complex conjugation.
We mention that
$\bar\psi_{T}\psi_{T}=\bar\psi\psi$.

Moreover,
equation \eqref{eq:nld-pseudocubic}
remains invariant under the time reversal from
\cite[Section 2.5.7]{MR1219537},
which is a combination of the Wigner time reversal
and the charge conjugation
$\psi(t,x)\mapsto \psi\sb C(t,x)=-\jj\gamma^2\bm{K}\psi(t,x)$:
\begin{eqnarray}\label{time-reversal-2}
\psi(t,x)\mapsto\psi_{CT}(t,x)
=
-\jj\beta\gamma^5\psi(-t,x)
=
\begin{bmatrix}
0&-i I_2\\i I_2&0
\end{bmatrix}
\psi(-t,x)
.
\end{eqnarray}
Above,
$\gamma^5=\jj\gamma^0\gamma^1\gamma^2\gamma^3
=\begin{bmatrix}0&I_2\\I_2&0\end{bmatrix}
$.
Since $\bar\psi_{CT}\psi_{CT}=\bar\psi_{C}\psi_{C}=-\bar\psi\psi$,
equation \eqref{eq:nld-cubic}
is going to be invariant with respect to this transformation
if the nonlinearity $f$ in
\eqref{nld}
is even: $f(\tau)=f(-\tau)$, $\tau\in\R$.
%
\end{remark}

\begin{remark}
If $f(\tau)$ in \eqref{nld} is even, then,
applying to the solitary waves the time reversal
\eqref{time-reversal-2},
we see that
there is a symmetry $\omega\leftrightarrow-\omega$ of solitary
waves:
if $\phi(x)e^{-\jj\omega t}$ is a solitary wave solution to \eqref{nld},
then so is
$-\jj\beta\gamma^5\phi(x)e^{\jj\omega t}$.
More generally, in any dimension $n\ge 1$,
if $f$ in \eqref{nld} is even,
then,
given a solitary wave solution $\phi\sb\omega(x) e^{-\jj\omega t}$
to \eqref{nld}
with $\phi\sb\omega(x)$ as in \eqref{sol-forms},
there is also a solitary wave solution
\[
\begin{bmatrix}
u(r,\omega)\frac{x}{r}\cdot\bm\upsigma\sp\ast\,\bm{n}
\\
\jj v(r,\omega)\bm{n}
\end{bmatrix}
e^{\jj\omega t},
\qquad
\bm{n}\in\C^{N/2},
\quad
\abs{\bm{n}}=1.
\]
\end{remark}

\begin{remark}
Let us mention
that
the stationary waves $\phi_\omega e^{-i\omega t}$ constructed in \cite{MR847126}
in the three-dimensional case
satisfy $\bar\phi_\omega\phi_\omega>0$ for all $x\in\R^3$,
and
the same is true for the families of solitary waves
that we construct in Theorem~\ref{theorem-solitary-waves}
(cf. \eqref{phi-beta-phi-large});
such solitary waves will be solutions to
\eqref{nld} if $f(\tau)$ is substituted by $f(\abs{\tau})$.
\end{remark}

\begin{remark}
By \cite{MR695536},
the pure power stationary Schr\"odinger equation \eqref{def-uk}
with $n\ge 3$
has infinitely many distinct radial solutions,
and one expects that
there is a family of solitary waves of \eqref{nld}
bifurcating from any of these radial solutions
(similarly to what we state in Theorem~\ref{theorem-solitary-waves}).
\end{remark}

\begin{remark}
Let us also consider the following question:
{\it
Given a sequence of solitary wave solutions
corresponding to $\omega\sb j\to m$,
does this sequence
(up to symmetries and extraction of a subsequence)
always converge to a solution of a nonlinear Schr\"odinger equation,
in the sense of the above lemma?
}
The answer to this question is negative in general.
One obstacle can be illustrated
as follows.
In particular, in dimension $n=3$,
according to \cite[Theorem 1]{MR1344729},
there are solitary wave solutions
to \eqref{nld}
for the pure power nonlinearity with $k=1$ or any real $k\ge 2$
(so that $\abs{\tau}^{k+1}$ remains $C^2$ at $\tau=0$,
meeting the assumptions of \cite{MR1344729});
see also earlier works \cite{MR847126,MR949625}.
On the other hand,
if $k\ge \frac{2}{n-2}=2$,
the nonrelativistic limit can not converge to
a stationary solution of the nonlinear Schr\"odinger equation,
which does not exist for such values of $k$.

We do not know whether
in the case $k\leq \frac{2}{n-2}$,
any sequence of solitary waves $\phi\sb\omega$,
$\omega\lesssim m$,
could be obtained as a bifurcation
from an NLS solitary wave.
\end{remark}

\medskip

Theorem~\ref{theorem-existence-c1}~\itref{theorem-existence-c1-i}
is proved in Section~\ref{sect-waves-nrl-1},
and
the Vakhitov--Kolokolov inequality
in the critical case
(Theorem~\ref{theorem-existence-c1}~\itref{theorem-existence-c1-ii})
is analyzed in Section~\ref{sect-critical}.

\section{Solitary waves in the nonrelativistic limit.
The case $f\in C$}
\label{sect-waves-nrl}

In this section,
we prove Theorem~\ref{theorem-solitary-waves},
constructing a particular family of solitary waves
bifurcating from solitary waves
of the nonlinear Schr\"odinger equation.

First of all, we need to rewrite the assumption
$f(\tau)=\abs{\tau}^k+o(\abs{\tau}^k)$
in a more convenient form.
Fix $k>0$ (with $k<2/(n-2)$ if $n\ge 3$).
For $\hat V$, $\hat U$ from \eqref{Vhatdef}, let us denote
\begin{eqnarray}\label{def-Lambda}
\Lambda\sb k
:=\sup\sb{x\in\R^n}
\abs{\hat V(x)}
+
m\sup\sb{x\in\R^n}\abs{\hat U(x)}
<\infty.
\end{eqnarray}
We focus on solitary waves
with $\tilde V(t,\epsilon)$, $\tilde U(t,\epsilon)$
from \eqref{def-V-U-hat}
(we recall that $\epsilon = \sqrt{m^2-\omega^2}$),
satisfying
\begin{eqnarray}\label{UV0-0}
\abs{\tilde V(t,\epsilon)}
+m\abs{\tilde U(t,\epsilon)}
<\Lambda\sb k,
\qquad
\forall t\in\R;
\end{eqnarray}
we will see below that
this imposes certain smallness assumptions onto $\epsilon>0$.
It follows from
\eqref{def-Lambda} and \eqref{UV0-0}
that
\begin{eqnarray}\label{no-more}
\abs{V(t,\epsilon)}
\le 2\Lambda\sb k,
\qquad
m\abs{U(t,\epsilon)}
\le 2\Lambda\sb k,
\qquad
\abs{V(t,\epsilon)^2-\epsilon^2 U(t,\epsilon)^2}
<4\Lambda\sb k^2,
\\
\nonumber
\forall t\in\R.
\end{eqnarray}

In the present analysis,
we build small amplitude solitary waves,
and the proof below would not be affected
by a change of the nonlinearity $f(\tau)$
outside of an open neighborhood of $\tau=0$,
hence, by \eqref{ass-fo-0},
we could assume that
\begin{eqnarray}\label{ass-fo-1}
\abs{f(\tau)}\le 2\abs{\tau}^{k},
\qquad
\tau\in\R,
\end{eqnarray}
and that
\begin{eqnarray}\label{ass-fo-whatever}
\abs{f(\tau)-\abs{\tau}^{k}}
\le \abs{\tau}^k H(\tau),
\qquad
\tau\in\R,
\end{eqnarray}
where $H\in C(\R)$ is monotonically increasing
for $\tau\ge 0$, with $H(0)=0$.
It will be convenient for us to define
\begin{eqnarray}\label{def-h}
h(\epsilon):=
\max\left(
H(\epsilon^{2/k}4\Lambda\sb k^2),
\,\epsilon^{2k},
\,\epsilon^2
\right).
\end{eqnarray}
Note that, by \eqref{no-more},
\begin{eqnarray}\label{hh}
H(v^2-u^2)
=
H(\epsilon^{2/k}(V^2-\epsilon^2 U^2))
\le
H\big(\epsilon^{2/k}4\Lambda\sb k^2\big)
\le h(\epsilon);
\end{eqnarray}
from \eqref{ass-fo-whatever}
and \eqref{hh}
we obtain the following convenient estimate for later use:
\begin{eqnarray}\label{ass-fo}
\big|
f\big(\epsilon^{2/k}(V^2-\epsilon^2 U^2)\big)
-\epsilon^2\abs{V^2-\epsilon^2 U^2}^{k}
\big|
\le C\epsilon^2\abs{V^2-\epsilon^2 U^2}^{k}
h(\epsilon),
\end{eqnarray}
with $h(\epsilon)$ 
continuous, monotonically increasing
for $\epsilon\ge 0$, with $h(0)=0$.
\medskip

Now we are ready to start the proof of Theorem~\ref{theorem-solitary-waves}.
We extend the argument of~\cite[Section 4.2]{MR3208458}.
Substituting the Ansatz \eqref{sol-forms} into the nonlinear
Dirac equation~\eqref{nld} gives
the system
\begin{eqnarray}\label{soleq-nd}
\left\{ \begin{array}{l}
\p\sb r u + \frac{n-1}{r} u
+(m-\omega) v
=f(v\sp 2-u\sp 2)v,
\\
\p\sb r v+(m+\omega )u
=f(v\sp 2-u\sp 2)u,
\end{array}\right.
\qquad
r>0,
\end{eqnarray}
for the pair of real-valued functions $v = v(r,\omega)$,
$u = u(r,\omega)$.
We will always impose the condition
\begin{eqnarray}\label{u-zero-zero}
\lim\sb{r\to 0}u(r,\omega)=0
\end{eqnarray}
(cf. \eqref{U-zero-zero});
this allows us to extend $v(r,\omega)$ and $u(r,\omega)$
continuously onto $\R$
so that $v$ is even and $u$ is odd:
\begin{eqnarray}\label{v-even-u-odd}
v(r,\omega)=v(-r,\omega),
\qquad
u(r,\omega)=-u(-r,\omega),
\qquad
r\le 0.
\end{eqnarray}
Then \eqref{soleq-nd}
extends onto the whole real axis:
\begin{eqnarray}\label{soleq-nd-t}
\left\{ \begin{array}{l}
\p\sb r u + \frac{n-1}{r}u
+(m-\omega)v
=f(v\sp 2-u\sp 2)v,
\\
\p\sb r v+(m+\omega )u
=f(v\sp 2-u\sp 2)u,\\
u\at{r=0}=0,
\end{array}\right.
\qquad
r\in\R.
\end{eqnarray}
In \eqref{soleq-nd-t},
the term
$\frac{u(r,\omega)}{r}$ at $r=0$ is understood
as the limit $\lim\sb{r\to 0}\frac{u(r,\omega)}{r}=\p\sb r u(0,\omega)$.

By \eqref{soleq-nd-t}, $V$ and $U$
from \eqref{def-V-U}
are to satisfy
\[
\left\{ \begin{array}{l}
\displaystyle
\epsilon\sp 2
\Big(\p\sb t U + \frac{n-1}{t} U\Big)
+(m-\omega)V
=f V,
\\
\p\sb t V
+(\omega+m)U
=f U,
\\
U\at{t=0}=0,
\end{array}\right.
\qquad
t\in\R,
\]
with $t = \epsilon r$
and with
\[
f=f\big(
\epsilon^{2/k}\big(V(t,\epsilon)^2-\epsilon^2 U(t,\epsilon)^2\big)
\big).
\]
According to \eqref{v-even-u-odd},
$V(t,\epsilon)$ is even in $t\in\R$
and $U(t,\epsilon)$ is odd.
The term $U/t$ at $t=0$
is understood as the limit
$\lim\sb{t\to 0}U(t,\epsilon)/t=\p\sb t U(0,\epsilon)$.
We rewrite the above system as
\begin{eqnarray}\label{zero-is-phi1}
\left\{ \begin{array}{l}
\displaystyle
\p\sb t U + \frac{n-1}{t} U
+\frac{1}{m+\omega} V
=\frac{f}{\epsilon^2}V,
\\
\p\sb t V
+(m+\omega)U
=f U,
\\
U\at{t=0}=0,
\end{array}\right.
\qquad
t\in\R.
\end{eqnarray}
We note that the system \eqref{def-hat-phi}
corresponds to the limit of \eqref{zero-is-phi1}
as $\epsilon\to 0$ (that is, $\omega\to m$)
after the substitution \eqref{def-V-U-hat}.
For sufficiently small $\epsilon>0$,
we will construct the solution $(V,U)$
as a bifurcation from $(\hat V,\hat U)$.

Substituting
$
V(t,\epsilon)=\hat{V}(t)+\tilde{V}(t,\epsilon)
$
and
$
U(t,\epsilon)=\hat{U}(t)+\tilde{U}(t,\epsilon)
$
into \eqref{zero-is-phi1}
and then subtracting equations~\eqref{def-hat-phi},
we arrive at
\begin{eqnarray}\label{w-a-a}
\begin{cases}
(\p\sb t + \frac{n-1}{t}) \tilde{U}
+\frac{1}{m+\omega}\tilde{V}
=
(1+2k) |\hat V|\sp{2k}\tilde{V}
-G_1(\epsilon,\tilde V,\tilde U),
\\[1ex]
\p\sb t \tilde V
+(m+\omega)\tilde U =
G_2(\epsilon,\tilde V,\tilde U),
\end{cases}
t\in\R,
\ \ \epsilon>0,
\end{eqnarray}
where
\begin{eqnarray}\label{def-g1}
G_1(\epsilon,\tilde V,\tilde U)
=
-\epsilon^{-2}f\big(\epsilon^{2/k}(V\sp 2- \epsilon\sp 2 U^2)\big)V
+
\hat V\sp{2k}\hat V
+(1+2k)\hat V\sp{2k}\tilde{V}
\\
\nonumber
+
\left(\frac{1}{m+\omega}-\frac{1}{2m}\right)\hat V,
\end{eqnarray}
\begin{eqnarray}\label{def-g2}
G_2(\epsilon,\tilde V,\tilde U)
=
f\big(\epsilon^{2/k}(V\sp 2- \epsilon\sp 2 U^2)\big)U
+(m-\omega)\hat U,
\qquad
\omega=\sqrt{m^2-\epsilon^2},
\end{eqnarray}
with \eqref{def-V-U-hat}
giving the relations
between $V$, $U$ and $\tilde V$, $\tilde U$.
Let us denote
\begin{eqnarray}\label{def-g}
G(\epsilon, \tilde W)=
\begin{bmatrix}
G_1(\epsilon,\tilde V,\tilde U)
\\
G_2(\epsilon,\tilde V,\tilde U)
\end{bmatrix},
\qquad
\tilde W
=
\begin{bmatrix}
\tilde V\\ \tilde U
\end{bmatrix},
\end{eqnarray}
and introduce the operator
\begin{eqnarray}\label{a-epsilon}
A(\epsilon)=
\left[
\begin{array}{cc}
-\frac{1}{m+\omega}
+(1+2k) |\hat V|\sp{2k}
&
-\p\sb t-\frac{n-1}{t}\\
\p\sb t &m+\omega
\end{array}\right],
\qquad
\omega=\sqrt{m^2-\epsilon^2},
\end{eqnarray}
defined for $\epsilon\ge 0$,
with the domain
\[
D(A(\epsilon))
=
H\sp 1\sb{e,o}(\R,\abs{t}^{n-1}\mathrm{d}t;\,\C^2),
\]
where
\[
H\sp 1\sb{e,o}(\R,\abs{t}^{n-1}\mathrm{d}t;\,\C^2)
:=
H\sb{\mathrm{even}}^1(\R,\abs{t}^{n-1}\mathrm{d}t;\C)
\times
H\sb{\mathrm{odd}}^1(\R,\abs{t}^{n-1}\mathrm{d}t;\,\C)
\]
denotes the subspace of $\C^2$-valued functions on $\R$
such that the first component is even
and the second is odd.
We similarly define the space
$L^2\sb{e,o}(\R,\abs{t}^{n-1}\mathrm{d}t;\C\sp 2)$
and note that
\[
A(\epsilon):\;
H^1\sb{e,o }(\R, |t|^{n-1}\mathrm{d}t;\,\C\sp 2)
\to
L^2\sb{e,o}(\R, |t|^{n-1}\mathrm{d}t;\,\C\sp 2).
\]
Now the system \eqref{w-a-a}
takes the form
\begin{eqnarray}\label{w-a-a-3}
A(\epsilon)\tilde W(t,\epsilon)=G(\epsilon,\tilde W(t,\epsilon)),
\qquad
\epsilon>0.
\end{eqnarray}
Notice that the differential
operator $A(\epsilon)$, $\epsilon\in[0,m]$,
is self-adjoint 
on $H\sp 1\sb{e,o}(\R,\abs{t}^{n-1}\mathrm{d}t;\C^2)$.
We also notice that
the essential spectrum of
$A(\epsilon)$, $\epsilon\in[0,m]$,
with $\hat V$ substituted by zero is
$(-\infty,-\frac{1}{m+\omega}]\cup[m+\omega,+\infty)$ (see \cite[Satz 2.1]{MR664527} or \cite[Theorem 4.18]{MR1219537}).
Applying Weyl's criterion~\cite[Corollary 2 of Theorem XIII.14]{MR0493421},
we deduce that the essential spectrum of $A(\epsilon)$
is also given by
\[
\sigma\sb{\mathrm{ess}}\big(A(\epsilon)\big)
=
\Big(-\infty,-\frac{1}{m+\omega}\Big]
\cup
\Big[m+\omega,+\infty\Big),
\qquad
\epsilon\in[0,m].
\]
Since the inclusion
$
\left[\begin{array}{c}\xi\\\eta\end{array}\right]\in\ker A(0)$
would lead to $\eta(t) = -\frac{1}{2m}\xi'(t)$
for $t\in\R$ and then to
$
\xi(\abs{x}) \in \ker\eurl\sb{+},
$
with $\eurl\sb{+}$ defined in \eqref{def-l-small-pm}
and $x\in\R^n$,
while
the restriction of $\eurl\sb{+}$
to spherically symmetric functions has zero kernel
(see~\cite[Proof of Lemma 2.1, case $k=0$]{MR2368894}),
we see that $\ker A(0)\at{H^1\sb{e,o}(\R, |t|^{n-1}\mathrm{d}t;\,\C\sp 2)} = \{ 0 \}$.
Thus, $\lambda=0$
does not belong to the spectrum of
$A(0)\at{L^2\sb{e,o}}$,
hence $A(0)\sp{-1}$ is bounded from
$L^2\sb{e,o}(\R, |t|^{n-1}\mathrm{d}t;\,\C\sp 2)$
to $H^1\sb{e,o}(\R, |t|^{n-1}\mathrm{d}t;\,\C\sp 2)$.
By continuity in $\epsilon$ in the norm resolvent sense,
there is $\epsilon_0>0$ such that
the mapping
\begin{eqnarray}\label{a-bounded-0}
A(\epsilon)^{-1}:\;
L^2\sb{e,o}(\R, |t|^{n-1}\mathrm{d}t;\,\C\sp 2)
\to H^1\sb{e,o}(\R, |t|^{n-1}\mathrm{d}t;\,\C\sp 2),
\qquad
\epsilon\in[0,\epsilon_0]
\end{eqnarray}
is continuous,
with the norm bounded uniformly in $\epsilon\in[0,\epsilon_0]$.

We actually need a stronger statement
on continuity of $A^{-1}$ in the following spaces
(cf. \eqref{def-space-x}, \eqref{def-space-x1}):
\[
X\sb{e,o}:=
L^2\sb{e,o}(\R, |t|^{n-1}\mathrm{d}t;\,\C\sp 2) \cap L^\infty(\R;\,\C\sp 2),
\qquad
X^1\sb{e,o}:=
H^1\sb{e,o}(\R, \langle t\rangle^{n-1}\mathrm{d}t;\,\C\sp 2),
\]
with the norms
$\norm{(\xi_1,\xi_2)}\sb{X\sb{e,o}}^2
=\norm{\xi_1}\sb{X}^2
+\norm{\xi_2}\sb{X}^2$
for $(\xi_1,\xi_2)\in X\sb{e,o}$
and
$\norm{(\xi_1,\xi_2)}\sb{X^1\sb{e,o}}^2
=\norm{\xi_1}\sb{X^1}^2
+\norm{\xi_2}\sb{X^1}^2$
for $(\xi_1,\xi_2)\in X^1\sb{e,o}$.
Abusing the notations,
we will denote these norms
by $\norm{\cdot}\sb{X}$ and $\norm{\cdot}\sb{X^1}$,
respectively.

\begin{lemma}\label{lemma-a-invertible}
The restriction of the mapping \eqref{a-bounded-0}
to $X\sb{e,o}\subset L^2\sb{e,o}(\R,\abs{t}^{n-1}\mathrm{d}t;\C^2)$
defines a continuous map
\begin{eqnarray}\label{a-bounded}
A(\epsilon)^{-1}:\;X\sb{e,o}\to X^1\sb{e,o},
\qquad
\epsilon\in[0,\epsilon_0],
\end{eqnarray}
with the norm
bounded uniformly in $\epsilon\in[0,\epsilon_0]$.
\end{lemma}

\begin{proof}
The uniform continuity in $\epsilon$ will follow as in the previous case from the resolvent identity.
Due to the continuity of the mapping \eqref{a-bounded-0},
we already know that for any
$(b,a)
\in L^2\sb{e,o}(\R,|t|^{n-1}\mathrm{d}t;\,\C\sp 2)\cap L^\infty(\R;\,\C\sp 2)$
the solution of
\begin{eqnarray}\label{a-v-b}
A(\epsilon)\begin{bmatrix}v\\u\end{bmatrix}
=\begin{bmatrix}b\\a\end{bmatrix}
\in L^2\sb{e,o}(\R,|t|^{n-1}\mathrm{d}t;\,\C\sp 2)\cap L^\infty(\R;\,\C\sp 2)
\end{eqnarray}
in $L^2\sb{e,o}(\R,|t|^{n-1}\mathrm{d}t;\,\C\sp 2)$
satisfies $(v,u)\in H^1\sb{e,o}(\R, |t|^{n-1}\mathrm{d}t;\,\C\sp 2)$.

In the case $n=1$, we are done.

In the case $n\ge 2$, we proceed as follows.
We already know that
\[
(v,u)\in H^1\sb{e,o}(\R, |t|^{n-1}\mathrm{d}t;\,\C\sp 2)
\subset H^1\sb{e,o}(\R\setminus[-1,1],\langle t\rangle^{n-1}\mathrm{d}t;\C\sp 2)
.
\]
It suffices to prove that
$(v,u)$ also satisfies
\[
(v,u)\in L^\infty([-1,1];\,\C\sp 2),
\qquad
(v',u')\in L^\infty([-1,1];\,\C\sp 2),
\]
with the norms bounded by
$\norm{(a,b)}\sb{L^2(\R,\abs{t}^{n-1}\mathrm{d}t;\,\C^2)}
+\norm{(a,b)}\sb{L^\infty(\R;\,\C^2)}$
(times a constant factor).
Equation \eqref{a-v-b} can be written out as
the following system:
\begin{eqnarray}\label{a-v-b-s}
\begin{cases}
\Big((1+2k)\hat V(t)^{2k}-\frac{1}{m+\omega}\Big)v(t)
-\p\sb t u-\frac{n-1}{t}u(t)=b(t),
\\
\p\sb t v+(m+\omega)u(t)=a(t).
\end{cases}
\end{eqnarray}
From $(v,u) \in H^1\sb{e,o}(\R, |t|^{n-1}\mathrm{d}t;\,\C\sp 2)$  we deduce that
\[
v,\,u\in C(\R\setminus\{0\}),
\]
and that
$|t|^\frac{n-1}{2}v\in L^\infty([-1,1])$
and $|t|^\frac{n-1}{2}u\in L^\infty([-1,1])$, as a consequence of Sobolev inequality and Hardy inequalities (see \cite[Appendix A.4 ($r=1$ and $p=2$)]{MR0290095} for the later)
 and moreover
\begin{eqnarray}\label{assume-such-alpha-0}
\||t|^\frac{n-1}{2}v\|_{L^\infty([-1,1])}+\||t|^\frac{n-1}{2}u\|_{L^\infty([-1,1])}
\leq 
C\|(v,u)\|_{H^1(\R, |t|^{n-1}\mathrm{d}t;\,\C\sp 2)}
\nonumber
\\[1ex]
\leq
C'\|(b,a)\|_{L^2(\R, |t|^{n-1}\mathrm{d}t;\,\C\sp 2)},
\end{eqnarray}
with some $C,\,C'<\infty$.

We will proceed by induction;
let us assume that, more generally,
\begin{eqnarray}\label{assume-such-alpha}
&&
\||t|^\alpha v\|_{L^\infty([-1,1])}+\||t|^\alpha u\|_{L^\infty([-1,1])}
\nonumber
\\
&&
\qquad
\qquad
\leq
C
\left(
\|(b,a)\|_{L^2(\R, |t|^{n-1}\mathrm{d}t;\,\C\sp 2)}
+
\|(b,a)\|_{L^\infty(\R;\,\C\sp 2)}
\right),
\end{eqnarray}
with $C<\infty$ independent on $(a(t),b(t))$
and with some $\alpha\in[1/2,(n-1)/2]$.
Note that the upper bound is meaningless as $t$ is bounded but
we already know by \eqref{assume-such-alpha-0}
that \eqref{assume-such-alpha}
holds with $\alpha=(n-1)/2$.
The first equation from \eqref{a-v-b-s} can be rewritten as
\begin{eqnarray}\label{r-as}
 \p\sb t \big(t^{n-1}u\big)
=t^{n-1}
\Big((1+2k)\hat V(t)^{2k}-\frac{1}{m+\omega}\Big)v(t)
-t^{n-1}b(t).
\end{eqnarray}
Since
$u\in H^1(\R,\abs{t}^{n-1}\mathrm{d}t)
\subset C(\R\setminus\{0\})$,
$|t|^\frac{n-1}{2}u\in L^\infty([-1,1])$,
and
$n\ge 2$,
one has
$t^{n-1}u(t)\to 0$ as $t\to 0$;
therefore, integrating the relation \eqref{r-as},
we arrive at
\[
t^{n-1}u(t)
=
\int_0^t
\Big(
s^{n-1}\Big((1+2k)\hat V(s)^{2k}-\frac{1}{m+\omega}\Big)v(s)
-s^{n-1}b(s)
\Big)\,ds,
\]
which yields
\begin{eqnarray}\label{u-t-small-0}
|t|^{n-1}
|u(t)|
\leq
\Big(
\frac{C}{n-\alpha}\abs{t}^{n-\alpha}
\norm{\abs{t}^\alpha v}\sb{L^\infty([-1,1])}
+\frac{1}{n}\abs{t}^n\norm{b}\sb{L^\infty([-1,1])}
\Big),
\end{eqnarray}
$t\in[-1,1]$,
with $C$ dependent on $k$ and $\hat V$ only;
hence,
\begin{eqnarray}\label{u-t-small}
\abs{u(t)}
\le
\Big(
\frac{C}{n-\alpha}\abs{t}^{1-\alpha}
\norm{\abs{t}^\alpha v}\sb{L^\infty([-1,1])}
+\frac{1}{n}\abs{t}\norm{b}\sb{L^\infty([-1,1])}
\Big),
\quad
t\in[-1,1].
\end{eqnarray}
Similarly, from the second equation in \eqref{a-v-b-s}
we deduce that
\begin{eqnarray}\label{wdt}
 \||t|^\alpha\partial_t v\|_{L^\infty([-1,1])}
\le
|m+\omega|\||t|^\alpha u\|_{L^\infty([-1,1])}
+\||t|^\alpha a\|_{L^\infty([-1,1])}=:C\sb\ast,
\end{eqnarray}
so that one has $\abs{v'(t)}\le C\sb\ast \abs{t}^{-\alpha}$,
therefore
\begin{eqnarray}\label{wdt-2}
\abs{v(t)}\le
\abs{v(1)}
+\abs{v(-1)}
+C\sb\ast\abs{t}^{-\alpha+1}/\abs{\alpha-1},
\qquad
t\in[-1,1]\setminus\{0\}
\end{eqnarray}
for $\alpha\in\R\sb{+}\setminus(\frac 1 2,\frac 3 2)$ to have a uniform bound. For $\alpha\in(1/2,3/2)$,
we substitute $\alpha$ in \eqref{wdt} with $\alpha=3/2$,
again arriving at \eqref{wdt-2}.

To sum-up,  given the estimates \eqref{assume-such-alpha}
on $u$ and $v$
with
$\alpha\in\{1/2\}\cup[3/2,+\infty)$,
the estimates
\eqref{u-t-small} and \eqref{wdt-2}
yield
\eqref{assume-such-alpha}
with $\max(\alpha-1,0)$ in place of $\alpha$;
while given the estimates \eqref{assume-such-alpha}
with $\alpha\in(1/2,3/2)$,
we arrive at \eqref{assume-such-alpha}
with $1/2$ in place of $\alpha$.
It follows that \eqref{assume-such-alpha}
could be improved up to $\alpha=0$
in a finite number of steps.
Having improved \eqref{assume-such-alpha}
to $\alpha=0$,
we use
\eqref{u-t-small-0}, \eqref{u-t-small}
one more time, now with $\alpha=0$,
obtaining the bound
\begin{eqnarray}\label{u-t-small-last}
\abs{u(t)/t}
\le
\Big(
\frac{C}{n}
\norm{v}\sb{L^\infty([-1,1])}
+\frac{1}{n}\norm{b}\sb{L^\infty([-1,1])}
\Big),
\qquad
t\in[-1,1]\setminus\{0\}.
\end{eqnarray}
Using the resulting bounds on
$\norm{v}\sb{L^\infty([-1,1])}$
and
$\norm{u/t}\sb{L^\infty([-1,1])}$
in the system \eqref{a-v-b-s}
yields the desired bounds on
$\norm{v'}\sb{L^\infty([-1,1])}$
and on
$\norm{u'}\sb{L^\infty([-1,1])}$.
The continuity of the mapping \eqref{a-bounded} is proved.
\end{proof}

\begin{remark}
We note that
$(v,u)\in X^1\sb{e,o}\subset C(\R,\C^2)$;
by \eqref{u-t-small-last},
this implies that $u(0)=0$.
\end{remark}

The assumption
$(\tilde V,\tilde U)\in X^1\sb{e,o}\subset X\sb{e,o}$
leads to
$\big(G_1(\tilde V,\tilde U),G_2(\tilde V,\tilde U)\big)
\in X\sb{e,o}$,
with $G_1$, $G_2$ defined in \eqref{def-g1} and \eqref{def-g2}.
Due to invertibility of
$A(\epsilon):\,X^1\sb{e,o}\to X\sb{e,o}$
(Lemma~\ref{lemma-a-invertible}),
the relation \eqref{w-a-a-3} leads to
\begin{eqnarray}\label{Xi}
\tilde W
=A(\epsilon)\sp{-1}G(\epsilon, \tilde W),
\qquad
\tilde W=\tilde W(t,\epsilon).
\end{eqnarray}

\begin{remark}\label{remark-mu}
The continuity of $f$
is not enough to conclude that the map
\[
\mu:\;
X\sb{e,o}\to X^1\sb{e,o}\subset X\sb{e,o},
\qquad
\mu:\;
\tilde W\mapsto A(\epsilon)^{-1}G(\epsilon,\tilde W)
\]
is a contraction,
so we can not apply the contraction mapping principle
to claim a unique fixed point of $\mu$;
we will retreat to the Schauder fixed point theorem instead,
proving the existence of a fixed point
but missing its uniqueness.
In the case $f\in C^1$,
indeed the mapping $\mu$ can be shown to be a contraction
on a particular subspace
(see Lemma~\ref{lemma-contraction} below);
this will allow us to prove uniqueness
of a fixed point.
\end{remark}

To be able to consider
non-integer values of $k>0$
(in particular, we are going to treat the critical cases,
when $k=2/n$),
we need the following result.

\begin{lemma}\label{lemma-a-a}
For any $k>0$,
one has:
\begin{eqnarray}\label{a-a-1}
\qquad
\Abs{\abs{a+b}^k-\abs{a}^k}
\le
3^k
\left(
\abs{a}^{k-\min(1,k)}
+
\abs{b}^{k-\min(1,k)}
\right)
\abs{b}^{\min(1,k)},
\qquad
a,\,b\in\R;
\end{eqnarray}
\begin{eqnarray}\label{a-a-2}
\Abs{\abs{a+b}^k-\abs{a}^k-k\abs{a}^{k-1}b\sgn a}
\le
3^k
\left(
\abs{a}^{k-\min(2,k)}+\abs{b}^{k-\min(2,k)}
\right)
\abs{b}^{\min(2,k)},
\nonumber
\\
a,\,b\in\R.
\end{eqnarray}
\end{lemma}

\begin{proof}
Since the inequalities \eqref{a-a-1} and \eqref{a-a-2}
are homogeneous of degree $k$ in $a$ and $b$,
it is enough to give a proof for $a=1$, $b\in\R$.

If $\abs{b}\ge 1/2$,
then
$\Abs{\abs{1+b}^k-1}
\le
\max(\abs{1+b}^k,1)
\le 3^k\abs{b}^k$.
If $\abs{b}<1/2$,
then, by the mean value theorem,
\begin{eqnarray}\label{a-a}
\Abs{\abs{1+b}^k-1}
\le
\max\sb{c\in[1/2,3/2]}
k\abs{c}^{k-1}\abs{b}.
\end{eqnarray}
If $k\ge 1$, the right-hand side
is bounded by
$k(3/2)^{k-1}\abs{b}
\le 3^k\abs{b}$
(since $k(3/2)^{k-1}<3^k$, $\forall k\in\R$).
If $k\in(0,1)$, the right-hand side of \eqref{a-a}
is bounded by
$
k2^{1-k}\abs{b}
=k\abs{2b}^{1-k}\abs{b}^{k}
\le k\abs{b}^{k}
\le 3^k\abs{b}^{k}$.
This completes the proof of \eqref{a-a-1}.

Now let us prove \eqref{a-a-2};
again, we only need to consider the case $a=1$.
For $b\ge 1/2$, one has
\[
\abs{\abs{1+b}^k-1-k b}
\le
\max((3b)^k,1+k b)
\le \max(3^k,\,2^k+2^{k-1}k)b^k
\le
3^k(b^{k-\min(2,k)}+b^k).
\]
In the last inequality,
we took into account that,
with $b\ge 1/2$,
\[
1
\le
2^{\min(2,k)}
b^{\min(2,k)}
\le
3^k b^{\min(2,k)},
\qquad
k b
\le k 2^{k-1}b^k
\le 3^k b^k.
\]
For $b\le -1/2$,
one similarly obtains
\[
\abs{\abs{1+b}^k-1-k b}
\le
\max(\abs{b}^k+k\abs{b},
\,1)
\le\max(1+2^{k-1}k,\,2^k)\abs{b}^k
\le
3^k\abs{b}^k,
\]
since $1+2^{k-1}k<3^k$ for $k>0$.

Finally, for $\abs{b}<1/2$,
by the mean value theorem,
\begin{eqnarray}\label{a-a-prime}
\Abs{\abs{1+b}^k-1-k b}
\le
\max\sb{c\in[1/2,3/2]}
\frac{k\abs{k-1}}{2}\abs{c}^{k-2}\abs{b}^2.
\end{eqnarray}
If $k\ge 2$, the right-hand side
is bounded by
\[
\frac 1 2 k(k-1)(3/2)^{k-2}\abs{b}^2
\le 3^k\abs{b}^2,
\]
since
$\frac 1 2 k\abs{k-1}(3/2)^{k-2}<3^k$,
$\forall k>0$.
If $k\in(0,2)$, the right-hand side of \eqref{a-a-prime}
is bounded by
\[
k\abs{k-1}2^{2-k}\abs{b}^2
=k\abs{k-1}\abs{2b}^{2-k}\abs{b}^{k}
\le k\abs{k-1}\abs{b}^{k}
\le 3^k\abs{b}^{k}.
\qedhere
\]
\end{proof}

Recall that $\Lambda\sb k<\infty$
was defined in \eqref{def-Lambda}.

\begin{lemma}\label{lemma-g1-g2-bounds}
There is $C<\infty$ such that
for any numbers
$\hat V,\,\hat U,\,\tilde V,\,\tilde U\in[-\Lambda\sb k,\Lambda\sb k]$,
$V=\hat V+\tilde V$, and $U=\hat U+\tilde U$,
one has
\begin{eqnarray}
&&
\abs{G_1(\epsilon,\tilde V,\tilde U)}
\le
C
\Big(
h(\epsilon)
(\abs{V}+\abs{U})^{1+2k}
+
\hat V^{1+2k-\min(2,1+2k)}
\abs{\tilde V}^{\min(2,1+2k)}
\nonumber
\\
&&
\qquad\qquad\qquad
\qquad\qquad\qquad
\qquad\qquad\qquad
\qquad\qquad\qquad
+
\abs{\tilde V}^{1+2k}
+
\epsilon^2\hat V
\Big),
\nonumber
\\
&&
\abs{G_2(\epsilon,\tilde V,\tilde U)}
\le
C
\left(
\epsilon^2
\abs{V^2-\epsilon^2 U^2}^k\abs{U}
+
\epsilon^2\abs{\hat U}
\right),
\nonumber
\end{eqnarray}
for all $\epsilon\in(0,\epsilon_0)$,
with $\epsilon_0>0$
from Theorem~\ref{theorem-solitary-waves} 
\end{lemma}

\begin{proof}
Although most terms in the definition of $G$
(cf. \eqref{def-g}) are small,
we have to be careful when we consider
the general case $k>0$ when $f'(\tau)$
may not be uniformly bounded near $\tau=0$.
To bound $G_1$ (cf. \eqref{def-g1}),
we proceed as follows:
\begin{eqnarray}\label{g-g-1-0}
\abs{G_1(\epsilon,\tilde V,\tilde U)}
&\le&
\Abs{
\epsilon^{-2}f\big(\epsilon^{2/k}(V\sp 2- \epsilon\sp 2 U^2)\big)V
-
\hat V\sp{2k}\hat V
-(1\!+\!2k)\hat V\sp{2k}\tilde{V}}
+
\Abs{
\frac{\hat{V}}{m+\omega}-\frac{\hat{V}}{2m}
}
\nonumber
\\
&\le&
\Abs{
\epsilon^{-2}f\big(\epsilon^{2/k}(V\sp 2- \epsilon\sp 2 U^2)\big)
-\abs{V\sp 2- \epsilon\sp 2 U^2}^{k}}\abs{V}
\nonumber
\\
&&
+
\Abs{\abs{V\sp 2- \epsilon\sp 2 U^2}^{k}-\abs{V}^{2k}}\abs{V}
\nonumber
\\
&&
+
\Abs{
\abs{V}^{2k}V-\hat V\sp{2k}\hat V
-(1+2k)\hat V\sp{2k}\tilde{V}
}
\,+\,
\Big(
\frac{1}{m+\omega}-\frac{1}{2m}
\Big)
\hat V.
\end{eqnarray}
We use \eqref{ass-fo}
to estimate the first term
in the right-hand side
by $h(\epsilon)\abs{V^2-\epsilon^2 U^2}^k\abs{V}$.
Other terms are dealt with
by Lemma~\ref{lemma-a-a}:
we apply
\eqref{a-a-1} to the second term
and \eqref{a-a-2} (with $1+2k$ instead of $k$)
to the third term,
getting
\begin{eqnarray}
&&
\abs{G_1(\epsilon,\tilde V,\tilde U)}
\le
C h(\epsilon)
\abs{V^2-\epsilon^2U^2}^k|V|
+
3^k
\left(
\abs{V}^{2k-2\min(1,k)}
\abs{\epsilon U}^{2\min(1,k)}
+
\abs{\epsilon U}^{2k}
\right)
\!|V|
\nonumber
\\
&&
\qquad\qquad
+3^{1+2k}
\left(
\hat V^{1+2k-\min(2,1+2k)}
\abs{\tilde V}^{\min(2,1+2k)}
+
\abs{\tilde V}^{1+2k}
\right)
\,+\,
\Big(\frac{1}{m+\omega}-\frac{1}{2m}\Big)|\hat V|
,
\nonumber
\end{eqnarray}
which yields the desired bound on $G_1$.
We took into account the definition of $h(\epsilon)$
in \eqref{def-h}.

The estimate on $\abs{G_2}$
immediately follows from \eqref{def-g2}
and \eqref{ass-fo}.
\end{proof}

To apply the fixed point theorem,
we will use the exponential weights,
introducing compactness into \eqref{Xi}.
We fix
\begin{eqnarray}\label{gamma-small}
\gamma\in(0,\gamma_0),
\qquad
\mbox{where}
\quad
\gamma\sb 0
:=
\frac{1}{1+2k}
\inf\sb{\epsilon\in[0,\epsilon_0]}
\frac{1}{
1+\norm{A(\epsilon)^{-1}}\sb{X\sb{e,o}\to X^1\sb{e,o}}
};
\end{eqnarray}
we note that,
by Lemma~\ref{lemma-a-invertible},
one has $\gamma\sb 0>0$.
Due to the exponential decay of $\hat V(t)$, $\hat U(t)$
(see Lemma~\ref{lemma-bl} in Appendix~\ref{sect-nls-smooth}),
since $\gamma<1/(2k+1)<1$,
there are the following inclusions:
\begin{eqnarray}\label{such-w-2}
e^{(1+2k)\gamma\langle t\rangle}\hat{U}
\in X^1,
\qquad
e^{(1+2k)\gamma\langle t\rangle}\hat{V}
\in X^1,
\end{eqnarray}
with $X^1$ from \eqref{def-space-x1}.
We define
\begin{eqnarray}\label{a-gamma}
A\sb\gamma(\epsilon)
:=
e^{(1+2k)\gamma\langle t\rangle}\circ A(\epsilon)
\circ e^{-(1+2k)\gamma\langle t\rangle}
=A(\epsilon)
-
(1+2k)
\gamma\frac{t}{\langle t\rangle}
\begin{bmatrix}0&1\\1&0\end{bmatrix}.
\end{eqnarray}
Due to Lemma~\ref{lemma-a-invertible}
and the choice of $\gamma\sb 0$ in \eqref{gamma-small},
for any $\epsilon\in[0,\epsilon_0]$
the operator \eqref{a-gamma}
is closed and invertible,
so that the mapping
\begin{eqnarray}\label{ttm}
&&
A\sb\gamma(\epsilon)^{-1}
=e^{(1+2k)\gamma\langle t\rangle}\circ A(\epsilon)^{-1}\circ e^{-(1+2k)\gamma\langle t\rangle},
\\
&&
A\sb\gamma(\epsilon)^{-1}:\;
X\sb{e,o}\to X^1\sb{e,o}
:=H^1\sb{e,o}(\R,\langle t\rangle^{n-1}\mathrm{d}t;\,\C^2)
\nonumber
\end{eqnarray}
is bounded uniformly in $\epsilon\in[0,\epsilon_0]$.
We multiply the fixed point problem \eqref{Xi}
by $e^{\gamma\langle t\rangle}$,
rewriting it in the form
\begin{eqnarray}\label{itf}
e^{\gamma\langle t\rangle}
\tilde W =
e^{-2k\gamma\langle t\rangle}
A\sb\gamma(\epsilon)\sp{-1}
e^{(1+2k)\gamma\langle t\rangle}
G\big(\epsilon, e\sp{-\gamma\langle t\rangle}e^{\gamma\langle t\rangle}\tilde W\big).
\end{eqnarray}

\begin{lemma}\label{lemma-small-g-new}
There is $C<\infty$ such that
for any $\epsilon\in(0,\epsilon_0)$
and any
$\begin{bmatrix}\tilde V\\\tilde U\end{bmatrix}\in X$
which satisfies
\begin{eqnarray}\label{tilde-smaller-than-hat-new}
\Norm{e^{\gamma\langle t\rangle}
\begin{bmatrix}\tilde V\\\tilde U\end{bmatrix}
}\sb{X}
\le
\Norm{e^{\gamma\langle t\rangle}
\begin{bmatrix}\hat V(t)\\ \hat U(t)\end{bmatrix}
}\sb{X},
\end{eqnarray}
one has
\begin{eqnarray}\label{g-small-new}
\norm{
e^{(1+2k)\gamma\langle t\rangle}
G(\epsilon, \tilde V,\tilde U)}\sb{X}
\leq
C\left(h(\epsilon)
+\Norm{e^{\gamma\langle t\rangle}
\begin{bmatrix}\tilde V\\\tilde U\end{bmatrix}
}\sb{X}\sp{1+\min(1,2k)}
\right),
\end{eqnarray}
with $h(\epsilon)$ from \eqref{def-h}.
\end{lemma}

\begin{proof}
We use the pointwise estimates on $G_1$, $G_2$ from
Lemma~\ref{lemma-g1-g2-bounds}.
There, the first term in the right-hand side of
the bound on $G_1$
has a factor $h(\epsilon)$.
Multiplying this term 
by $e^{(1+2k)\gamma\langle t\rangle}$
and using \eqref{such-w-2} and \eqref{tilde-smaller-than-hat-new},
and also the fact that the space
$X$ defined in \eqref{def-space-x}
is closed under multiplication,
we bound the resulting $X$-norm
by $C h(\epsilon)$,
with some $C<\infty$.

The terms
\[
\hat V^{1+2k-\min(2,1+2k)}\abs{\tilde{V}}^{\min(2,1+2k)}
+\abs{\tilde V}^{1+2k}
\]
in the right-hand side of
the bound on $G_1$
in Lemma~\ref{lemma-g1-g2-bounds},
having no $\epsilon$-factor,
are of order higher than one in $\tilde V$,
benefiting us when $\abs{\tilde V}$ is small.
Multiplying them by the factor $e^{(1+2k)\gamma\langle t\rangle}$,
which is absorbed
by the terms which are homogeneous of order
$(1+2k)$ in
$\hat V$ and $\tilde V$,
we bound the $X$-norm of the result
by $C\norm{e^{\gamma\langle t\rangle}\tilde V}_{X}^{1+\min(1,2k)}$.
We note that
$\norm{e^{\gamma\langle t\rangle}\hat V}_{X}^{1+\min(1,2k)}$
and
$\norm{e^{\gamma\langle t\rangle}\hat U}_{X}^{1+\min(1,2k)}$
are finite due to $\gamma<1$ (cf. \eqref{gamma-small})
and
due to the exponential decay of $\hat V$
and $\hat U$
which follows from \eqref{Vhatdef}
and Lemma~\ref{lemma-bl} (see Appendix~\ref{sect-nls-smooth}).

For the last term
in the right-hand side
of the bound on $G_1$
from Lemma~\ref{lemma-g1-g2-bounds}
multiplied by $e^{(1+2k)\gamma\langle t\rangle}$,
its $X$-norm is bounded by $C\epsilon^2$
with the aid of \eqref{such-w-2}.
We conclude that
there is a constant
$C<\infty$
such that there is the desired bound
\[
\norm{
e^{(1+2k)\gamma\langle t\rangle}
G_1(\epsilon,\tilde V,\tilde U)}\sb{X}
\le
C\left(h(\epsilon)
  +\norm{e^{\gamma\langle t\rangle}\tilde{V}}\sb{X}^{1+\min(1,2k)}
\right),
\qquad
\forall
\epsilon\in(0,\epsilon_0).
\]

We now consider $G_2$.
Due to the factor
$\epsilon^2$ in the right-hand side
of the bound on $G_2$
in Lemma~\ref{lemma-g1-g2-bounds}
and due to the exponential decay of $\hat V$, $\hat U$
(together with the bound \eqref{such-w-2}),
as well as due to the
assumption \eqref{tilde-smaller-than-hat-new}
about the exponential decay of $\tilde V$ and $\tilde U$,
one has
\[
\norm{
e^{(1+2k)\gamma\langle t\rangle}G_2
(\epsilon,\tilde V,\tilde U)
}\sb{X}
\le
C\epsilon^2,
\qquad
\forall
\epsilon\in(0,\epsilon_0).
\]
Lemma~\ref{lemma-small-g-new}
is proved.
\end{proof}

Let us now complete
the proof of Theorem~\ref{theorem-solitary-waves}.
We consider the mapping
\begin{eqnarray}\label{mapping-mu}
&&
\qquad\qquad\qquad\qquad
\mu_\gamma(\epsilon,\cdot):\;X\sb{e,o}\to X\sb{e,o}\to X\sb{e,o}^1
\subset X\sb{e,o},
\\[1ex]
\nonumber
&&
\mu_\gamma(\epsilon,\cdot):\;
Z
\mapsto
e^{(1+2k)\gamma\langle t\rangle}G(\epsilon,e^{-\gamma\langle t\rangle}Z)
\mapsto
e^{-2k\gamma\langle t\rangle}
A\sb\gamma(\epsilon)^{-1}
e^{(1+2k)\gamma\langle t\rangle}G(\epsilon,e^{-\gamma\langle t\rangle}Z).
\end{eqnarray}
Note that $\tilde W$ is a solution to \eqref{w-a-a-3}
if and only if $Z=e^{\gamma\langle t\rangle}\tilde W$
is a fixed point of this map.

\begin{lemma}\label{lemma-mu-into}
One can take $\epsilon_0>0$ smaller if necessary
so that there is
$a_0>0$
such that
\begin{eqnarray}\label{mu-ball-x-x1}
\mu_\gamma\left(\epsilon,\,\overline{\mathbb{B}\sb{\rho}(X\sb{e,o})}\right)
\subset\overline{\mathbb{B}\sb{\rho}(X^1\sb{e,o})},
\quad
\rho=a_0 h(\epsilon),
\qquad
\forall\epsilon\in(0,\epsilon_0),
\end{eqnarray}
with $h(\epsilon)$ from \eqref{def-h}.
\end{lemma}


\begin{proof}
If $Z$ belongs to a closed ball
$\overline{\mathbb{B}\sb{\rho}(X\sb{e,o})}
=\{\xi\in X\sb{e,o}\sothat\norm{\xi}\sb{X}\le\rho\}$,
with
\[
\rho
\le
\Norm{e^{\gamma\langle t\rangle}
\begin{bmatrix}\hat V(t)\\ \hat U(t)\end{bmatrix}}\sb{X},
\]
then
Lemma~\ref{lemma-small-g-new}
applies to
$\tilde W=e^{-\gamma\langle t\rangle}Z$,
giving us
\begin{eqnarray}\label{e-g-small}
\norm{
e^{(1+2k)\gamma\langle t\rangle}G(\epsilon,e^{-\gamma\langle t\rangle}Z)}\sb{X}
\leq C
\left\{
h(\epsilon)+\norm{Z}\sb{X}\sp{1+\min(1,2k)}
\right\}.
\end{eqnarray}
Therefore, to find the sufficient condition
for \eqref{mu-ball-x-x1} to be satisfied,
we use the definition of $\mu\sb\gamma$ from \eqref{mapping-mu}
and apply the estimate \eqref{e-g-small},
arriving at the requirement
\begin{eqnarray}\label{into-itself}
\norm{e^{-2k\gamma\langle t\rangle}\circ A\sb\gamma(\epsilon)^{-1}}\sb{X\sb{e,o}\to X\sb{e,o}^1}
C
\left\{h(\epsilon)+\rho^{1+\min(1,2k)}\right\}
\le\rho.
\end{eqnarray}
Noting the continuity of the mapping \eqref{ttm},
the first factor in the left-hand side is bounded;
thus, one can satisfy \eqref{into-itself}
by taking $\rho=O(h(\epsilon))$.
This finishes the proof.
\end{proof}

Since it is not clear that the mapping
$\mu_\gamma(\epsilon,\cdot):\;X\sb{e,o}\to X^1\sb{e,o}\subset X\sb{e,o}$
defined in \eqref{mapping-mu}
is a contraction
without assuming that $f$ is sufficiently regular
we can not apply the Banach fixed point theorem
to \eqref{mapping-mu}.
Instead, we use the Schauder fixed point theorem
(see e.g. \cite[Corollary 11.2]{MR1814364}):

\begin{verse}
\noindent
{\it
Let $Q$ be a closed, convex, bounded subset of a
Banach space $\mathscr{X}$
and $\mu:\;Q\to Q$ a continuous compact map;
then $\mu$ has a fixed point in $Q$.
}
\end{verse}

Clearly, the mapping
$\mu_\gamma(\epsilon,\cdot):\;X\sb{e,o}\to X\sb{e,o}^1$ is continuous;
note that, in particular,
\[
(V,U)\mapsto
\epsilon^{-2} f\big(\epsilon^{2/k}(V^2-\epsilon^2 U^2)\big)V
\]
is continuous in the norm of the space $X$
since the map $(V,U)\mapsto
\epsilon^{-2} f\big(\epsilon^{2/k}(V^2-\epsilon^2 U^2)\big)$
is continuous as a map from $L^\infty(\R,\C^2)$ to $L^\infty(\R)$.
Then the mapping
\[
e\sp{-2k\gamma\langle t\rangle}
\circ
A\sb\gamma(\epsilon)\sp{-1}:
\,
X\sb{e,o}\to X\sb{e,o}^1\to X\sb{e,o},
\]
is compact,
since the multiplication
by the decaying exponential weight
is a compact map from $X\sb{e,o}^1$ to $X\sb{e,o}$.
Therefore, so is the mapping $\mu_\gamma(\epsilon,\cdot)$
when considered as a map from $X\sb{e,o}$ into itself.
By Lemma~\ref{lemma-mu-into},
the Schauder fixed point theorem
gives a fixed point of the map $\mu_\gamma(\epsilon,\cdot)$
which belongs to a closed ball
$\overline{\mathbb{B}\sb\rho(X^1\sb{e,o})}$
of radius $\rho=a_0 h(\epsilon)$,
with $a_0>0$ which does not depend on
$\epsilon\in(0,\epsilon_0)$.
It follows that $\tilde W=e^{-k\gamma\langle t\rangle}Z$ satisfies
\begin{eqnarray}\label{v-u-tilde-small}
\norm{e^{k\gamma\langle t\rangle}\tilde W}\sb{X^1}
=\norm{Z}\sb{X^1}\le\rho\le a_0 h(\epsilon),
\qquad
\forall
\epsilon\in(0,\epsilon_0).
\end{eqnarray}
This yields \eqref{v-u-tilde-small-0}.

\begin{remark}
The map $\tilde W(\epsilon)$
is not a sufficiently well-defined function
to make it continuous in $\epsilon$
since the solution
provided by the Schauder fixed point
theorem is not necessarily unique,
due to the absence of the contraction.
The uniqueness of the mapping
$\epsilon\mapsto\tilde W(\epsilon)$,
under stronger assumptions on $f$,
will be addressed in Section~\ref{sect-diff}.
\end{remark}

We note that
\[
\norm{\tilde W}\sb{L^\infty}
\le
\norm{e^{k\gamma\langle t\rangle}\tilde W}\sb{L^\infty}
\le
\norm{e^{k\gamma\langle t\rangle}\tilde W}\sb{X^1}
\le a_0 h(\epsilon),
\qquad
\forall\epsilon\in(0,\epsilon_0);
\]
thus, we can impose the condition that
$\epsilon_0>0$ is small enough so that
\[
\abs{\tilde V(t,\epsilon)}
+m\abs{\tilde U(t,\epsilon)}
<\Lambda\sb k,
\qquad
\forall t\in\R,
\qquad
\forall\epsilon\in(0,\epsilon_0),
\]
to satisfy our assumption
\eqref{UV0-0}.

\medskip

Finally, let us prove that $V,\,U\in C^1(\R)$.
Due to the continuity of $\hat V$ and $\hat U$
(which follows from Lemma~\ref{lemma-bl}
and from \eqref{Vhatdef})
and of $\tilde V$ and $\tilde U$
(which follows from
applying Lemma~\ref{lemma-a-invertible}
to \eqref{w-a-a-3}),
we know that $V$ and $U$
are continuous on the whole real axis.

\begin{lemma}\label{lemma-v-u-c1}
Fix $\epsilon\in(0,\epsilon_0)$.
If $f\in C(\R)$
and if $V,\,U\in C(\R)$, with $V$ even and $U$ odd,
are solutions to \eqref{zero-is-phi1},
then $V,\,U\in C^1(\R)$
and $H(t):=U(t)/t$, $t\ne 0$
could be extended to a continuous function on $\R$.

Moreover, if there is $C<\infty$ such that
\[
\abs{V(t)}+\abs{U(t)}
\le C,
\qquad
\forall t\in\R,
\]
then there is $C'<\infty$ such that
\[
\abs{\p\sb t V(t)}+\abs{\p\sb t U(t)}
\le
C',
\qquad
\forall t\in\R.
\]
\end{lemma}

\begin{proof}
The second equation in \eqref{zero-is-phi1}
immediately gives $V\in C^1(\R)$.
To prove that one also has $U\in C^1(\R)$,
we write the first equation in \eqref{zero-is-phi1}
as
\begin{eqnarray}\label{up-u-b}
U'+(n-1)\frac{U}{t}=B(t),
\qquad t\in\R,
\end{eqnarray}
with $B\in C(\R)$
given by
\begin{eqnarray}\label{def-b-t}
B(t)=\frac{
f\big(\epsilon^{2/k}\big(V(t)^2-\epsilon^2 U(t)^2\big)\big)}{\epsilon^2}
V(t)
-\frac{1}{m+\omega}V(t).
\end{eqnarray}
It is enough to prove that
$H(t)=U(t)/t\in C(\R\setminus\{0\})$ could be extended
to a continuous function on $\R$
(then the same is true for $U'$).
Thus, we need to show that
$H(t)$ has a finite limit as $t\to 0$.
From \eqref{up-u-b} we arrive at
\[
\p\sb t(U(t)t^{n-1})=B(t)t^{n-1},\qquad t\in\R,
\]
hence, one has
\begin{eqnarray}\label{def-h-t}
H(t)=\frac{U(t)}{t}=\frac{\int_0^t B(\tau)\tau^{n-1}\,d\tau}{t^n},
\qquad
t>0,
\end{eqnarray}
which has a well-defined limit at the origin:
\[
\lim\sb{t\to 0}
H(t)
=
\lim\sb{t\to 0}
\frac{\int_0^t B(\tau)\tau^{n-1}\,d\tau}{t^n}
=
\lim\sb{t\to 0}
\frac{B(t)}{n}=\frac{B(0)}{n}.
\]

Let us show the uniform boundedness of the derivatives of $V$ and $U$.
From the system \eqref{zero-is-phi1}, due to bounds \eqref{no-more},
we conclude that
\[
\abs{\p\sb t V(t)}
\le C,
\qquad
\forall t\in\R,
\qquad
\forall\epsilon\in(0,\epsilon_0),
\]
with some $C<\infty$.
Then, since
$B(t)$ in \eqref{def-b-t}
satisfies
\[
\abs{B(t)}
=
\Abs{\frac{f}{\epsilon^2}V-\frac{1}{m+\omega}U}
\le
C,
\qquad
\forall t\in\R,
\qquad
\forall\epsilon\in(0,\epsilon_0),
\]
with some $C<\infty$,
we conclude from \eqref{def-h-t} that
$
\abs{H(t)}
\le
\norm{B}\sb{L^\infty}
$
and then from \eqref{zero-is-phi1}
that
$
\abs{\p\sb t U(t)}
\le
2\norm{B}\sb{L^\infty}$,
for all $t\in\R$.
\end{proof}

The proof of Theorem~\ref{theorem-solitary-waves}~\itref{theorem-solitary-waves-i}
is finished.

\section{Positivity of
$\bar\phi\phi$
and improved estimates}
\label{sect-shooting}

\subsection{Positivity of
$\bar\phi\phi$
in the nonrelativistic limit
via the shooting argument}

\label{sect-shooting-1}

To be able to consider
the nonlinearity $f(\tau)=\abs{\tau}^k+\dots$
which is not differentiable at $\tau=0$
unless $k\ge 1$,
we will show that the quantity
$\phi\sp\ast\beta\phi$,
which is the argument of $f(\cdot)$ in \eqref{nld-stationary},
remains positive if $\omega\lesssim m$.
This will allow us to treat
the nonlinear Dirac equation
with fractional power nonlinearity
using the Taylor-style estimates on the remainders
instead of weaker estimates from Lemma~\ref{lemma-a-a}.

So we proceed to the proof of
Theorem~\ref{theorem-solitary-waves}~\itref{theorem-solitary-waves-ii},
showing that $U$ is pointwise dominated by $V$.

\begin{proposition}\label{prop-u-le-v}
There is $\epsilon_1\in(0,\epsilon_0)$
such that for all $\epsilon\in(0,\epsilon_1)$
one has
\[
\epsilon_1\abs{U(t,\epsilon)}
\le
\frac{1}{2}\abs{V(t,\epsilon)},
\qquad
\forall t\in\R,
\qquad
\forall\epsilon\in(0,\epsilon_1).
\]
\end{proposition}

Above, $\epsilon_0>0$ is from
Theorem~\ref{theorem-solitary-waves}~\itref{theorem-solitary-waves-i}.

\begin{proof}
We rewrite \eqref{zero-is-phi1} as follows:
\begin{eqnarray}\label{zero-is-phi2}
\qquad
\begin{cases}
\p\sb t U
=-\frac{1}{m+\omega} V
-\frac{n-1}{t} U
+|V|\sp{2k} V
+\big(
\epsilon^{-2}f\big(\epsilon^{2/k}(V^2-\epsilon^2 U^2)\big)-|V|\sp{2k}
\big)V,
\\[2ex]
\p\sb t V
=-(m+\omega)U
+\epsilon\sp 2 |V|\sp{2k} U
+\big(
f\big(\epsilon^{2/k}(V^2-\epsilon^2 U^2)\big)-\epsilon^2|V|\sp{2k}
\big)U.
\end{cases}
\end{eqnarray}
For any $\delta>0$
and any $\nu\in(0,\nu_0)$,
$\nu_0=\min(\delta/8,m\delta/8)$,
define the following closed sets
(see Figure~\ref{fig-region-omega}):
\begin{eqnarray}
&&
\calK\sp{+}\sb{\delta,\nu}
=\left\{
(V,U)\in\overline{\mathbb{B}\sb{\delta}^2}\subset\R^2
\sothat
U\ge
\max\Big(
0,\,\frac{V+\nu}{m},\,\frac{2V}{m}
\Big)
\right\},
\nonumber
\\[1ex]
&&
\calK\sp{0}\sb\delta
\;=\;\left\{
(V,U)\in\overline{\mathbb{B}\sb{\delta}^2}\subset\R^2
\sothat
V\ge 0,
\ \ \frac{V}{4m}\le U\le \frac{2V}{m}
\right\},
\nonumber
\\[1ex]
&&
\calK\sp{-}\sb{\delta,\nu}=
\left\{
(V,U)\in\overline{\mathbb{B}\sb\delta^2}\subset\R^2
\sothat
V\ge 0,\,\,U\le\min\left(\frac{V-\nu}{2m},
\,\,\frac{V}{4m}
\right)
\right\}.
\nonumber
\end{eqnarray}

\begin{figure}[htbp]
\begin{center}
\setlength{\unitlength}{1pt}
\begin{picture}(0,90)(0,-40)
\font\gnuplot=cmr10 at 10pt
\gnuplot
\put(32,48){$U=2V/m$}
\put(62,20){$U=V/(4m)$}
\put(16,18){$\calK\sp{0}\sb\delta$}
\put(-52,-9){$-\delta$}
\put(42,-8){$\delta$}
\put(2,42){$\delta$}
\put(-15,-42){$-\delta$}
\put(-22,-6){$-\nu$}
\put(58,-2){$V$}
\put(-9,53){$U$}

\put(6,15){$\blacksquare$}
\put(15,2){$\blacksquare$}

\put(-50,0){\vector(1,0){105}}
\put(0,-50){\vector(0,1){105}}
\put(-20,20){$\calK\sp{+}\sb{\delta,\nu}$}
\put(13,-20){$\calK\sp{-}\sb{\delta,\nu}$}
\put(10,20){\circle*{2}}
\put(-10,0){\circle*{2}}
\put(10,0){\circle*{2}}
\put(0,-5){\circle*{2}}
\put(0,40){\circle*{2}}
\put(40,0){\circle*{2}}
\put(-40,0){\circle*{2}}
\put(0,-40){\circle*{2}}
\put(20,5){\circle*{2}}
\put(39,9.75){\circle*{2}}
\put(18,36){\circle*{2}}

\qbezier(0,0)(10,20)(30,60)
\qbezier(0,0)(20,5)(60,15)

\linethickness{1pt}

\qbezier(0,0)(10,20)(18,36)

\qbezier(40,0)(40,15)(32,24)
\qbezier(32,24)(28,29)(24,32)
\qbezier(0,40)(15,40)(24,32)
\qbezier(0,40)(-15,40)(-24,32)

\qbezier(-40,0)(-10,0)(-10,0)
\qbezier(-10,0)(10,20)(10,20)
\qbezier(0,0)(10,20)(10,20)

\qbezier(-40,0)(-40,15)(-32,24)
\qbezier(-32,24)(-28,29)(-24,32)
\qbezier(0,40)(-15,40)(-24,32)
\qbezier(0,40)(15,40)(24,32)

\qbezier(40,0)(40,-15)(32,-24)
\qbezier(32,-24)(28,-29)(24,-32)
\qbezier(0,-40)(15,-40)(24,-32)

\qbezier(0,-40)(0,-5)(0,-5)

\qbezier(0,-5)(20,5)(20,5)
\qbezier(0,0)(20,5)(39,9.75)

\end{picture}
\caption{\footnotesize
The regions
$\calK\sp{+}\sb{\delta,\nu}$,
$\calK\sp{0}\sb\delta$,
$\calK\sp{-}\sb{\delta,\nu}$
inside $\overline{\mathbb{B}\sb\delta^2}$.
}
\label{fig-region-omega}
\end{center}
\end{figure}

The value of $\nu_0$ is chosen so that
for $\nu\in(0,\nu_0)$
the corner points
of both $\calK\sp{+}\sb{\delta,\nu}$ and $\calK\sp{-}\sb{\delta,\nu}$
inside the first quadrant,
$(\nu,2\nu/m)$
and $(2\nu,\nu/(2m))$
(marked by black squares on Figure~\ref{fig-region-omega}),
belong to $\mathbb{B}^2\sb{\delta/2}$:
\begin{eqnarray}\label{inside-half-ball}
(\nu,2\nu/m)\in\mathbb{B}^2\sb{\delta/2},
\qquad
(2\nu,\nu/(2m))\in\mathbb{B}^2\sb{\delta/2}.
\end{eqnarray}

\begin{lemma}\label{lemma-inside-omega-plus}
If $\delta>0$ is sufficiently small,
then any $C^1$-solution to \eqref{zero-is-phi2}
with $\epsilon\in(0,\epsilon_0)$
(with $\epsilon_0>0$ from Theorem~\ref{theorem-solitary-waves})
which satisfies
\[
(V(T),\,U(T))\in\calK\sp{+}\sb{\delta,\nu}
\]
at some $T\ge 2n$,
can only leave the region $\calK\sp{+}\sb{\delta,\nu}$
through
the boundary of the $\delta$-disc:
either
$(V(t),\,U(t))\in\calK\sp{+}\sb{\delta,\nu}$
for all $t\ge T$,
or else there is $T_\ast\in(T,+\infty)$ such that
$(V(t),\,U(t))\in\calK\sp{+}\sb{\delta,\nu}$
for
$T\le t\le T_\ast$,
$(V(T_\ast),\,U(T_\ast))\in\mathbb{S}^1\sb\delta$.
\end{lemma}

\begin{proof}
It suffices to check that at all pieces of
$\p\calK\sp{+}\sb{\delta,\nu}
\setminus \mathbb{S}^1\sb\delta
$
the integral curves of
\eqref{zero-is-phi2} are directed strictly inside
$\calK\sp{+}\sb{\delta,\nu}$;
that is,
at the points
$U=\max\big(0,(V+\nu)/m,2V/m\big)$,
one has
$\bm{n}\cdot(\dot V,\dot U)>0$,
with $\bm{n}$ the inner normal to $\p\calK\sp{+}\sb{\delta,\nu}$
(as long as $t\ge 2n$).

On the piece
$\{(V,0)\sothat -\delta\le V\le-\nu\}
\subset\p\calK\sp{+}\sb{\delta,\nu}$,
we compute:
\[
(0,1)\cdot(\dot V,\dot U)=\dot U
=-\frac{V}{m+\omega}+o(V)>0,
\]
as long as $\delta>0$ is sufficiently small.

On the piece
$\left\{\big(V,(V+\nu)/m\big)\sothat-\nu\le V\le\nu\right\}
\subset\p\calK\sp{+}\sb{\delta,\nu}$,
since $T\geq 2n$, one has:
\begin{eqnarray}
(-1,m)\cdot(\dot V,\dot U)
&=&
(m+\omega)U
-m\Big(\frac{V}{m+\omega}+\frac{(n-1)U}{t}\Big)
+o(\abs{U}+\abs{V})
\nonumber
\\
&\ge&
\Big(\frac{m}{2}+\omega\Big)U
-\frac{m V}{m+\omega}
+o(\abs{U}+\abs{V})
\nonumber
\\
&=&
\Big(\frac{1}{2}+\frac{\omega}{m}\Big)(V+\nu)
-\frac{m V}{m+\omega}
+o(\abs{V+\nu}+\abs{V}).
\nonumber
\end{eqnarray}
When $-\nu\le V<0$, the first two terms in the right-hand side
are positive,
dominating the last term if $\delta$ is sufficiently small.
For $0\le V\le\nu$,
due to $\omega>m/2$ (cf. \eqref{omega-large}),
the positive first term in the right-hand side
dominates both the second term and the last term
since
\[
\frac{m V}{m+\omega}\le \frac 2 3 V
\leq\frac 1 3(V+\nu).
\]

On the piece of the boundary
$\{(V,2V/m)\sothat V\ge \nu\}\cap\p\calK\sp{+}\sb{\delta,\nu}$,
we get
\[
(-2,m)\cdot(\dot V,\dot U)
=-2\dot V+m\dot U
=
2(m+\omega)U
-
m\Big(
\frac{V}{m+\omega}+\frac{n-1}{t}U
\Big)
+o(V)
\]
\[
\ge
2(m+\omega)\frac{2V}{m}
-
\Big(
V+\frac{n-1}{t}2V
\Big)
+o(V)
\ge
4V+o(V)>0.
\]
We took into account that $\omega>m/2$
and that $t\ge 2n$.
\end{proof}

\begin{lemma}\label{lemma-inside-omega-minus}
If $\delta>0$ is sufficiently small,
then any $C^1$-solution to \eqref{zero-is-phi2}
with $0<\epsilon\le \frac{m}{4}$
which satisfies
\[
(V(T),U(T))\in\calK\sp{-}\sb{\delta,\nu}
\]
at some $T\ge 2n$
can only exit the region
$\calK\sp{-}\sb{\delta,\nu}$
through the boundary of the
$\delta$-disc:
either
$(V(t),\,U(t))\in\calK\sp{-}\sb{\delta,\nu}$
for all $t\ge T$,
or else there is $T_\ast\in(T,+\infty)$ such that
$(V(t),\,U(t))\in\calK\sp{-}\sb{\delta,\nu}$
for
$T\le t\le T_\ast$,
$(V(T_\ast),\,U(T_\ast))\in\mathbb{S}^1\sb\delta$.
\end{lemma}


\begin{proof}
The proof is similar to that of
Lemma~\ref{lemma-inside-omega-minus};
we keep checking the positivity of the dot products
of the inner normals to the boundary with $(\dot V,\dot U)$.
For the pieces of the boundary
given by $V=0$,
the proof is immediate
(from \eqref{zero-is-phi2},
one can see that
$\dot V>0$,
as long as $\delta>0$ is small enough
so that the nonlinear terms
are dominated by the linear part).
On the piece given by
$U=(V-\nu)/(2m)$,
$0\le V\le 2\nu$,
\begin{eqnarray}\label{i-n-b}
(1,-2m)\cdot(\dot V,\dot U)
&=&\dot V-2m\dot U
\nonumber
\\
&=&-(m+\omega)U+\frac{2m V}{m+\omega}+\frac{2m(n-1)U}{t}
+o(\abs{V}+\abs{U}).
\end{eqnarray}
At $V=0$, $U=-\nu/(2m)$,
the linear part
of the right-hand side of \eqref{i-n-b} equals
$
\frac{m+\omega}{2m}\nu
-\frac{n-1}{t}\nu,
$
which is positive for $\nu>0$,
$\omega\in(0,m)$, $t\ge 2n$.
At the other end of the interval,
at $V=2\nu$, $U=\nu/(2m)$,
the linear part of \eqref{i-n-b} equals
$-\frac{m+\omega}{2m}\nu
+\frac{4m}{m+\omega}\nu
+\frac{n-1}{t}\nu$,
which is strictly positive
for $t\ge 2n$, $\nu>0$,
$\omega\in(m/2,m)$.
Since the linear part is strictly positive,
it dominates
the error term $o(\abs{V}+\abs{U})$
in \eqref{i-n-b}
as long as $\delta>0$ is sufficiently small.

On the piece of the boundary
of $\calK\sp{-}\sb{\delta,\nu}$
given by
$U=V/(4m)$, $2\nu\le V\le\delta$,
one has
\begin{eqnarray}
(1,-4m)\cdot(\dot V,\dot U)
=\dot V-4m\dot U
=-(m+\omega)U+\frac{4m V}{m+\omega}+\frac{4m(n-1)U}{t}
+o(\abs{V}+\abs{U})
\nonumber
\\[0.5ex]
=
\Big(
-\frac{m+\omega}{4m}
+\frac{4m}{m+\omega}+\frac{n-1}{t}
\Big)V
+o(\abs{V}).
\nonumber
\end{eqnarray}
Since $\omega\in(m/2,m)$ (cf. \eqref{omega-large}),
the linear part
in the right-hand side
is strictly positive,
dominating the nonlinear part
as long as $\delta>0$ is sufficiently small.
\end{proof}

Back to the proof of the proposition,
we choose $\delta>0$ small enough so that
both Lemma~\ref{lemma-inside-omega-plus}
and Lemma~\ref{lemma-inside-omega-minus}
are satisfied.
By \cite{MR695535},
$\hat V>0$ and $\hat U\ge 0$ are exponentially decaying,
hence we can choose
$T_1\ge 2n$ large enough
and take $\delta>0$ smaller if necessary
so that
\begin{eqnarray}\label{t0-2n}
\big(\hat V(T_1),\hat U(T_1)\big)
\in
Q\sb\delta
:=
(\mathbb{B}^2\sb{3\delta/4}\setminus \mathbb{B}^2\sb{2\delta/3})
\cap
\{(V,U)\sothat V\ge 0,\,U\ge 0
\},
\end{eqnarray}
and so that
\begin{eqnarray}\label{hat-v-u-always-small}
\big(\hat V(t),\hat U(t)\big)\in\mathbb{B}^2\sb{3\delta/4},
\qquad
\forall t\ge T_1.
\end{eqnarray}
By \eqref{v-u-tilde-small},
\begin{eqnarray}\label{V-U-small}
\norm{\tilde V(\cdot,\epsilon)}\sb{L^\infty}
+\norm{\tilde U(\cdot,\epsilon)}\sb{L^\infty}
=O(h(\epsilon)).
\end{eqnarray}
Since
$Q\sb\delta$
is strictly inside
$
\calK\sp{+}\sb{\delta,\nu}
\cup
\calK\sp{0}\sb{\delta}\cup\calK\sp{-}\sb{\delta,\nu}
$
(this is due to
choosing $\nu_0>0$ such that
\eqref{inside-half-ball} is satisfied
for $\nu\in(0,\nu_0)$),
we use
\eqref{t0-2n}
and \eqref{V-U-small}
to conclude that
there is
$\epsilon_1\in(0,\epsilon_0)$
such that
\begin{eqnarray}\label{initially-in-K}
\big(
V(T_1,\epsilon),
U(T_1,\epsilon)
\big)
=
\big(
\hat V(T_1)+\tilde V(T_1,\epsilon),
\ \hat U(T_1)+\tilde U(T_1,\epsilon)
\big)
\in
\calK\sp{+}\sb{\delta,\nu}
\cup
\calK\sp{0}\sb{\delta}\cup\calK\sp{-}\sb{\delta,\nu},
\nonumber
\\[1ex]
\forall\epsilon\in(0,\epsilon_1).
\end{eqnarray}
Moreover,
by
\eqref{hat-v-u-always-small} and \eqref{V-U-small},
we could take
\[
\epsilon_1\in(0,\epsilon_0)
\]
smaller if necessary
so that
\begin{eqnarray}\label{smaller-than-delta}
\big(
V(t,\epsilon),
U(t,\epsilon)
\big)
=
\big(
\hat V(t)+\tilde V(t,\epsilon),
\ \hat U(t)+\tilde U(t,\epsilon)
\big)
\in
\mathbb{B}^2\sb\delta,
\\[1ex]
\nonumber
\forall t\ge T_1,
\quad
\forall\epsilon\in(0,\epsilon_1).
\end{eqnarray}

\begin{lemma}\label{lemma-large-t-0}
One has
\begin{eqnarray}
\nonumber
V(t,\epsilon)>0,
\qquad
U(t,\epsilon)>0,
\qquad
\frac{V(t,\epsilon)}{4m}
<U(t,\epsilon)<\frac{2 V(t,\epsilon)}{m},
\\[1ex]
\nonumber
\forall t\ge T_1,
\qquad
\forall\epsilon\in(0,\epsilon_1).
\end{eqnarray}
\end{lemma}

\begin{proof}
We claim that the solution
$(V(t,\epsilon),U(t,\epsilon))$
stays in $\calK\sp{0}\sb\delta$
for all $t\ge T_1$.
First, we notice that if
$(V(T_1,\epsilon),U(T_1,\epsilon))\in\calK\sp{0}\sb{\delta}$,
then
for $t\ge T_1$ the trajectory $(V(t),U(t))$
could not leave $\calK\sp{0}\sb\delta$
through the arc of the $\delta$-circle
in the first quadrant (due to \eqref{smaller-than-delta}).
At the same time,
it can not leave $\calK\sp{0}\sb\delta$
through $(V,U)=(0,0)\in\calK\sp{0}\sb\delta$
because of the uniqueness of the solution
passing through $(0,0)$
(for $t\ge T_1\ge 2n$,
the right-hand side of the system
\eqref{zero-is-phi2}
is Lipschitz
in $(V,U)\in\calK\sp{0}\sb\delta$);
this unique solution is $V(t)\equiv U(t)\equiv 0$,
$t\ge T_1$.

The solution also could not leave
$\calK\sp{0}\sb\delta$ through the side
$U=2V/m$ (with $V>0$).
Indeed, the assumption that
$U(T_\ast,\epsilon)=2V(T_\ast,\epsilon)/m>0$
at some $T_\ast\ge T_1$
leads to a contradiction:
we choose $\nu>0$ small enough
(one can take $\nu=\min(\nu_0,V(T_\ast,\epsilon))>0$)
so that
$(V(T_\ast,\epsilon),U(T_\ast,\epsilon))
\in\calK\sp{+}\sb{\delta,\nu}$,
and then
Lemma~\ref{lemma-inside-omega-plus}
together with the bound \eqref{smaller-than-delta}
show that the solution would be trapped in
$\calK\sp{+}\sb{\delta,\nu}$
for all $t\ge T_1$,
hence would not be able to converge to zero
as $t\to\infty$.
For the same reason, the solution can not start
in this region initially, at $t=T_1$:
one should have
$(V(T_1,\epsilon),U(T_1,\epsilon))
\not\in\calK\sp{+}\sb{\delta,\nu}$ for any $\nu\in(0,\nu_0]$.

The same argument
(now with the aid of Lemma~\ref{lemma-inside-omega-minus})
shows that one can not have
$U=V/(4m)$, $V>0$ at some $T_\ast\ge T_1$,
neither can the solution start at $t=T_1$
in $\calK\sp{-}\sb{\delta,\nu}$ for any $\nu\in(0,\nu_0]$:
the solution $(V(t,\epsilon),U(t,\epsilon))$
would be trapped in $\calK\sp{-}\sb{\delta,\nu}$
for all $t\ge T_1$ and thus could not converge to zero.

Thus, by \eqref{initially-in-K},
the trajectory $(V(t,\epsilon),U(t,\epsilon))$
starts strictly inside $\calK\sp{0}\sb\delta$ at $t=T_1$
and stays there for all $t\ge T_1$.
The statement of the lemma follows.
\end{proof}

Due to $V$ being even and $U$ being odd in $t$,
Lemma~\ref{lemma-large-t-0}
also yields the inequality
\begin{eqnarray}\label{u-le-v}
\abs{U(t,\epsilon)}< \frac{2}{m} V(t,\epsilon),
\qquad
\abs{t}\ge T_1,
\qquad
\epsilon\in(0,\epsilon_1).
\end{eqnarray}
Let us now consider the case $\abs{t}\le T_1$.
By \eqref{v-u-tilde-small},
there is $C>0$ such that
\begin{eqnarray}\label{u-small-small-t}
\sup\sb{\abs{t}\le T_1}
\abs{U(t,\epsilon)}
\le
\sup\sb{\abs{t}\le T_1}
\hat U(t)
+
\norm{\tilde U(\cdot,\epsilon)}\sb{L^\infty}
\le
\sup\sb{\abs{t}\le T_1}
\hat U(t)
+
C h(\epsilon);
\end{eqnarray}
on the other hand,
again using \eqref{v-u-tilde-small},
we have, for all $\epsilon\in(0,\epsilon_1)$:
\begin{eqnarray}\label{u-small-small-tt}
\inf\sb{\abs{t}\le T_1}
V(t,\epsilon)
&\ge&
\inf\sb{\abs{t}\le T_1}\hat V(t)
-\norm{\tilde V(\cdot,\epsilon)}\sb{L^\infty}
\nonumber
\\[1ex]
&\ge&
\inf\sb{\abs{t}\le T_1}\hat V(t)
-C h(\epsilon)
\ge
\inf\sb{\abs{t}\le T_1}\hat V(t)/2>0
\end{eqnarray}
if we choose $\epsilon_1>0$ is so small that
$C h(\epsilon_1)<\inf\sb{\abs{t}\le T_1}\hat V(t)/2$.
It follows from
\eqref{u-small-small-t} and \eqref{u-small-small-tt}
that for some $C'<\infty$ we could write
\begin{eqnarray}\label{u-le-v-1}
\abs{U(t,\epsilon)}< C' V(t,\epsilon),
\qquad
\abs{t}\le T_1,
\qquad
\epsilon\in(0,\epsilon_1).
\end{eqnarray}
We require that $\epsilon_1>0$
be small enough, satisfying
$\epsilon_1\le\min\left(m/2,\,1/(2C')\right)$;
then the inequalities
\eqref{u-le-v} and \eqref{u-le-v-1}
yield \eqref{UV1},
finishing the proof of Proposition~\ref{prop-u-le-v}.
\end{proof}

Using the inequality \eqref{UV1},
one derives the bound \eqref{phi-beta-phi-large}:
\begin{eqnarray}
\nonumber
\phi\sb\omega\sp\ast\beta\phi\sb\omega
=v^2-u^2
=
\epsilon^{\frac 2 k}(V^2-\epsilon^2 U^2)
\ge
\epsilon^{\frac 2 k}\frac{3V^2}{4}
\ge
\epsilon^{\frac 2 k}\frac{2V^2+2\epsilon^2U^2}{4}
=\frac{\phi\sb\omega\sp\ast\phi\sb\omega}{2},
\\[1ex]
\nonumber
\omega\in(\omega\sb 1,m),
\end{eqnarray}
with $\omega\sb 1=\sqrt{m^2-\epsilon_1^2}$.
This completes the proof of
Theorem~\ref{theorem-solitary-waves}~\itref{theorem-solitary-waves-ii}.

\subsection{Sharp decay asymptotics and optimal estimates}
\label{sect-sharp}


We now prove
Theorem~\ref{theorem-solitary-waves}~\itref{theorem-solitary-waves-iii}.
We will derive the sharp
exponential decay of each of
$V$, $\hat V$, $U$, $\hat U$
and then prove that,
as the matter of fact,
$\tilde V$ and $\tilde U$
are pointwise dominated by $V$.
We recall that
$\hat V$ and $\hat U$
are obtained from NLS solitary waves
and that
\[
V(t,\epsilon)=\hat V(t)+\tilde V(t,\epsilon),
\qquad
U(t,\epsilon)=\hat U(t)+\tilde U(t,\epsilon);
\]
cf. \eqref{Vhatdef},
\eqref{def-V-U-hat}.

\begin{lemma}\label{lemma-v-exp}
There are $C_1>c_1>0$
such that
for all $\epsilon\in(0,\epsilon_1)$
and all $t\ge T_1$
one has
\begin{eqnarray}\label{v-exp-1}
\abs{V(t,\epsilon)}\ge c_1 t^{-(n-1)/2}e^{-t},
\qquad
\abs{V(t,\epsilon)}
+\abs{U(t,\epsilon)}
\le
C_1 t^{-(n-1)/2}e^{-t};
\end{eqnarray}
\begin{eqnarray}\label{v-exp-2}
\hat V(t)\ge c_1 t^{-(n-1)/2}e^{-t},
\qquad
\hat V(t)
+\abs{\hat U(t)}
\le
C_1 t^{-(n-1)/2}e^{-t};
\end{eqnarray}
\begin{eqnarray}\label{v-exp-3}
\abs{\tilde V(t,\epsilon)}
+\abs{\tilde U(t,\epsilon)}
\le
C_1 t^{-(n-1)/2}e^{-t}.
\end{eqnarray}
\end{lemma}
Above, $\epsilon_1>0$ is from Theorem~\ref{theorem-solitary-waves}~\itref{theorem-solitary-waves-ii}
and $T_1<\infty$ is from \eqref{t0-2n}.

\begin{proof}
The inequality \eqref{v-exp-3}
follows from
\eqref{v-exp-1} and \eqref{v-exp-2}.

The inequalities \eqref{v-exp-1} and \eqref{v-exp-2}
are proved similarly.
We will focus on \eqref{v-exp-1},
which is more involved;
then the inequalities \eqref{v-exp-2}
could be obtained by taking the limit $\epsilon\to 0$.

We introduce $\scrV(t,\epsilon)$ and $\scrU(t,\epsilon)$
such that
\begin{eqnarray}\label{v-scr-v}
&&
V(t,\epsilon)=t^{-(n-1)/2}\scrV(t,\epsilon),
\\
\label{u-scr-u}
&&
U(t,\epsilon)=
t^{-(n-1)/2}
\Big(
\scrU(t,\epsilon)
+
\frac{n-1}{2\upmu t}\scrV(t,\epsilon)
\Big),
\end{eqnarray}
where we use the notation
\[
\upmu=m+\omega,
\qquad
\omega=\sqrt{m^2-\epsilon^2}.
\]
Below, we will omit the dependence of
$V$, $U$, $\scrV$, $\scrU$,
$\omega$, and $\upmu$ on $\epsilon$.
By Lemma~\ref{lemma-large-t-0},
for $t\ge T_1$,
one has $\scrV(t)>0$ (since so is $V(t)$).
Then, applying inequalities from
Lemma~\ref{lemma-large-t-0} to the relation
\[
\scrU(t,\epsilon)
=t^{(n-1)/2}
\Big(
U(t,\epsilon)
-
\frac{n-1}{2\upmu t}V(t,\epsilon)
\Big)
\]
and using $\omega>m/2$, $t\ge T_1\ge 2n$
(cf. \eqref{omega-large} and \eqref{t0-2n}),
we obtain:
\begin{eqnarray}\label{mu-positive}
\scrU
\ge t^{(n-1)/2}
\Big(\frac{V}{4m}-\frac{n-1}{2(3m/2)2n}V\Big)
\ge t^{(n-1)/2}\frac{V}{12m}
=\frac{\scrV}{12m}>0,
\\[1ex]
\nonumber
\forall t\ge T_1,
\quad
\forall
\epsilon\in(0,\epsilon_1).
\end{eqnarray}
Substituting the expressions \eqref{v-scr-v}, \eqref{u-scr-u}
into the system \eqref{zero-is-phi1},
we obtain the equation
\[
\p\sb t\Big(\scrU+\frac{n-1}{2\upmu t}\scrV\Big)
-\frac{n-1}{2t}\Big(\scrU+\frac{n-1}{2\upmu t}\scrV\Big)
+\frac{n-1}{t}
\Big(\scrU+\frac{n-1}{2\upmu t}\scrV\Big)
+\frac{\scrV}{\upmu}
=\epsilon^{-2}f\scrV,
\]
which takes the form
\begin{eqnarray}\label{u-p-v}
\p\sb t\scrU
+\frac{n-1}{2\upmu t}\p\sb t\scrV
+\frac{n-1}{2t}\scrU
+\frac{(n-1)^2\scrV}{4\upmu t^2}
-\frac{(n-1)\scrV}{2\upmu t^2}
+\frac{\scrV}{\upmu}
=\frac{f}{\epsilon^2}\scrV,
\end{eqnarray}
and the equation
\begin{eqnarray}\label{v-p-u}
\p\sb t\scrV
-\frac{n-1}{2t}\scrV
+\upmu
\Big(\scrU+\frac{n-1}{2\upmu t}\scrV\Big)
=
\p\sb t\scrV+\upmu\scrU
=
\Big(\scrU+\frac{n-1}{2\upmu t}\scrV\Big)f.
\end{eqnarray}
Above, $f$ is evaluated at
$\tau=\epsilon^{2/k}V(t,\epsilon)^2-\epsilon^{2+2/k}U(t,\epsilon)^2$.
Multiplying \eqref{u-p-v} by $\upmu$
and adding \eqref{v-p-u},
we get:
\begin{eqnarray}
\p\sb t(\scrV+\upmu \scrU)
+(\scrV+\upmu \scrU)
+\frac{n-1}{2t}\p\sb t\scrV
+\frac{\upmu (n-1)}{2t}\scrU
+\frac{(n-1)(n-3)\scrV}{4t^2}
\nonumber
\\
=
\upmu
\frac{f}{\epsilon^2}\scrV
+\Big(\scrU+\frac{n-1}{2\upmu t}\scrV\Big)f.
\end{eqnarray}
Using \eqref{v-p-u}
to simplify the two terms
in the left-hand side
which contain a factor $\frac{n-1}{2t}$,
we get 
\begin{eqnarray*}
\p\sb t(\scrV+\upmu \scrU)
+(\scrV+\upmu \scrU)
+
\frac{n-1}{2t}
\Big(\scrU+\frac{n-1}{2\upmu t}\scrV\Big)f
+\frac{(n-1)(n-3)\scrV}{4t^2}
\\
=
\upmu\frac{f}{\epsilon^2}\scrV
+\Big(\scrU+\frac{n-1}{2\upmu t}\scrV\Big)f,
\end{eqnarray*}
which yields the inequality
\begin{eqnarray}\label{dv-v}
\Abs{
\p\sb t(\scrV+\upmu \scrU)
+(\scrV+\upmu \scrU)
}
\le
\frac{C}{t^2}
(\scrV+\scrU)
+C\frac{\abs{f}}{\epsilon^2}
(\scrV+\scrU),
\\[1ex]
\nonumber
\forall t\ge T_1,
\qquad
\forall\epsilon\in(0,\epsilon_1),
\end{eqnarray}
with some $C<\infty$;
we took into account
that both $\scrV$ and $\scrU$ are positive
(cf. \eqref{mu-positive}).
Since one has
$0<\frac{\scrV}{\scrV+\upmu \scrU}\le 1$
and $0<\frac{\scrU}{\scrV+\upmu\scrU}\le \frac{1}{\upmu}\le\frac{1}{m}$,
it follows from \eqref{dv-v} that
there is $C'<\infty$ such that
\begin{eqnarray}\label{le-le}
-1-\frac{c}{t^2}-C'\frac{\abs{f}}{\epsilon^2}
\le
\frac{\p\sb t(\scrV+\upmu\scrU)}{\scrV+\upmu\scrU}
\le
-1+\frac{c}{t^2}+C'\frac{\abs{f}}{\epsilon^2},
\\[1ex]
\nonumber
\forall t\ge T_1,
\qquad
\forall\epsilon\in(0,\epsilon_1).
\end{eqnarray}
We note that, by \eqref{ass-fo-1},
\[
\frac{
\abs{f\big(\epsilon^{2/k}(V(t,\epsilon)^2-\epsilon^2 U(t,\epsilon)^2)\big)}}
{\epsilon^2}
\le
2\abs{V(t,\epsilon)^2-\epsilon^2 U(t,\epsilon)^2}^k
,
\]
which is bounded
and exponentially decreasing as $t\to+\infty$
(uniformly in $\epsilon\in(0,\epsilon_1)$)
due to the exponential decay of
$V(t,\epsilon)=\hat V(t)+\tilde V(t,\epsilon)$,
$U(t,\epsilon)=\hat U(t)+\tilde U(t,\epsilon)$
in $t$,
which we proved
in Theorem~\ref{theorem-solitary-waves}.
Thus,
\[
\int\sb{T_1}^\infty
\Big(
\frac{C'}{t^2}+C'\frac{\abs{f}}{\epsilon^2}
\Big)\,dt
\le C''
\]
is bounded by some $C''<\infty$
which does not depend on $\epsilon\in(0,\epsilon_1)$.
This allows us to integrate \eqref{le-le}
from $T_1$
to an arbitrary value $t\ge T_1$;
we get
\[
-(t-T_1)-C''
\le
\ln(\scrV(t)+\upmu\scrU(t))
-\ln(\scrV(T_1)+\upmu\scrU(T_1))
\le -(t-T_1)+C'',
\]
which yields the desired inequalities
\eqref{v-exp-1}.
\end{proof}

The following result
immediately follows from the inequality
\eqref{v-exp-2} in Lemma~\ref{lemma-v-exp}
due to $\inf\sb{\abs{t}\le T_1}\hat V>0$.

\begin{corollary}\label{cor-v-exp-ast}
There are $C_1\sp\ast>c_1\sp\ast>0$
such that
\[
\hat V(t)\ge c_1\sp\ast\langle t\rangle^{-(n-1)/2}e^{-\abs{t}},
\qquad
\hat V(t)
+\abs{\hat U(t)}
\le
C_1\sp\ast\langle t\rangle^{-(n-1)/2}e^{-\abs{t}},
\qquad
\forall t\in\R.
\]
\end{corollary}

We claim that
the bound \eqref{v-exp-3} from Lemma~\ref{lemma-v-exp}
could be improved as follows.

\begin{lemma}\label{lemma-v-tilde-exp}
There is $C_2<\infty$
and
$T_2\in(T_1,+\infty)$
such that
\[
\abs{\tilde V(t,\epsilon)}+\abs{\tilde U(t,\epsilon)}
\le C_2 h(\epsilon) t^{-(n-1)/2}e^{-t},
\qquad
\forall t\ge T_2,
\quad
\forall\epsilon\in(0,\epsilon_1),
\]
with $h(\epsilon)$ from \eqref{def-h}.
\end{lemma}

Above, $\epsilon_1>0$
is from Theorem~\ref{theorem-solitary-waves}~\itref{theorem-solitary-waves-ii}
and $T_1<\infty$
is as in Lemma~\ref{lemma-v-exp}.

\begin{proof}
We define
$\tilde\scrV(t,\epsilon)$, $\tilde\scrU(t,\epsilon)$
by the relations similar to \eqref{v-scr-v}, \eqref{u-scr-u}:
\begin{eqnarray}\label{v-scr-v-tilde}
&&
\tilde V(t,\epsilon)
=t^{-\frac{n-1}{2}}\tilde\scrV(t,\epsilon),
\\[1ex]
\label{u-scr-u-tilde}
&&
\tilde U(t,\epsilon)
=
t^{-\frac{n-1}{2}}
\Big(\tilde\scrU(t)
+\frac{n-1}{2\upmu t}\tilde\scrV(t,\epsilon)\Big),
\end{eqnarray}
where
\[
\upmu=m+\omega,
\qquad
\omega=\sqrt{m^2-\epsilon^2},
\qquad
\epsilon\in(0,\epsilon_1).
\]
By \eqref{w-a-a},
the functions
$\tilde\scrV$, $\tilde\scrU$
satisfy
\[
\p\sb t\Big(\tilde\scrU+\frac{n-1}{2\upmu t}\tilde\scrV\Big)
+\frac{n-1}{2t}
\Big(\tilde\scrU+\frac{n-1}{2\upmu t}\tilde\scrV\Big)
+\frac{\tilde\scrV}{\upmu}
=(1+2k)\hat V^{2k}\tilde\scrV
-t^{\frac{n-1}{2}}G_1
\]
and
\[
\p\sb t\tilde\scrV-\frac{n-1}{2t}\tilde\scrV
+\upmu\big(\tilde\scrU+\frac{n-1}{2\upmu t}\tilde\scrV\big)
=t^{\frac{n-1}{2}}G_2,
\]
which we rewrite as
\begin{eqnarray}\label{u-p-v-1}
&&
\p\sb t\tilde\scrU
+\frac{n-1}{2\upmu t}\p\sb t\tilde\scrV
+\frac{n-1}{2t}
\tilde\scrU
+
\frac{(n-1)(n-3)}{4\upmu t^2}\tilde\scrV
+\frac{\tilde\scrV}{\upmu}
\\
\nonumber
&&
\qquad\qquad\qquad\qquad
\qquad\qquad
=(1+2k)\hat V^{2k}\tilde\scrV
-t^{\frac{n-1}{2}}G_1,
\\
\label{v-p-u-1}
&&
\p\sb t\tilde\scrV+\upmu\tilde\scrU
=t^{\frac{n-1}{2}}G_2.
\end{eqnarray}
We multiply \eqref{u-p-v-1} by $\upmu$;
adding and subtracting \eqref{v-p-u-1},
we obtain, respectively,
\begin{eqnarray}
&&
\p\sb t(\upmu\tilde\scrU+\tilde\scrV)
+(\upmu\tilde\scrU+\tilde\scrV)
\nonumber
\\
&&
\qquad\qquad\qquad
=
(1+2k)\upmu\hat V^{2k}\tilde\scrV
-\textstyle{\frac{(n-1)(n-3)}{4 t^2}}
\tilde\scrV
+t^{\frac{n-1}{2}}\Big(G_2-\upmu G_1-\frac{n-1}{2t}G_2\Big),
\nonumber
\\[1ex]
&&
\p\sb t(\upmu\tilde\scrU-\tilde\scrV)
-(\upmu\tilde\scrU-\tilde\scrV)
\nonumber
\\
&&
\qquad\qquad\qquad
=
(1+2k)\upmu\hat V^{2k}\tilde\scrV
-\textstyle{\frac{(n-1)(n-3)}{4 t^2}}
\tilde\scrV
-t^{\frac{n-1}{2}}\Big(\upmu G_1+\frac{n-1}{2t}G_2+G_2\Big).
\nonumber
\end{eqnarray}
Multiplying the above relations by
$e^t$ and $e^{-t}$, respectively,
we rewrite them as
\begin{eqnarray}
&&
\quad
\p\sb t(e^t(\upmu\tilde\scrU+\tilde\scrV))
\nonumber
\\
&&
\qquad
=
e^t
\Big(
(1+2k)\upmu\hat V^{2k}\tilde\scrV
-
\textstyle{\frac{(n-1)(n-3)}{4 t^2}}
\tilde\scrV
+t^{\frac{n-1}{2}}\big(G_2-\upmu G_1-
\textstyle{\frac{n-1}{2t}}
G_2\big)
\Big),
\label{d-u-plus-v}
\\
&&
\quad
\p\sb t(e^{-t}(\upmu\tilde\scrU-\tilde\scrV))
\nonumber
\\
&&
\qquad
=
e^{-t}
\Big(
(1+2k)\upmu\hat V^{2k}\tilde\scrV
-
\textstyle{\frac{(n-1)(n-3)}{4 t^2}}
\tilde\scrV
-t^{\frac{n-1}{2}}
\big(\upmu G_1
+
\textstyle{\frac{n-1}{2t}}G_2
+G_2\big)
\Big).
\label{d-u-minus-v}
\end{eqnarray}
We are to integrate the above relations in $t$;
before we do this, we need
a special treatment for the last term
in the right-hand side of \eqref{d-u-plus-v}.

\begin{lemma}\label{lemma-int-finite}
There is $C<\infty$ such that
\[
\Abs{
\int_T^{T'}
t^{\frac{n-1}{2}}e^t\Big(G_2-\upmu G_1-\frac{n-1}{2t}G_2\Big)\,dt
}
\le C h(\epsilon),
\qquad
\forall T'\ge T\ge T_1,
\quad
\forall\epsilon\in(0,\epsilon_1),
\]
with $h(\epsilon)$ from \eqref{def-h}.
\end{lemma}

\begin{proof}
Applying the bounds on $G_1$ and $G_2$
from Lemma~\ref{lemma-g1-g2-bounds},
we can treat all the terms
(obtaining the desired bound $O(h(\epsilon))$)
except for
the ones linear in $\hat V$ and $\hat U$;
the worry comes from e.g.
$\langle t\rangle^{(n-1)/2}e^t\hat V(t)\ge c_1\sp\ast>0$
(cf. Corollary~\ref{cor-v-exp-ast}),
whose contribution to the integral considered in the lemma
would not be bounded uniformly in $T$, $T'$;
let us try to combine all such terms.
The expression
$G_2-\upmu G_1-\frac{n-1}{2t}G_2$
contributes the following terms
which are linear in $\hat V$ and $\hat U$:
\[
(m-\omega)\hat U-(m+\omega)\frac{m-\omega}{2m(m+\omega)}\hat V
-\frac{n-1}{2t}(m-\omega)\hat U
=
(m-\omega)
\Big(
\hat U-\frac{\hat V}{2m}
-\frac{n-1}{2t}\hat U
\Big).
\]
Using
\eqref{def-hat-phi},
we rewrite the above as
$
(m-\omega)
\Big(
\hat U
+\hat U'
+\frac{n-1}{2t}\hat U
-\abs{\hat V}\sp{2k}\hat V
\Big).
$
Since
\[
\int\limits_{T}^{T'}
t^{\frac{n-1}{2}}e^t
\Big(
\hat U
+\hat U'
+\frac{n-1}{2t}\hat U
-\abs{\hat V}\sp{2k}\hat V
\Big)\,dt
=
\int\limits_{T}^{T'}
\p\sb t
\big(
t^{\frac{n-1}{2}}
e^t
\hat U
\big)\,dt
-
\int\limits_{T}^{T'}
t^{\frac{n-1}{2}}e^t
\abs{\hat V}\sp{2k}\hat V
\,dt,
\]
with both integrals in the right-hand side
being bounded
uniformly in $T'\ge T\ge T_1$
(due to the bounds on $\hat U$ and $\hat V$
from Lemma~\ref{lemma-v-exp}),
while $m-\omega=O(\epsilon^2)$,
the conclusion follows.
\end{proof}

For some fixed $T_2\ge T_1$
(to be specified later),
we denote
\begin{eqnarray}\label{def-sup}
M(\epsilon)=\sup\sb{t\ge T_2}
e^t\left(\abs{\tilde\scrV(t,\epsilon)}+\abs{\tilde\scrU(t,\epsilon)}\right),
\qquad
\epsilon\in(0,\epsilon_1).
\end{eqnarray}
We note that,
due to the bounds
\eqref{v-exp-3}
from Lemma~\ref{lemma-v-exp}
and the definitions \eqref{v-scr-v-tilde} and \eqref{u-scr-u-tilde},
one has
\[
\sup\sb{\epsilon\in(0,\epsilon_1)}M(\epsilon)<\infty.
\]
Integrating
\eqref{d-u-plus-v}
from $T_2$ to some $t\ge T_2$
and using Lemma~\ref{lemma-int-finite},
one gets:
\begin{eqnarray}\label{og-1}
\Abs{
e^{t}
\abs{
\upmu\tilde\scrU(t,\epsilon)+\tilde\scrV(t,\epsilon)}
-
e^{T_2}
\abs{
\upmu\tilde\scrU(T_2,\epsilon)+\tilde\scrV(T_2,\epsilon)}
}
\nonumber
\\
\le
C\int\sb{T_2}\sp{t}
\Big(\hat V(s)^{2k}+\frac{1}{s^2}\Big)
e^s\tilde\scrV(s,\epsilon)\,ds
+C h(\epsilon).
\end{eqnarray}
Taking into account that,
due to Theorem~\ref{theorem-solitary-waves}
and \eqref{v-scr-v-tilde} and \eqref{u-scr-u-tilde},
one has
\begin{eqnarray}\label{u-v-tilde-l-infty}
\abs{\tilde\scrV(t,\epsilon)}
+
\abs{\tilde\scrU(t,\epsilon)}
=O(h(\epsilon)),
\qquad
\forall t\ge T_1,
\quad
\forall\epsilon\in(0,\epsilon_1),
\end{eqnarray}
and using \eqref{def-sup},
we rewrite \eqref{og-1}
as
\begin{eqnarray}\label{og-1a}
e^{t}
\abs{
\upmu\tilde\scrU(t,\epsilon)+\tilde\scrV(t,\epsilon)}
\le
M(\epsilon) C\int\sb{T_2}\sp{t}
\Big(\hat V(s)^{2k}+\frac{1}{s^2}
\Big)\,ds
+C h(\epsilon),
\end{eqnarray}
with some $C<\infty$
(which does not
depend on $\epsilon\in(0,\epsilon_1)$,
$T_2\ge T_1$, and $t\ge T_2$).

We now integrate \eqref{d-u-minus-v}
from $t\ge T_2$ to $+\infty$.
Due to the presence of the factor $e^{-t}$
in the right-hand side,
the last term does not need a special treatment
such as in Lemma~\ref{lemma-int-finite}:
the bounds on $G_1$ and $G_2$
from Lemma~\ref{lemma-g1-g2-bounds}
together with the exponential decay of
$V$, $U$, $\hat V$, $\hat U$
from Lemma~\ref{lemma-v-exp}
are sufficient.
The integration yields
\[
e^{-t}
\abs{
\upmu\tilde\scrU(t,\epsilon)-\tilde\scrV(t,\epsilon)}
\le
C\int\limits\sb{t}^\infty
\Big(\hat V^{2k}(s)+\frac{1}{s^2}\Big)
e^{-s}
\tilde\scrV(s,\epsilon)\,ds
+C h(\epsilon)e^{-2t},
\]
again with some $C<\infty$
which does not
depend on $\epsilon\in(0,\epsilon_1)$,
$T_2$, and $t$.
We took into account that
in the left-hand side
the boundary term at $t=\infty$
disappears due to
\eqref{u-v-tilde-l-infty}.
Using \eqref{def-sup},
we rewrite the above relation as
\begin{eqnarray}\label{og-2}
e^{t}
\abs{
\upmu\tilde\scrU(t,\epsilon)-\tilde\scrV(t,\epsilon)}
\le
M(\epsilon) C
\int\limits\sb{t}^\infty
e^{2t-2s}
\Big(\hat V^{2k}(s)+\frac{1}{s^2}
\Big)
\,ds
+C h(\epsilon).
\end{eqnarray}
Since
\[
\abs{\tilde\scrV}+\abs{\tilde\scrU}
\le
\frac{\abs{\upmu\tilde\scrU+\tilde\scrV}
+
\abs{\upmu\tilde\scrU-\tilde\scrV}
}{2}
+
\frac{\abs{\upmu\tilde\scrU+\tilde\scrV}
+
\abs{\upmu\tilde\scrU-\tilde\scrV}
}{2\upmu},
\]
the inequalities \eqref{og-1a} and \eqref{og-2}
lead to the bound
\begin{eqnarray}\label{m-c-m}
M(\epsilon)
\le M(\epsilon) C\int\sb{T_2}^\infty
\Big(
\hat V(s)^{2k}+\frac{1}{s^2}
\Big)\,ds
+C h(\epsilon),
\qquad
\forall\epsilon\in(0,\epsilon_1),
\end{eqnarray}
with some constant $C<\infty$
(which does not depend on $\epsilon\in(0,\epsilon_1)$
and $T_2$).
Now we can choose $T_2$:
we set $T_2\ge T_1$ to be sufficiently large
so that the coefficient at $M(\epsilon)$ in the right-hand side
is smaller than $1/2$
(due to the exponential decay of
$\hat V$ and $\tilde\scrV$,
such a value of $T_2$ could be chosen
independent of $\epsilon\in(0,\epsilon_1)$).
Now \eqref{m-c-m}
turns into the inequality
$
M(\epsilon)\le 2 C h(\epsilon),
$
valid for all $\forall\epsilon\in(0,\epsilon_1)$,
and \eqref{def-sup}
gives
\[
\abs{\tilde\scrV(t)}
+\abs{\tilde\scrU(t)}
\le
2C h(\epsilon)e^{-t},
\qquad
\forall t\ge T_2,
\quad
\forall\epsilon\in(0,\epsilon_1),
\]
yielding the bounds stated in the lemma.
\end{proof}

\begin{lemma}\label{lemma-v-tilde-quarter}
There is
$C_3<\infty$
such that
\begin{eqnarray}\label{v-tilde-quarter}
\abs{\tilde V(t,\epsilon)}
+\abs{\tilde U(t,\epsilon)}
\le
C_3
h(\epsilon)
\hat V(t),
\qquad
\forall t\in\R,
\quad
\forall\epsilon\in(0,\epsilon_1),
\end{eqnarray}
with $h(\epsilon)$ from \eqref{def-h}.
\end{lemma}

Above,
$\epsilon_1>0$ is from
Theorem~\ref{theorem-solitary-waves}~\itref{theorem-solitary-waves-ii}.

\begin{proof}
Using the bound from below on
$\hat V$ from Lemma~\ref{lemma-v-exp}
and bound from above on $\tilde V$ and $\tilde U$
from Lemma~\ref{lemma-v-tilde-exp},
we conclude that
the inequality
\eqref{v-tilde-quarter}
takes place for $t\ge T_2$
(and also for $t\le -T_2$)
and for all $\epsilon\in(0,\epsilon_1)$.
Let us now consider the case $\abs{t}\le T_2$.
By the inequality \eqref{v-u-tilde-small},
there is $C<\infty$ such that
\[
\norm{\tilde V(\cdot,\epsilon)}\sb{L^\infty}
+
\norm{\tilde U(\cdot,\epsilon)}\sb{L^\infty}
\le C h(\epsilon)
\le \frac{C}{\hat V(T_2)}
h(\epsilon)\hat V(t),
\qquad
\forall
\abs{t}\le T_2,
\quad
\forall\epsilon\in(0,\epsilon_0);
\]
in the last inequality,
we used the fact that
$\hat V(t)$ is positive and monotonically decreasing
for $t>0$.
This proves the desired inequality for $\abs{t}\le T_2$.
\end{proof}

Lemma~\ref{lemma-v-tilde-quarter}
proves \eqref{v-u-tilde-smaller-hat-v}.

\medskip

The pointwise bound
\eqref{phi-asymptotics}
follows from the inequality
$\hat V(t)\le C_1\sp\ast\langle t\rangle^{-(n-1)/2}e^{-\abs{t}}$
for $t\in\R$
and $\epsilon\in(0,\epsilon_1)$
(cf. Corollary~\ref{cor-v-exp-ast})
and also from
\eqref{UV1}
and \eqref{v-u-tilde-smaller-hat-v}
which show that
$\tilde V$, $\hat U$, and $\tilde U$
are all pointwise dominated by $\hat V$.

This completes the proof
of Theorem~\ref{theorem-solitary-waves}~\itref{theorem-solitary-waves-iii}.

For our convenience, we take
$\epsilon_1$
small enough so that
$C_3 h(\epsilon_1)<1/2$;
then, for the later use, we have
\begin{eqnarray}\label{UV2}
\abs{\tilde V(t,\epsilon)}
+\abs{\tilde U(t,\epsilon)}
<\frac 1 2\hat V(t),
\qquad
\forall t\in\R,
\qquad
\forall\epsilon\in(0,\epsilon_1).
\end{eqnarray}

\section{Improved error estimates}
\label{sect-improvement}


Now we prove
Theorem~\ref{theorem-solitary-waves}~\itref{theorem-solitary-waves-v}.
The assumption \eqref{ass-f-0},
together with the bounds on the amplitude of solitary
waves \eqref{no-more},
allows us to assume that there is $c<\infty$
such that
\begin{eqnarray}\label{ass-f}
\abs{f(\tau)-\abs{\tau}^{k}}
\le c\abs{\tau}^{K},
\qquad
\abs{f(\tau)}\le (c+1)\abs{\tau}^{k},
\qquad
\tau\in\R.
\end{eqnarray}
The improvement of the estimates stated in
Theorem~\ref{theorem-solitary-waves}~\itref{theorem-solitary-waves-i}
and~\itref{theorem-solitary-waves-iii}
comes from having better bounds
on the second and third terms
from the right-hand side of \eqref{g-g-1-0}:
when estimating e.g.
$\abs{V^2-\epsilon^2 U^2}^k-\abs{V^2}^k$,
we no longer have to rely on
Lemma~\ref{lemma-a-a},
being able to use the Taylor expansions instead.

We recall that $\Lambda\sb k<\infty$
was defined in \eqref{def-Lambda}.

\begin{lemma}\label{lemma-f-is-u}
There is $C_4<\infty$
such that
for any numbers
$\hat V,\,\hat U,\,\tilde V,\,\tilde U\in[-\Lambda\sb k,\Lambda\sb k]$,
$V=\hat V+\tilde V$, and $U=\hat U+\tilde U$
which satisfy
\begin{eqnarray}\label{v-u-abstract}
\epsilon_1\abs{U}\le\frac 1 2 V,
\qquad
\abs{\tilde V}\le\frac 1 2 \hat V,
\end{eqnarray}
one has
\[
\Abs{f\big(\epsilon^{2/k}(V^2-\epsilon^2 U^2)\big)-\epsilon^2 \hat V^{2k}}
\le
C_4
\left(
\epsilon^{2+2\varkappa}\hat V^{2k}
+\epsilon^2\hat V^{2k-1}\abs{\tilde V}
\right),
\qquad
\forall
\epsilon\in(0,\epsilon_1).
\]
Above,
\[
\varkappa:=\min\Big(1,\frac{K}{k}-1\Big)
\]
was defined in \eqref{def-varkappa-1}.
\end{lemma}

\begin{proof}
We proceed as follows:
\begin{eqnarray}
&&
\abs{f\big(\epsilon^{2/k}(V^2-\epsilon^2 U^2)\big)-\epsilon^2\hat V^{2k}}
\nonumber
\\
&&
\le
\abs{f\big(\epsilon^{2/k}(V^2-\epsilon^2 U^2)\big)
-\epsilon^2(V^2-\epsilon^2 U^2)^k}
+
\epsilon^2\abs{(V^2-\epsilon^2 U^2)^k-V^{2k}}
+
\epsilon^2\abs{V^{2k}-\hat V^{2k}}
\nonumber
\\
&&
\le
c
\epsilon^{2K/k}(V^2-\epsilon^2 U^2)^K
+
O(\epsilon^2 V^{2(k-1)}\epsilon^2 U^2)
+
O(\epsilon^2\hat V^{2k-1}\tilde V),
\nonumber
\end{eqnarray}
where the three terms
from the second line were estimated
using \eqref{ass-f}
and \eqref{v-u-abstract}.
The conclusion follows.
\end{proof}


Here is an improvement of Lemma~\ref{lemma-g1-g2-bounds}.

\begin{lemma}\label{lemma-g1-g2-bounds-better}
There is $C<\infty$
such that
for any numbers
$\hat V,\,\hat U,\,\tilde V,\,\tilde U\in[-\Lambda\sb k,\Lambda\sb k]$,
$V=\hat V+\tilde V$, and $U=\hat U+\tilde U$
which satisfy
\eqref{v-u-abstract}
and additionally
\begin{eqnarray}\label{addition}
\abs{\hat{U}}\le \frac{C_1}{c_1}\hat{V},
\end{eqnarray}
with $c_1$ and $C_1$ from Lemma~\ref{lemma-v-exp},
one has:
\begin{eqnarray}
&&
\Abs{G_1
-
\Big(\frac{1}{m+\omega}-\frac{1}{2m}\Big)\hat V
-k\epsilon^2 V^{2k-1}U^2
}
\le
C
\big(
(\epsilon^{2\frac{K}{k}-2}+\epsilon^{4})\hat V^{2k+1}
+
\hat V^{2k-1}\tilde V^2
\big)
,
\nonumber
\\[2ex]
&&
\abs{G_2-\epsilon^2\hat V^{2k}\hat U-(m-\omega)\hat U}
\le
C\,(\epsilon^{2+2\varkappa}\hat V^{2k}
+\epsilon^2\hat V^{2k-1}\tilde V
)\abs{U}
+\epsilon^2\hat V^{2k}\abs{\tilde U},
\nonumber
\\[2ex]
&&
\abs{G_1}+\abs{G_2}
\le
C
\epsilon^{2\varkappa}\hat V
+
C\hat V^{2k-1}\tilde V^2,
\nonumber
\end{eqnarray}
for all $\epsilon\in(0,\epsilon_1)$.
Above,
$G_1=G_1(\epsilon,\tilde V,\tilde U)$
and $G_2=G_2(\epsilon,\tilde V,\tilde U)$.
\end{lemma}

\begin{proof}
We start with the definition \eqref{def-g1}
of $G_1(\epsilon,\tilde V,\tilde U)$
and apply the inequalities \eqref{v-u-abstract}:
\begin{eqnarray}
&&
G_1(\epsilon,\tilde V,\tilde U)
=
-\epsilon^{-2}f\big(\epsilon^{2/k}(V\sp 2- \epsilon\sp 2 U^2)\big)V
+
\hat V\sp{2k}\hat V
+(1+2k)\hat V\sp{2k}\tilde{V}
+\frac{\hat{V}}{m+\omega}
-\frac{\hat{V}}{2m}
\nonumber
\\
&&
=
-\big(
\epsilon^{-2}f\big(\epsilon^{2/k}(V^2-\epsilon^2 U^2)\big)
-\abs{V^2-\epsilon^2 U^2}^k\big)V
-
\big(\abs{V^2-\epsilon^2 U^2}^k-V^{2k}\big)V
\nonumber
\\
&&
\qquad
-\big(
V^{2k+1}-\hat V^{2k+1}
-(1+2k)\hat V^{2k}\tilde V
\big)
+\Big(\frac{1}{m+\omega}-\frac{1}{2m}\Big)\hat V
\nonumber
\\
&&
=
O(\epsilon^{2\frac{K}{k}-2}V^{2K+1})
+k\epsilon^2 V^{2k-1}U^2
+O(\epsilon^{4} V^{2k-3}U^4)
+O(\hat V^{2k-1}\tilde V^2)
+\Big(\frac{1}{m+\omega}-\frac{1}{2m}\Big)\hat V.
\nonumber
\end{eqnarray}
Let us point out that
the third term in the right hand side
in the line above
has the factor of $\epsilon^4$,
which contributes $\epsilon^4$
into the first conclusion of the lemma.

For $G_2(\epsilon,\tilde V,\tilde U)$ from \eqref{def-g2},
we have:
\begin{eqnarray}
\nonumber
G_2-\epsilon^2\hat V^{2k}\hat U-(m-\omega)\hat U
&=&
f\big(\epsilon^{2/k}(V^2-\epsilon^2 U^2)\big)U
-\epsilon^2\hat V^{2k}\hat U
\\
\nonumber
&=&
\big(f\big(\epsilon^{2/k}(V^2-\epsilon^2 U^2)\big)
-\epsilon^2\hat V^{2k}\big)U
+\epsilon^2\hat V^{2k}\tilde U.
\end{eqnarray}
Applying Lemma~\ref{lemma-f-is-u}
to the right-hand side, we have:
\[
\Abs{G_2
-\epsilon^2\hat V^{2k}\hat U-(m-\omega)\hat U}
\le
C_4\left(
\epsilon^{2+2\varkappa}\hat V^{2k}
+\epsilon^2\hat V^{2k-1}\abs{\tilde V}
\right)\abs{U}
+\epsilon^2\hat V^{2k}\abs{\tilde U}.
\]
The second conclusion of the lemma follows.

Taking into account
\eqref{v-u-abstract},
the bound on $\abs{G_1}+\abs{G_2}$ also follows;
we need to mention that,
due to
\eqref{v-u-abstract}
and \eqref{addition},
both $\abs{\hat U}$ and $\abs{\tilde U}$
are estimated by $\hat V$.
\end{proof}

We notice that, due to
\eqref{UV1} and \eqref{UV2},
the functions $\hat V(t)$, $\hat U(t)$, $\tilde V(t,\epsilon)$,
and $\tilde U(t,\epsilon)$
satisfy inequalities \eqref{v-u-abstract}
for all $t\in\R$ and $\epsilon\in(0,\epsilon_1)$.
Also,
$\hat U(t)$ and $\hat V(t)$
satisfy the inequality \eqref{addition}
due to
\eqref{v-exp-2}
from Lemma~\ref{lemma-v-exp}.
Using Lemma~\ref{lemma-g1-g2-bounds-better}
in place of Lemma~\ref{lemma-g1-g2-bounds},
we can rewrite the proof of Lemma~\ref{lemma-small-g-new}
as follows.

\begin{lemma}\label{lemma-small-g-new-1}
There is $C<\infty$ such that
for any $\epsilon\in(0,\epsilon_1)$
and any
$
\begin{bmatrix}\tilde V\\ \tilde U\end{bmatrix}
\in X\sb{e,o}$ which satisfies
\begin{eqnarray}\label{tilde-smaller-than-hat-new-1}
&
\Norm{e^{\gamma\langle t\rangle}
\begin{bmatrix}\tilde V\\ \tilde U\end{bmatrix}
}\sb{X}
\le\Norm{e^{\gamma\langle t\rangle}
\begin{bmatrix}
\hat V (t) \\ \hat U (t)
\end{bmatrix}
}\sb{X},
\\[2ex]
\label{v-u-abstract-epsilon}
&
\epsilon_1\abs{U(t,\epsilon)}
\le\frac 1 2 V(t,\epsilon),
\qquad
\abs{\tilde V(t,\epsilon)}
\le\frac 1 2\hat V(t,\epsilon),
\\[2ex]
&
\qquad\qquad\qquad\qquad\qquad
\forall t\in\R,
\qquad
\forall\epsilon\in(0,\epsilon_1),
\nonumber
\end{eqnarray}
where
$
V(t,\epsilon)=\hat V(t)+\tilde V(t,\epsilon)
$
and
$
U(t,\epsilon)=\hat U(t)+\tilde U(t,\epsilon),
$
one has
\begin{eqnarray}\label{g-small-new-1}
\|
e^{(1+2k)\gamma\langle t\rangle}
G(\epsilon, \tilde W)\|\sb{X}
\leq
C\left(\epsilon\sp{2\varkappa}
  + \|e^{\gamma\langle t\rangle}\tilde W\|\sb{X}^2
\right),
\end{eqnarray}
with $\varkappa$ from \eqref{def-varkappa-1}.
\end{lemma}

\begin{proof}
For $\tilde V(t,\epsilon)$ and $\tilde U(t,\epsilon)$
as in the assumptions of the lemma,
due to Lemma~\ref{lemma-g1-g2-bounds-better},
one has
\begin{eqnarray}\label{g-g-1-1}
\abs{G_1(\epsilon,\tilde V,\tilde U)}
+\abs{G_2(\epsilon,\tilde V,\tilde U)}
\le
C
\epsilon^{2\varkappa}\hat V(t,\epsilon)
+
C\hat V(t,\epsilon)^{2k-1}\tilde V(t,\epsilon)^2,
\\
\nonumber
\forall\epsilon\in(0,\epsilon_1),
\qquad
\forall t\in\R.
\end{eqnarray}
Multiplying the first term in the right-hand side
by $e^{(1+2k)\gamma\langle t\rangle}$
and using \eqref{such-w-2} and \eqref{tilde-smaller-than-hat-new-1},
we bound the resulting $X$-norm
by $C\epsilon^{2\varkappa}$,
with some $C<\infty$.
The second term
in the right-hand side of \eqref{g-g-1-1}
is homogeneous of order
$1+2k$
in $\tilde V$ and $\hat V$;
we multiply it by the factor
$e^{(1+2k)\gamma\langle t\rangle}$,
absorbing
$e^{\gamma\langle t\rangle}$
into each power of $\hat V$ and $\tilde V$
and bounding the $X$-norm of the result by
$C\norm{e^{\gamma\langle t\rangle}\hat V}_{X}^{2k-1}
\norm{e^{\gamma\langle t\rangle}\tilde V}_{X}^{2}
\le
C'\norm{e^{\gamma\langle t\rangle}\tilde V}_{X}^2$.
\end{proof}

Now we use
Lemma~\ref{lemma-small-g-new-1}
to improve the estimates on
$\tilde W$.

\begin{lemma}\label{lemma-mu-into-better}
One can take $\epsilon_1>0$ smaller if necessary so that,
for some $ b_1>0$,
\[
\norm{e^{\gamma\langle t\rangle}\tilde W(\epsilon)}_{X^1}
\le  b_1\epsilon^{2\varkappa},
\qquad
\forall\epsilon\in(0,\epsilon_1).
\]
\end{lemma}

\begin{proof}
We recall the relation \eqref{itf} satisfied by $\tilde W$:
\[
e^{\gamma\langle t\rangle}
\tilde W =
e^{-2k\gamma\langle t\rangle}
A\sb\gamma(\epsilon)\sp{-1}
e^{(1+2k)\gamma\langle t\rangle}
G\big(\epsilon, e\sp{-\gamma\langle t\rangle}e^{\gamma\langle t\rangle}\tilde W\big).
\]
Using the continuity of the mapping
\eqref{ttm}
and estimating $G(\epsilon,\tilde W)$
by Lemma~\ref{lemma-small-g-new-1},
we obtain:
\[
\norm{e^{\gamma\langle t\rangle}\tilde W}_{X^1}
=
\norm{e^{-2k\gamma\langle t\rangle}
A\sb\gamma(\epsilon)^{-1}}_{X\to X^1}
\norm{
e^{(1+2k)\gamma\langle t\rangle}
G(\epsilon,\tilde W)}_X
\le
C(e^{2\varkappa}+\norm{e^{\gamma\langle t\rangle} \tilde W}_X^2).
\]
Since
$
\norm{e^{\gamma\langle t\rangle} \tilde W}_X
\le
\norm{e^{\gamma\langle t\rangle} \tilde W}_{X^1}$
(cf. \eqref{x-x1}),
the above relation yields the bound stated in the lemma
as long as $\norm{e^{\gamma\langle t\rangle} \tilde W}_{X^1}$
is sufficiently small
(which holds due to \eqref{v-u-tilde-small}
as long as
$\epsilon\in(0,\epsilon_1)$
with $\epsilon_1>0$
small enough).
\end{proof}

Lemma~\ref{lemma-mu-into-better}
improves the estimates
from
Theorem~\ref{theorem-solitary-waves}~\itref{theorem-solitary-waves-i}
on the error terms $\tilde V$, $\tilde U$,
proving \eqref{v-u-tilde-small-better}.

\medskip

We also do the second pass
over the proof of
Theorem~\ref{theorem-solitary-waves}~\itref{theorem-solitary-waves-iii},
improving
in \eqref{v-u-tilde-smaller-hat-v}
the factor $h(\epsilon)$ to $\epsilon^{2\varkappa}$.
For this, we rewrite the proof of Lemma~\ref{lemma-v-tilde-exp},
where the bounds on $G_1$, $G_2$
come from Lemma~\ref{lemma-g1-g2-bounds-better}
instead of
Lemma~\ref{lemma-g1-g2-bounds}.
We also rewrite the proof of
Lemma~\ref{lemma-v-tilde-quarter}
with $\epsilon^{2\varkappa}$ instead of $h(\epsilon)$
(we use \eqref{v-u-tilde-small-better}
in place of \eqref{v-u-tilde-small}).
This brings us at
$
\abs{\tilde V(t,\epsilon)}
+\abs{\tilde U(t,\epsilon)}
\le
C
\epsilon^{2\varkappa}
\hat V(t),
$
with some $C<\infty$,
valid for all $t\in\R$ and all $\epsilon\in(0,\epsilon_1)$,
thus proving \eqref{v-u-tilde-smaller-hat-v-b3}.

\medskip

This completes the proof of
Theorem~\ref{theorem-solitary-waves}.

\section{Solitary waves in the nonrelativistic limit.
The case $f\in C^1$}
\label{sect-waves-nrl-1}

We now turn to the case
when $f\in C^1(\R\setminus\{0\})\cap C(\R)$
satisfies both
the assumption~\eqref{ass-f-1} and~\eqref{ass-fp-0}.
Just like the former assumption leads to
\eqref{ass-f},
the assumption \eqref{ass-fp-0}
allows us to accept that there is $C<\infty$
such that
\begin{eqnarray}\label{ass-fp}
\abs{\tau f'(\tau)-k\abs{\tau}^k}\le C\abs{\tau}^K,
\qquad
\abs{\tau f'(\tau)}\le (C+k)\abs{\tau}^k,
\qquad
\tau\in\R,
\end{eqnarray}
where $k\in(0,2/(n-2))$ (any $k>0$ if $n\le 2$)
and $K>k$.
Now we will be able to prove
uniqueness and regularity of the family of solitary waves bifurcating from the nonrelativistic limit. This amounts to noticing that in \eqref{mapping-mu},
taking into account Theorem~\ref{theorem-solitary-waves}~\itref{theorem-solitary-waves-ii},
we actually recover some features of the implicit function theorem. A careful analysis shows that the main obstacle to its application  is the lack of regularity of the mapping $f$
in \eqref{def-g1}, \eqref{def-g2}.
This closer look shows that the unique obstacle are the terms $f\big(\epsilon^{2/k}(V\sp 2- \epsilon\sp 2 U^2)\big)V$ and $f\big(\epsilon^{2/k}(V\sp 2- \epsilon\sp 2 U^2)\big)U$,
which with \eqref{ass-fp} can now be treated.

\subsection{Improved regularity of the groundstate}

Let us prove
Theorem~\ref{theorem-existence-c1}~\itref{theorem-existence-c1-i}.
By Theorem~\ref{theorem-solitary-waves},
we already have $\phi\sb\omega\in H^1(\R^n,\C^N)$,
$\omega\in(\omega\sb 1,m)$,
with $\omega\sb 1=\sqrt{m^2-\epsilon_1^2}$,
with $\epsilon_1>0$ from
Theorem~\ref{theorem-solitary-waves}~\itref{theorem-solitary-waves-ii};
we need to show how to get the improvement
in the regularity of $\phi\sb\omega$
under better regularity of $f$.

We start with the improvement of
regularity of $V,\,U$ proved in Lemma~\ref{lemma-v-u-c1}.

\begin{lemma}\label{lemma-v-u-c2}
Let $\omega\in(\omega\sb 1,m)$.
If in \eqref{zero-is-phi1} one has
$f\in C^1(\R\setminus\{0\})\cap C(\R)$
which satisfies
\eqref{ass-fp},
and if $V,\,U\in C^1(\R)$, with $V$ even and $U$ odd,
then $V,\,U\in C^2(\R)$,
and $H(t)=U(t)/t$ could be extended to a function
$H\in C^1(\R)$.
\end{lemma}

\begin{proof}
First we consider the case $f\in C^1(\R)$.
We proceed similarly to Lemma~\ref{lemma-v-u-c1}.
The inclusion $V\in C^2(\R)$ immediately follows
from the second equation in \eqref{zero-is-phi1}.
Let us prove that $U\in C^2(\R)$.
Equation \eqref{zero-is-phi1}
takes the form
\eqref{up-u-b}
with
\[
B(t)=\frac{f}{\epsilon^2}V(t)-\frac{1}{m+\omega}V(t),
\qquad
f=f\big(\epsilon^{\frac 2 k}V(t)^2-\epsilon^{2+\frac 2 k}U(t)^2\big).
\]
We note that now
$B\in C^1(\R)$
and is even.
It is enough to prove that $H(t)=U(t)/t$
could be extended to a $C^1$ function on $\R$.
Since $H(t)$ is even, it is enough to prove that
$\lim\sb{t\to 0}H'(t)=0$.
Taking the derivative of \eqref{def-h-t} at $t>0$,
we arrive at
\[
H'(t)
=
\frac{B(t)t^{n}-n\int_0^t B(\tau)\tau^{n-1}\,d\tau}{t^{n+1}}
=
\frac{\int_0^t B'(\tau)\tau^{n}\,d\tau}{t^{n+1}},
\]
therefore
\[
\lim\sb{t\to 0}
H'(t)
=
\lim\sb{t\to 0}
\frac{\int_0^t B'(\tau)\tau^{n}\,d\tau}{t^{n+1}}
=
\lim\sb{t\to 0}
\frac{B'(t)}{n+1}
=\frac{B'(0)}{n+1}=0,
\]
where we took into account that $B\in C^1(\R)$ is even.

The above argument still applies
if we only require that $f\in C^1(\R\setminus\{0\})\cap C(\R)$:
due to Theorem~\ref{theorem-solitary-waves},
the argument of $f$,
given by $\tau(t)=\epsilon^{\frac 2 k}V(t)^2-\epsilon^{2+\frac 2 k}U(t)^2$,
always belongs to $\R\sb{+}=(0,+\infty)$,
hence in \eqref{zero-is-phi1}
one has $f(\tau(t))$ which is a $C^1$ function of $t\in\R$.
Moreover, one can deduce from
\eqref{ass-fp}
that $\epsilon^{-2}V(t)\frac{d}{dt}f(\tau(t))$
remains bounded
pointwise by $C\hat V(t)^k(\abs{V'(t)}+\abs{U'(t)})$
uniformly
in $t\in\R$ and in $\epsilon\in(0,\epsilon_1)$:
\[
\Abs{
\epsilon^{-2}V(t)
\frac{d}{dt}
f(\tau(t))
}
=
\abs{
\epsilon^{-2}V(t)
f'(\tau(t))
(2\epsilon^{\frac 2 k}V(t)V'(t)
-2\epsilon^{2+\frac 2 k}U(t)U'(t))
}
\]
\[
\le
\epsilon^{-2}
\abs{V(t)}
\frac{
k\abs{\tau}^k
+C\abs{\tau}^K
}
{\abs{\tau}}
\abs{2\epsilon^{\frac 2 k}V(t)V'(t)
-2\epsilon^{2+\frac 2 k}U(t)U'(t)}
\]
\[
\le
\frac{C}{\epsilon^2}
(k\abs{\tau}^k+C\abs{\tau}^K)
(\abs{V'}+\abs{U'})
\le
C
\hat V(t)^k
(\abs{V'(t)}+\abs{U'(t)}).
\]
Above, we used \eqref{ass-fp}
to deal with $f'$
(note that $\tau>0$ by
Theorem~\ref{theorem-solitary-waves}~\itref{theorem-solitary-waves-ii}
and~\itref{theorem-solitary-waves-iii}),
and then
Theorem~\ref{theorem-solitary-waves}~\itref{theorem-solitary-waves-iii}
to estimate $\abs{V(t)}$ and $\abs{U(t)}$ with the aid of $\hat V(t)$.
So, we again have $B\in C^1(\R)$
and proceed as in the first part of the argument.
\end{proof}

Now we can show that $\phi\sb\omega\in H^2(\R^n,\C^N)$
for $\omega\in(\omega\sb 1,m)$.
From the Ansatz \eqref{sol-forms},
taking into account that
$H(t)=U(t)/t$ belongs to $C^1(\R)$
(as we proved in Lemma~\ref{lemma-v-u-c2}),
we conclude that
$\phi\sb\omega\in C^1(\R^n,\C^N)$.
Therefore,
the nonlinear term $f(\phi\sb\omega\sp\ast\beta\phi\sb\omega)\beta\phi\sb\omega$
is in $C^1(\R^n,\C^N)$ as a function of $x\in\R^n$,
and one has:
\begin{eqnarray}\label{nabla-f}
\abs{\nabla
\left(f(\phi\sb\omega\sp\ast\beta\phi\sb\omega)
\beta\phi\sb\omega\right)
}
&\le&
\abs{f'(\phi\sb\omega\sp\ast\beta\phi\sb\omega)}
\abs{\Re(\phi\sb\omega\sp\ast\beta\nabla\phi\sb\omega)}
\abs{\phi\sb\omega}
+
\abs{f(\phi\sb\omega\sp\ast\beta\phi\sb\omega)}
\abs{\nabla\phi\sb\omega}
\nonumber
\\
&\le&
C
\left(
\abs{f'(\phi\sb\omega\sp\ast\beta\phi\sb\omega)}
\abs{\phi\sb\omega}^2
+
\abs{f(\phi\sb\omega\sp\ast\beta\phi\sb\omega)}
\right)\abs{\nabla\phi\sb\omega}.
\end{eqnarray}
By Theorem~\ref{theorem-solitary-waves},
$\phi\sb\omega\in L^\infty(\R^n,\C^N)$
and
$\nabla \phi\sb \omega \in L^2(\R^n,\C^N)$;
using the bounds \eqref{ass-f}, \eqref{ass-fp},
we conclude that the right-hand side of \eqref{nabla-f}
is in $L^2(\R^n)$.
Then \eqref{nabla-f} shows that
$f(\phi\sb\omega\sp\ast\beta\phi\sb\omega)\beta\phi\sb\omega$
is in $H^1(\R^n,\C^N)$,
and then
from
\[
 \phi\sb\omega=-(D_m-\omega)^{-1}f(\phi\sb\omega\sp\ast\beta\phi\sb\omega)\beta\phi\sb\omega,
\]
with some $\omega\in(\omega\sb 1,m)$,
we deduce the inclusion $\phi\sb \omega \in H^2(\R^n,\C^N)$.

\subsection{Uniqueness, continuity,
and differentiability
of the mapping $\omega\mapsto\phi\sb\omega$}
\label{sect-diff}


We start with the following technical result.
Recall that $\Lambda\sb k<\infty$
was defined in \eqref{def-Lambda}.

\begin{lemma}\label{lemma-g-w-small}
There is $C<\infty$ such that
for all $\epsilon\in(0,\epsilon_1)$
and for all numbers
\[
\hat V,\,\hat U,\,\tilde V,\,\tilde U\in[-\Lambda\sb k,\Lambda\sb k],
\qquad
V=\hat V+\tilde V,
\qquad
U=\hat U+\tilde U
\]
which satisfy
\begin{eqnarray}\label{zz-f-p-mu-tau-0}
\epsilon_1\abs{U}
\le
\frac{1}{2}V,
\qquad
\abs{\tilde V}\le b_2\epsilon^{2\varkappa} \hat V,
\end{eqnarray}
one has
\[
\Norm{\frac{\p G(\epsilon,\tilde V,\tilde U)}
{\p(\tilde V,\tilde U)}}\sb{\End(\C^2)}
\le C\epsilon^{2\varkappa},
\]
where
$G(\epsilon,\tilde V,\tilde U)
=\begin{bmatrix}G_1(\epsilon,\tilde V,\tilde U)\\G_2(\epsilon,\tilde V,\tilde U)\end{bmatrix}
$
(cf. \eqref{def-g1}, \eqref{def-g2}).
\end{lemma}

Above, $\epsilon_1>0$ is from
Theorem~\ref{theorem-solitary-waves}~\itref{theorem-solitary-waves-ii}
and
$b_2<\infty$
is from
Theorem~\ref{theorem-solitary-waves}~\itref{theorem-solitary-waves-v}.

\begin{proof}
Denote
$V=\hat V+\tilde V$ and $U=\hat U+\tilde U$.
Let us consider
\begin{eqnarray}
\nonumber
\frac{\p G(\epsilon,\tilde V,\tilde U)}
{\p(\tilde V,\tilde U)}
&=&
\begin{bmatrix}
\p\sb{\tilde V}G_1&\p\sb{\tilde U}G_1
\\
\p\sb{\tilde V}G_2&\p\sb{\tilde U}G_2
\end{bmatrix}
\\
\nonumber
&=&
\begin{bmatrix}
-2f'\epsilon^{\frac 2 k-2}V^2
-\epsilon^{-2}f
+(1+2k)\hat V^{2k}
&
2f'\epsilon^{\frac 2 k}V U
\\
2f'\epsilon^{\frac 2 k}V U
&
f-2f'\epsilon^{2+\frac 2 k}U^2
\end{bmatrix}.
\end{eqnarray}
Above,
$f$ and $f'$
are evaluated at
$\tau=\epsilon^{2/k}(V^2-\epsilon^2 U^2)$.
All the terms except for $\p\sb{\tilde V}G_1$
are immediately $O(\epsilon^2)$;
we now focus on
$
\p\sb{\tilde V}G_1
$.
Denoting
$\tau=\epsilon^{2/k}(V^2-\epsilon^2 U^2)
=O(\epsilon^{2/k})$,
one has:
\[
\abs{\epsilon^{-2}f(\tau)-\hat V^{2k}}
\le
\epsilon^{-2}\abs{
f(\tau)-\tau^k
}
+\abs{(V^2-\epsilon^2 U^2)^k-V^{2k}}
+
\abs{V^{2k}-\hat V^{2k}}
\le
C\epsilon^{2\varkappa}.
\]
We estimated the three terms
in the middle using
\eqref{ass-f} and \eqref{zz-f-p-mu-tau-0}.
Similarly,
\begin{eqnarray}
\nonumber
&&
\abs{
f'(\tau)
\epsilon^{\frac 2 k-2}2V^2-2k\hat V^{2k}}
\\
\nonumber
&&
\le
\frac{2V^2\epsilon^{\frac 2 k}}{\epsilon^2}
\abs{f'(\tau)-k \tau^{k-1}}
+
\frac{2k V^2\epsilon^{\frac 2 k}}{\epsilon^2}
\abs{\tau^{k-1}-(\epsilon^{\frac 2 k}V^2)^{k-1}}
+
2k\abs{V^{2k}-\hat V^{2k}}
\le C\epsilon^{2\varkappa};
\end{eqnarray}
we used \eqref{ass-fp} and \eqref{zz-f-p-mu-tau-0}.
So,
\begin{eqnarray}
\nonumber
\abs{
\p\sb{\tilde V}G_1}
&=&
\abs{-2f'\epsilon^{\frac 2 k-2}V^2
-\epsilon^{-2}f
+(1+2k)\hat V^{2k}}
\\
\nonumber
&\le&
\abs{\epsilon^{-2}f-\hat V^{2k}}
+
\abs{2f'\epsilon^{\frac 2 k-2}V^2
-2k\hat V^{2k}}
=O(\epsilon^{2\varkappa})
.
\qedhere
\end{eqnarray}
\end{proof}

We claim that the mapping
\[
\mu:\;X\sb{e,o}\to X^1\sb{e,o}\subset X\sb{e,o},
\qquad
\mu:\;
\tilde W\mapsto A(\epsilon)^{-1}G(\epsilon,\tilde W)
\]
is a contraction
when considered on a certain subset of a sufficiently small ball.

\begin{lemma}\label{lemma-contraction}
Let $f\in C^1(\R\setminus\{0\})\cap C(\R)$
satisfy
\eqref{ass-fp}.
Then there is $\epsilon_2\in(0,\epsilon_1)$
such that
for any
$\epsilon\in(0,\epsilon_2)$
and any
\[
\tilde W\sb 0
=\begin{bmatrix}\tilde V_0\\\tilde U_0\end{bmatrix}
\in\overline{\mathbb{B}\sb\rho(X\sb{e,o})},
\qquad
\tilde W\sb 1
=\begin{bmatrix}\tilde V_1\\\tilde U_1\end{bmatrix}
\in\overline{\mathbb{B}\sb\rho(X\sb{e,o})},
\qquad
\mbox{with}
\quad
\rho= b_1\epsilon^{2\varkappa},
\]
with $ b_1>0$ from Lemma~\ref{lemma-mu-into-better},
which satisfy
\begin{eqnarray}\label{f-p-mu}
\epsilon_1\abs{\hat U(t)+\tilde U_s(t)}
\le \frac 1 2 (\hat V(t)+\tilde V_s(t)),
\qquad
\forall t\in\R,
\quad
\forall s=0,\,1,
\end{eqnarray}
\begin{eqnarray}\label{f-p-mu-2}
\abs{\tilde V_s(t)}\le b_2 \epsilon_2^{2\varkappa}\hat V(t),
\qquad \forall t\in\R,\quad\forall s=0,\,1,
\end{eqnarray}
one has
\[
\norm{\mu(\epsilon,\tilde W\sb 1)-\mu(\epsilon,\tilde W\sb 0)}\sb{X^1}
\le
\frac 1 2
\norm{\tilde W\sb 1-\tilde W\sb 0}\sb{X}.
\]
\end{lemma}

Above,
$b_2<\infty$ is from 
Lemma~\ref{lemma-g-w-small}.
We point out that, by Theorem~\ref{theorem-solitary-waves}~\itref{theorem-solitary-waves-ii},
the fixed points of $\mu(\epsilon,\cdot)$
satisfy \eqref{f-p-mu},
and by
Theorem~\ref{theorem-solitary-waves}~\itref{theorem-solitary-waves-iii}
these points also satisfy \eqref{f-p-mu-2}.

\begin{proof}
We consider the linear interpolations
\begin{eqnarray}\label{def-linear-interpolations}
\tilde V_s(t)
=
(1-s)\tilde V_0(t)+s\tilde V_1(t),
\qquad
\tilde U_s(t)
=
(1-s)\tilde U_0(t)+s\tilde U_1(t),
\\[1ex]
\nonumber
s\in[0,1],
\end{eqnarray}
and we also set
\begin{eqnarray}\label{def-linear-interpolations-1}
V_s(t)=\hat V(t)+\tilde V_s(t),
\qquad
U_s(t)=\hat U(t)+\tilde U_s(t).
\end{eqnarray}
We notice that, due to
\eqref{f-p-mu} and \eqref{f-p-mu-2},
these interpolations
are such that
$\begin{bmatrix}V_s\\U_s\end{bmatrix}\in X\sb{e,o}
$,
for all $s\in[0,1]$,
and they also satisfy the equivalents of \eqref{f-p-mu}
and \eqref{f-p-mu-2}:
\[
\epsilon_1\abs{U_s(t)}
\le
\frac{1}{2}V_s(t),
\qquad
\forall t\in\R,
\quad
\forall s\in[0,1],
\]
\[
\abs{\tilde V_s(t)}\le b_2\epsilon_2^{2\varkappa} \hat V(t),
\qquad \forall t\in\R,\quad\forall s\in[0,1].
\]
Let us pick $\epsilon\in(0,\epsilon_1)$
and consider the relation
\begin{eqnarray}\label{mu-mu}
\mu(\epsilon,\tilde W_1)-\mu(\epsilon,\tilde W_0)=
A(\epsilon)^{-1}
\big(G(\epsilon,\tilde W_1)-G(\epsilon,\tilde W_0)\big).
\end{eqnarray}
To estimate the right-hand side, we consider
\begin{eqnarray}\label{g1g1}
&&
\qquad
G_1(\epsilon,\tilde W_1)-G_1(\epsilon,\tilde W_0)
\\
\nonumber
&&
=
-\epsilon^{-2}\left(f(\epsilon^{\frac 2 k}(V_1\sp 2- \epsilon\sp 2 U_1^2))V_1-f(\epsilon^{\frac 2 k}(V_0\sp 2- \epsilon\sp 2 U_0^2))V_0\right)
+(1+2k)\hat V\sp{2k}(\tilde{V}_1-\tilde{V}_0),
\end{eqnarray}
\[
G_2(\epsilon,\tilde W_1)-G_2(\epsilon,\tilde W_0)
=
f\big(\epsilon^{\frac 2 k}(V_1\sp 2- \epsilon\sp 2 U_1^2)\big)U_1
-f\big(\epsilon^{\frac 2 k}(V_0\sp 2- \epsilon\sp 2 U_0^2)\big)U_0.
\]
For \eqref{g1g1}, we have:
\begin{eqnarray}\label{g1-g1}
G(\epsilon,\tilde W_1)-G(\epsilon,\tilde W_0)
&=&\int_0^1\,ds\,\frac{d}{d s}
G\big(\epsilon,(1-s)\tilde W_0+s\tilde W_1\big)
\nonumber
\\
&=&
(\tilde W_1-\tilde W_0)
\int_0^1
\p\sb{\tilde W}G\big(\epsilon,(1-s)\tilde W_0+s\tilde W_1\big)
\,ds.
\end{eqnarray}
Applying Lemma~\ref{lemma-g-w-small}
to \eqref{g1-g1},
we have:
\[
\norm{G(\epsilon,\tilde W_1)-G(\epsilon,\tilde W_0)}\sb{X}
\le
C\epsilon^{2\varkappa}\norm{\tilde W_1-\tilde W_0}\sb{X},
\]
with some $C<\infty$.
We take $\epsilon_2\in(0,\epsilon_1)$ so small that
\begin{eqnarray}\label{half-half}
C \epsilon_2^{2\varkappa}
\sup\sb{\epsilon\in[0,\epsilon_1]}
\norm{A(\epsilon)^{-1}}\sb{X\sb{e,o}\to X^1\sb{e,o}}
\le 1/2;
\end{eqnarray}
then the lemma follows
from applying \eqref{half-half} to \eqref{mu-mu}.
\end{proof}

For each $\epsilon\in(0,\epsilon_2)$,
Lemma~\ref{lemma-contraction}
proves the uniqueness of the fixed point of $\mu(\epsilon,\cdot)$
in $X\sb{e,o}$
which satisfies
\[
\tilde W=A(\epsilon)^{-1}G(\epsilon,\tilde W),
\qquad
\tilde W\in\overline{\mathbb{B}\sb\rho(X\sb{e,o})},
\qquad
\mbox{where}
\quad
\rho= b_1\epsilon^{2\varkappa};
\]
this is the fixed point $\tilde W$
which we constructed in Theorem~\ref{theorem-solitary-waves}.
Thus, we have a well-defined map
\begin{eqnarray}\label{epsilon-to-w}
&&
(0,\epsilon_2)\to \mathbb{B}\sb\rho(X\sb{e,o}^1),
\qquad
\rho= b_1\epsilon_2^{2\varkappa};
\nonumber
\\[1ex]
&&
\epsilon\mapsto \tilde W(t,\epsilon),
\qquad
\norm{e^{\gamma\langle t\rangle}\tilde W(\cdot,\epsilon)}\sb{H^1(\R,\C^2)}
\le  b_1\epsilon^{2\varkappa}.
\end{eqnarray}
The above argument
also implies the continuity of the fixed point
$\tilde W(\epsilon)$ as a function of $\epsilon$,
since for any $\epsilon,\epsilon'\in(0,\epsilon_2)$ one has
\begin{eqnarray}
\nonumber
&&
\tilde W(\epsilon')-\tilde W(\epsilon)
\\
\nonumber
&&
=
A(\epsilon')^{-1}
\big(
G(\epsilon',\tilde W(\epsilon'))-G(\epsilon',\tilde W(\epsilon))
\big)
+
A(\epsilon')^{-1}G(\epsilon',\tilde W(\epsilon))
-
A(\epsilon)^{-1}G(\epsilon,\tilde W(\epsilon)).
\end{eqnarray}
We evaluate $X$-norm
of the above relation,
applying Lemma~\ref{lemma-contraction}
to the first term in the right-hand side;
this yields
\[
 \|\tilde W(\epsilon')-\tilde W(\epsilon)\|_X
\leq 2
\big\|
A(\epsilon')^{-1}
G(\epsilon',\tilde W(\epsilon))
-
A(\epsilon)^{-1}
G(\epsilon,\tilde W(\epsilon))
\big\|_X,
\qquad
\forall\epsilon,\,\epsilon'\in(0,\epsilon_2).
\]
Due to the continuous dependence
of $A$ and $G$ on $\epsilon>0$,
the above relation proves the continuity
of the map \eqref{epsilon-to-w}
in $\epsilon\in(0,\epsilon_2)$.

We now turn to the differentiability of $\tilde W$ with respect to $\epsilon$.
Let us take $\alpha,\,\beta\in(0,\epsilon_2)$
(with $\epsilon_2>0$ from Lemma~\ref{lemma-contraction}).
Without loss of generality, we may assume that
$\alpha<\beta$.
For both $\alpha$ and $\beta$,
we denote the unique fixed points
of $\mu(\alpha,\cdot)$ and $\mu(\beta,\cdot)$
(the images of $\alpha,\,\beta\in(0,\epsilon_2)$
under the mapping \eqref{epsilon-to-w})
by
$\tilde W(t,\alpha)
=\begin{bmatrix}\tilde V(t,\alpha)\\\tilde U(t,\alpha)\end{bmatrix}$
and
$\tilde W(t,\beta)
=\begin{bmatrix}\tilde V(t,\beta)\\\tilde U(t,\beta)\end{bmatrix}$.
By Theorem~\ref{theorem-solitary-waves}~\itref{theorem-solitary-waves-ii}
and~\itref{theorem-solitary-waves-iii},
these fixed points satisfy
\[
\epsilon_1\abs{U(t,\alpha)}
\le
\frac{1}{2}V(t,\alpha),
\qquad
\abs{\tilde V(t,\alpha)}\le b_2\alpha^{2\varkappa} \hat V(t),
\qquad
\forall t\in\R,
\]
\[
\epsilon_1\abs{U(t,\beta)}
\le
\frac{1}{2}V(t,\beta),
\qquad
\abs{\tilde V(t,\beta)}\le b_2\beta^{2\varkappa} \hat V(t),
\qquad
\forall t\in\R,
\]
therefore the linear interpolation
\[
\tilde W_s(t)
=\begin{bmatrix}\tilde V_s(t)\\\tilde U_s(t)\end{bmatrix}
=(1-s)\tilde W(t,\alpha)+s\tilde W(t,\beta),
\qquad
s\in[0,1],
\]
satisfies
\begin{eqnarray}\label{line-3}
\epsilon_1\abs{U_s(t)}
\le
\frac{1}{2}V_s(t),
\qquad
\abs{\tilde V_s(t)}\le b_2\beta^{2\varkappa} \hat V(t),
\qquad
\forall t\in\R,
\quad
\forall s\in[0,1],
\end{eqnarray}
where
$V_s(t)=\hat V(t)+\tilde V_s(t)$
and
$U_s(t)=\hat U(t)+\tilde U_s(t)$
(cf. \eqref{def-linear-interpolations-1});
in the last inequality in \eqref{line-3}, we took
into account that $\alpha<\beta$.
We have:
\begin{align*}
&
\frac{\tilde W(\beta)-\tilde W(\alpha)}{\beta-\alpha}
=
\frac{\mu(\beta,\tilde W(\beta))-\mu(\beta,\tilde W(\alpha))}{\beta-\alpha}
+
\frac{\mu(\beta,\tilde W(\alpha)
-\mu(\alpha,\tilde W(\alpha))}{\beta-\alpha}
\\[1ex]
&=
A(\beta)^{-1}
\left(
 \int_0^1
\p\sb{\tilde W}G\big(
\beta,(1-s)\tilde W(\alpha)+s\tilde W(\beta)\big)\,ds
\right)
\,\frac{\tilde W(\beta)-\tilde W(\alpha)}{\beta-\alpha}
\\[1ex]
&
\qquad\qquad\qquad\qquad
+\frac{\mu(\beta,\tilde W(\alpha))
-\mu(\alpha,\tilde W(\alpha))}{\beta-\alpha}.
\end{align*}
The above relation takes place at each $t\in\R$;
we omitted the dependence on $t$.
By Lemma~\ref{lemma-g-w-small},
which is applicable due to \eqref{line-3},
we can choose $\epsilon_2\in(0,\epsilon_1)$
smaller if necessary
so that the operator
$B(t,\alpha,\beta)\in\End(\C^2)$
defined by
\[
B(t,\alpha,\beta)
=I_{\C^2}-A(\beta)^{-1}\int_0^1
\p\sb{\tilde W}G\big(\beta,(1-s)\tilde W(t,\alpha)+s\tilde W(t,\beta)\big)\,ds
\]
is invertible,
with the inverse
bounded uniformly in
$t\in\R$ and $\alpha,\,\beta\in(0,\epsilon_2)$;
we then have:
\[
\frac{\tilde W(\beta)-\tilde W(\alpha)}{\beta-\alpha}
=B(\alpha,\beta)^{-1}
\frac{
\mu(\beta,\tilde W(\alpha))-\mu(\alpha,\tilde W(\alpha))}{\beta-\alpha}.
\]
Since $B$ is continuous in $\alpha$ and $\beta$ while 
$\mu(\epsilon,\tilde W)=A(\epsilon)^{-1}G(\epsilon,\tilde W)$,
with both
$A(\epsilon)^{-1}$ and $G(\epsilon,\tilde W)$
differentiable in $\epsilon$,
we deduce that
$\big(\tilde W(t,\beta)-\tilde W(t,\alpha)\big)/(\beta-\alpha)$
has a limit as
$\beta\to\alpha$;
setting $\alpha=\epsilon$, we have:
\begin{eqnarray}\label{pw-a-pw-0}
\p\sb\epsilon\tilde W
=B^{-1}
\frac{\p}{\p\epsilon}
\big(A^{-1}G(\epsilon,\tilde W)\big)
&=&
B^{-1}
 A^{-1}
 \big(
 -\p\sb\epsilon A\,A^{-1}
 G(\epsilon,\tilde W)
 +
 \p\sb\epsilon
 G(\epsilon,\tilde W)
 \big)
\nonumber
\\
&=&
B^{-1}
A^{-1}
\big(
-\p\sb\epsilon A
\,\tilde W
+
\p\sb\epsilon
G(\epsilon,\tilde W)
\big),
\end{eqnarray}
where
$\tilde W=\tilde W(t,\epsilon)$,
\begin{eqnarray}\label{def-b-0}
A=A(\epsilon),
\qquad
B=B(t,\epsilon):=B(t,\epsilon,\epsilon)
=I_{\C^2}-A(\epsilon)^{-1}
\p\sb{\tilde W}G\big(\epsilon,\tilde W(t,\epsilon)\big).
\end{eqnarray}
In the last equality in \eqref{pw-a-pw-0},
we took into account
that $\tilde W(t,\epsilon)=A(\epsilon)^{-1}G(\epsilon,\tilde W)$
(cf. \eqref{w-a-a-3}).

\begin{lemma}\label{lemma-g-epsilon-small}
One has:
\[
\left\|
e^{\gamma\langle t\rangle}
\frac{\p G}{\p\epsilon}
\big(\epsilon,\tilde W(t,\epsilon)\big)
\right\|\sb{X}
=
O(\epsilon^{2\varkappa-1}),
\qquad
\epsilon\in(0,\epsilon_2),
\]
\[
\left\|
e^{\gamma\langle t\rangle}
\left(
\frac{\p G}{\p\epsilon}
\big(\epsilon,\tilde W(t,\epsilon)\big)
-
\epsilon
\begin{bmatrix}
2k\hat U^2\hat V^{2k-1}
+\frac{\hat V}{4m^3}
\\
2\hat U\hat V^{2k}
+\frac{\hat U}{m}
\end{bmatrix}
\right)
\right\|\sb{X}
=
O\big(\epsilon^{\frac{2K}{k}-3}\big)
+o(\epsilon),
\quad
\epsilon\in(0,\epsilon_2).
\]
\end{lemma}

\begin{proof}
Since
$2\varkappa-1\le 1$
and due to the exponential decay of $\hat V$ and $\hat U$
(cf. Lemma~\ref{lemma-bl}),
the first estimate stated in the lemma
follows from the second one.
By \eqref{def-g1} and \eqref{def-g2},
$\p\sb\epsilon G$ is given by
\begin{eqnarray}\label{g-epsilon}
&&
\quad \frac{\p G(\epsilon,\tilde W)}{\p\epsilon}
\\
\nonumber
&&
\quad
=
\begin{bmatrix}
\left(
2\epsilon^{-3}f
-\epsilon^{-2}
\frac{2}{k}\epsilon^{\frac 2 k-1}
(V^2-\epsilon^2U^2)f'
+
2\epsilon^{-2}
\epsilon^{1+\frac 2 k}U^2 f'
\right)V
+\frac{\hat V}{(m+\omega)^2}\frac{\epsilon}{\omega}
\\
(\frac{2}{k}\epsilon^{\frac 2 k-1}V^2
-\frac{2+2k}{k}\epsilon^{1+\frac 2 k}U^2)U f'
+\hat U\frac{\epsilon}{\omega}
\end{bmatrix},
\end{eqnarray}
with $f$, $f'$
evaluated at
$\tau=\epsilon^{2/k}(V^2-\epsilon^2 U^2)$.
We recall that
$\tilde W=
\begin{bmatrix}\tilde V\\\tilde U\end{bmatrix}$,
$V=\hat V+\tilde V$,
$U=\hat U+\tilde U$;
cf. \eqref{Vhatdef}, \eqref{def-V-U-hat}.
By \eqref{ass-fp},
taking into account
the exponential decay of $\hat V$ and $\hat U$,
and also
$\norm{e^{\gamma\langle t\rangle}
\tilde W}\sb{H^1(\R,\C^2)}=O(\epsilon^{2\varkappa})$
(cf. Theorem~\ref{theorem-solitary-waves}~\itref{theorem-solitary-waves-v}),
one has:
\[
\norm{
e^{\gamma\langle t\rangle}
\big(f(\tau)-\frac{\tau}{k}f'(\tau)\big)}_{X}
=\norm{
e^{\gamma\langle t\rangle}
O(\abs{\tau}^K)}_{X}
=O(\epsilon^{2K/k}),
\]
\[
\norm{
e^{\gamma\langle t\rangle}
\epsilon^{2/k}U^2 f'(\tau)}_{X}
\le
C\norm{
e^{\gamma\langle t\rangle}
\epsilon^{2/k}V^2 f'(\tau)}_{X}
\]
\[
\le
C\norm{
e^{\gamma\langle t\rangle}
\tau f'(\tau)}_{X}
=\norm{O(\abs{\tau}^{k})}_{X}
=O(\epsilon^{2}),
\]
where $\tau(t)=\epsilon^{2/k}(V(t)^2-\epsilon^2 U(t)^2)$.
Applying these estimates to terms in \eqref{g-epsilon},
one arrives at the second estimate
stated in the lemma.
\end{proof}

Multiplying \eqref{pw-a-pw-0} by
$e^{\gamma\langle t\rangle}$, we have:
\begin{eqnarray}\label{pw-a-pw}
\qquad
e^{\gamma\langle t\rangle}
\p\sb\epsilon\tilde W
=
B^{-1}
\circ
e^{\gamma\langle t\rangle}
\circ
A^{-1}
\circ
e^{-\gamma\langle t\rangle}
\circ
\big(
\p\sb\epsilon A
\circ
e^{\gamma\langle t\rangle}
\circ
\tilde W
+
e^{\gamma\langle t\rangle}
\p\sb\epsilon
G(\epsilon,\tilde W)
\big).
\end{eqnarray}
Above,
$e^{\pm\gamma\langle t\rangle}$
are understood as the multiplication operators;
we note that they commute with
\[
\p\sb\epsilon A(\epsilon)
=
\frac\epsilon\omega
\begin{bmatrix}
-\frac{1}{(m+\omega)^2}&0\\0&-1
\end{bmatrix}.
\]
The operator $B(t,\epsilon)$
(cf. \eqref{def-b-0})
defines a mapping
\begin{eqnarray}\label{def-b}
B(t,\epsilon)^{-1}:\;
X^1\to X^1
\end{eqnarray}
which is continuous
since both
$\norm{B(t,\epsilon)^{-1}}\sb{\End(\C^2)}$
and
$\norm{\p\sb t B(t,\epsilon)}\sb{\End(\C^2)}$
are bounded uniformly in $t\in\R$
and $\epsilon\in(0,\epsilon_2)$,
as long as $\epsilon_2>0$ is sufficiently small;
we took into account that
$\norm{\p\sb{\tilde W}G(\epsilon,\tilde W)}\sb{\End(\C^2)}
=O(\epsilon^{2\varkappa})$ by Lemma~\ref{lemma-g-w-small},
while the derivatives
$\p\sb t V(t,\epsilon)$ and $\p\sb t U(t,\epsilon)$
are bounded pointwise,
uniformly in $t\in\R$ and $\epsilon\in(0,\epsilon_2)$,
due to Lemma~\ref{lemma-v-u-c1},
and hence so is $\p\sb t\tilde W(t,\epsilon)$.

Since
$\norm{e^{\gamma\langle t\rangle}\tilde W}\sb{X^1}
=O(\epsilon^{2\varkappa})$
(cf. Lemma~\ref{lemma-mu-into-better})
and
the mapping
$e^{\gamma\langle t\rangle}\circ
A(\epsilon)^{-1}\circ e^{-\gamma\langle t\rangle}:\;X\to X^1$
is continuous
(just like the mapping
$e^{(1+2k)\gamma\langle t\rangle}\circ
A(\epsilon)^{-1}\circ e^{-(1+2k)\gamma\langle t\rangle}:\;X\to X^1$
in \eqref{ttm}),
while
\eqref{def-b} is continuous in $X^1$,
it follows that
the $X^1$-norm of the right-hand side of \eqref{pw-a-pw}
is bounded by
\[
C\left(
\epsilon
\norm{
e^{\gamma\langle t\rangle}
\tilde W(t,\epsilon)}\sb{X^1}
+
\norm{
e^{\gamma\langle t\rangle}
\p\sb\epsilon G(\epsilon,\tilde W(t,\epsilon))}\sb{X}
\right)
=O(\epsilon^{1+2\varkappa})+O(\epsilon^{2\varkappa-1})
=O(\epsilon^{2\varkappa-1}),
\]
for all $\epsilon\in(0,\epsilon_2)$;
we estimated the second term in the left-hand side
with the aid of Lemma~\ref{lemma-g-epsilon-small}.
Thus, the relation \eqref{pw-a-pw} gives
\begin{eqnarray}\label{p-epsilon-w-tilde}
\p\sb\epsilon \tilde W\in X^1\sb{e,o},
\qquad
\norm{e^{\gamma\langle t\rangle}\p\sb\epsilon\tilde W}\sb{X^1}
=
O(\epsilon^{2\varkappa-1}),
\qquad
\epsilon\in(0,\epsilon_2),
\end{eqnarray}
proving \eqref{norm-p-epsilon-w}.

We can now estimate
$\norm{\p\sb\omega\phi\sb\omega}\sb{L^2}^2$.
We have:
\begin{eqnarray}
\nonumber
\norm{\p\sb\omega\phi\sb\omega}\sb{L^2}^2
&=&
\frac{\epsilon^2}{\omega^2}
\Norm{\frac{d}{d\epsilon}\phi\sb\omega}\sb{L^2}^2
\\
\nonumber
&=&
\frac{\epsilon^2\mathop{\mathrm{vol}}(\mathbb{S}^{n-1})}{\omega^2}
\int\sb{0}\sp\infty
\left(
(\p\sb\epsilon(\epsilon^{\frac 1 k}V(\epsilon r,\epsilon)))^2
+
(\p\sb\epsilon(\epsilon^{1+\frac 1 k}U(\epsilon r,\epsilon)))^2
\right)
r^{n-1}\,dr.
\end{eqnarray}
Let us estimate the above integral.
Since
\[
\p\sb\epsilon(\epsilon^{\frac 1 k}V(\epsilon r,\epsilon))
=
\frac{1}{k}\epsilon^{\frac{1}{k}-1}
V(\epsilon r,\epsilon)
+
\epsilon^{\frac{1}{k}}r \p\sb t V(\epsilon r,\epsilon)
+\epsilon^{\frac 1 k}\p\sb\epsilon V(\epsilon r,\epsilon),
\]
we have:
\begin{eqnarray}
\nonumber
&&
\int\sb{0}\sp\infty
(\p\sb\epsilon(\epsilon^{\frac 1 k}V(\epsilon r,\epsilon)))^2
r^{n-1}\,dr
\\
\nonumber
&&=
\epsilon^{-n}
\int\sb{0}\sp\infty
\Big(
\frac{\epsilon^{\frac{1}{k}-1}}{k}
V(t,\epsilon)
+
\epsilon^{\frac{1}{k}-1}t \p\sb t V(t,\epsilon)
+\epsilon^{\frac 1 k}\p\sb\epsilon V(t,\epsilon)
\Big)^2
t^{n-1}\,dt
\\
\nonumber
&&=
\epsilon^{-n+\frac 2 k-2}
\int_0^\infty
\left(
\frac{V(t,\epsilon)}{k}
+
t \p\sb t V(t,\epsilon)
+\epsilon\p\sb\epsilon V(t,\epsilon)
\right)^2
t^{n-1}\,dt
\\
\nonumber
&&=
\epsilon^{-n+\frac 2 k-2}
\left(
C+O(\epsilon^{2\varkappa})
\right),
\end{eqnarray}
with
\[
C=\int_0^\infty
\Big(\frac{\hat V(t)}{k}+t\p\sb t\hat V(t)\Big)^2 t^{n-1}\,dt>0.
\]
We used
Theorem~\ref{theorem-solitary-waves}~\itref{theorem-solitary-waves-v}
for the $L^2$-norm of $t \p\sb t V(t,\epsilon)$
and
\eqref{norm-p-epsilon-w}
for the $L^2$-norm of
$\p\sb\epsilon V(t,\epsilon)
=\p\sb\epsilon\tilde V(t,\epsilon)$.
We omit the computations for the part containing $U$
since its contribution will be of the order $O(\epsilon^2)$ smaller,
which is dominated by the $O(\epsilon^{2\varkappa})$ error term.
It follows that
\[
\norm{\p\sb\omega\phi\sb\omega}\sb{L^2}^2
=
\frac{\epsilon^2}{\omega^2}
\norm{\p\sb\epsilon\phi\sb\omega}\sb{L^2}^2
=
\epsilon^{-n+\frac 2 k}
\frac{\mathop{\mathrm{vol}}(\mathbb{S}^{n-1})}{\omega^2}
(C+O(\epsilon^{2\varkappa})),
\]
proving \eqref{d-phi-d-omega-bound}.

This completes the proof of Theorem~\ref{theorem-existence-c1}~\itref{theorem-existence-c1-i}.

\section{Vakhitov--Kolokolov condition
for the nonlinear Dirac equation}
\label{sect-critical}

Finally, let us prove
Theorem~\ref{theorem-existence-c1}~\itref{theorem-existence-c1-ii}.
We start with the focusing nonlinear Schr\"odinger equation in $n$ dimensions:
\begin{eqnarray}\label{nls-c}
\jj\dot\psi=-\frac{1}{2m}\Delta\psi-\abs{\psi}^{2k}\psi,
\qquad
\psi(t,x)\in\C,
\qquad
x\in\R^n.
\end{eqnarray}
Above, $k\in(0,n/(n-2))$ (any $k>0$ if $n\le 2$).
Given a positive solution $u_k$
to the stationary Schr\"odinger equation
\[
-\frac{1}{2m}u_k=-\frac{1}{2m}\Delta u_k-u_k^{1+2k}
\]
(cf. \eqref{def-uk}),
one can use $u_k$ to construct
the solitary wave solutions to \eqref{nls-c}
for any $\omega<0$:
\[
\varphi\sb\omega(x)
=(2m\abs{\omega})^{1/(2k)}u_k\big(\sqrt{2m\abs{\omega}}x\big).
\]
When $k=2/n$,
it follows that
the $L^2$-norm of $\varphi\sb\omega$ does not depend on $\omega$;
$\frac{d}{d\omega}\norm{\varphi\sb\omega}^2=0$.

We are going to show that
in the case of the nonlinear Dirac equation in $(n+1)$D
with the ``critical'' value $k=2/n$
(and absent or sufficiently small higher order terms),
the charge is no longer constant;
instead,
$\p\sb\omega Q(\phi\sb\omega)<0$ for $\omega\lesssim m$.
This reduces the degeneracy of the zero eigenvalue
of the linearization at the corresponding solitary wave;
see e.g. \cite{MR3311594}.

\begin{lemma}\label{lemma-hat-v-tilde-v-p}
Assume that
$f\in C^1(\R\setminus\{0\})\cap C(\R)$
satisfies
the assumption \eqref{ass-fp}
with some $K>k>0$.
One has:
\[
\langle \hat V,\p\sb\epsilon\tilde V\rangle
=
\epsilon q_1
+
\epsilon\Big(\frac{1}{k}-\frac{n}{2}\Big)q_2
+O(\epsilon^{\frac{2K}{k}-3}+\epsilon^{4\varkappa-1}
)
+o(\epsilon),
\qquad
\epsilon\in(0,\epsilon_2),
\]
with
\[
q_1
=
\int\sb{\R^n}(4m \hat V^{2k}\hat U^2+\hat U^2)\,dy>0,
\qquad
q_2
=
\int\sb{\R^n}
\Big(
\frac{\hat V^2}{4m^2}
+
2m \hat V^{2k}\hat U^2+\hat U^2
\Big)\,dy>0.
\]
\end{lemma}

\begin{proof}
By \eqref{weird},
$
\eurl\sb{+}
\big(\frac 1 k \hat V+x\cdot\nabla \hat V\big)
=-\frac{1}{m}\hat V$;
hence,
\[
A(0)\begin{bmatrix}
\frac{1}{k}\hat V+x\cdot\nabla\hat V
\\
-\frac{1}{2m}\p\sb r(\frac{1}{k}\hat V+x\cdot\nabla\hat V)
\end{bmatrix}
=
\frac{1}{m}
\begin{bmatrix}\hat V\\0\end{bmatrix}.
\]
Therefore,
\begin{eqnarray}
\frac{\langle \hat V,\p\sb\epsilon\tilde V\rangle}{m}
&=&
\left\langle
\frac 1 m
\begin{bmatrix}
\hat V\\ 0
\end{bmatrix},
\ \p\sb\epsilon\begin{bmatrix}
\tilde V\\\tilde U
\end{bmatrix}
\right\rangle
=
\left\langle
A(0)
\begin{bmatrix}
\frac{1}{k}\hat V+y\cdot \nabla\hat V
\\
-\frac{1}{2m}\p\sb r
(\frac{1}{k}\hat V+y\cdot \nabla\hat V)
\end{bmatrix},
\ \p\sb\epsilon\begin{bmatrix}
\tilde V\\\tilde U
\end{bmatrix}
\right\rangle
\nonumber
\\[1ex]
&=&
\left\langle
\begin{bmatrix}
\frac{1}{k}\hat V+y\cdot \nabla\hat V
\\
-\frac{1}{2m}\p\sb r
(\frac{1}{k}\hat V+y\cdot \nabla\hat V)
\end{bmatrix},
\ A(0)\p\sb\epsilon\begin{bmatrix}
\tilde V\\\tilde U
\end{bmatrix}
\right\rangle
\nonumber
\\[1ex]
&=&
\left\langle
\begin{bmatrix}
\frac{1}{k}\hat V+y\cdot \nabla\hat V
\\
-\frac{1}{2m}\p\sb r
(\frac{1}{k}\hat V+y\cdot \nabla\hat V)
\end{bmatrix},
\ A(\epsilon)\p\sb\epsilon\begin{bmatrix}
\tilde V\\\tilde U
\end{bmatrix}
\right\rangle
+O(\epsilon^2)\norm{\p\sb\epsilon \tilde W}_{L^\infty}.
\nonumber
\end{eqnarray}
We took into account that the operator $A(\epsilon)$
defined in \eqref{a-epsilon}
is self-adjoint
on $X^1\sb{e,o}$
and that
$\norm{A(\epsilon)-A(0)}\sb{L^\infty(\R,\End(\C^2))}=O(\epsilon^2)$.
Taking the derivative of~\eqref{w-a-a-3}
with respect to $\epsilon$,
we derive:
\begin{eqnarray}
&&
\frac{\langle \hat V,\p\sb\epsilon\tilde V\rangle}{m}
\nonumber
\\[1ex]
&&
\quad
=
\left\langle
\begin{bmatrix}
\frac{1}{k}\hat V+y\cdot \nabla\hat V
\\
-\frac{1}{2m}\p\sb r
(\frac{1}{k}\hat V+y\cdot \nabla\hat V)
\end{bmatrix},
\ \p\sb{\tilde W}G\p\sb\epsilon\tilde W
+
\p\sb\epsilon G-\p\sb\epsilon A(\epsilon)\tilde W
\right\rangle
+O\big(\epsilon^2\big)\norm{\p\sb\epsilon \tilde W}_{L^\infty}
\nonumber
\\[1ex]
&&
\quad
=
\left\langle
\begin{bmatrix}
\frac{1}{k}\hat V+y\cdot \nabla\hat V
\\
-\frac{1}{2m}\p\sb r
(\frac{1}{k}\hat V+y\cdot \nabla\hat V)
\end{bmatrix},
\ \p\sb\epsilon G
\right\rangle
+
O\big(
\epsilon^{4\varkappa-1}
+
\epsilon^{2\varkappa+1}
\big).
\nonumber
\end{eqnarray}
We used the estimates
$
\norm{\p\sb{\tilde W}G}\sb{L^\infty(\R,\End(\C^2))}
=O(\epsilon^{2\varkappa})
$
(cf. Lemma~\ref{lemma-g-w-small}),
$
\norm{\tilde W}_{L^\infty}
=O(\epsilon^{2\varkappa})
$
(cf. Theorem~\ref{theorem-solitary-waves}~\itref{theorem-solitary-waves-v}),
and
$
\norm{\p\sb\epsilon \tilde W}_{L^\infty}
=O(\epsilon^{2\varkappa-1})
$
(cf. Theorem~\ref{theorem-existence-c1}~\itref{theorem-existence-c1-i}).

Taking into account Lemma~\ref{lemma-g-epsilon-small}
to express $\p\sb\epsilon G(\epsilon,\tilde W)$,
we continue:
\begin{eqnarray}
&&
\frac{\langle\hat V,\p\sb\epsilon\tilde V\rangle}{m}
\nonumber
\\[1ex]
&&
=
\epsilon
\left\langle
\begin{bmatrix}
\frac{1}{k}\hat V+y\cdot \nabla\hat V
\\
-\frac{1}{2m}\p\sb r
(\frac{1}{k}\hat V+y\cdot \nabla\hat V)
\end{bmatrix},
\ \begin{bmatrix}
2k\hat U^2\hat V^{2k-1}
+\frac{\hat V}{4m^3}
\\
2\hat U\hat V^{2k}+\frac{\hat U}{m}
\end{bmatrix}
\right\rangle
+O\big(\epsilon^{\frac{2K}{k}-3}
+\epsilon^{4\varkappa-1}
\big)
+o(\epsilon)
\nonumber
\\[1ex]
&&
=
\epsilon
\int\limits\sb{\R^n}
\Big[
\Big(\frac{1}{k}\hat V+y\!\cdot\!\nabla\hat V\Big)
\Big(
2k\hat V^{2k-1}\hat U^2
+
\frac{\hat V}{4m^3}
\Big)
+
\Big(
\frac{1+k}{k}\hat U+y\!\cdot\!\nabla\hat U
\Big)
\Big(2\hat V^{2k}\hat U+\frac{\hat U}{m}
\Big)
\Big]
\,dy
\nonumber
\\[1ex]
&&
+O\big(\epsilon^{\frac{2K}{k}-3}
+\epsilon^{4\varkappa-1}\big)
+o(\epsilon);
\nonumber
\end{eqnarray}
we took into account that
$
-
\frac{1}{2m}
\p\sb r
\Big(
\frac 1 k\hat V+y\!\cdot\!\nabla\hat V
\Big)
=
\frac 1 k\hat U+y\!\cdot\!\nabla\hat U+\hat U.
$
The integral
$\int\sb{\R^n}[\dots]\,dy$
is evaluated by parts
as follows:
\begin{eqnarray}
&&
\int\limits\sb{\R^n}
\Big[
\frac{\hat V^2}{4m^3 k}
+2\hat V^{2k}\hat U^2
+\frac{y\!\cdot\!\nabla \hat V^2}{8m^3}
+y\!\cdot\!\nabla(\hat V^{2k}\hat U^2)
\nonumber
\\
&&
\qquad\qquad\qquad\qquad\qquad
+\frac{(1+k)\hat U^2}{m k}
+\frac{2(1+k)\hat V^{2k}\hat U^2}{k}
+\frac{y\!\cdot\!\nabla\,\hat U^2}{2m}
\Big]\,dy
\nonumber
\\[1ex]
&&=
\int\limits\sb{\R^n}
\Big[
\frac{\hat V^2}{4m^3 k}
+2\hat V^{2k}\hat U^2
-\frac{n\hat V^2}{8m^3}
-n\hat V^{2k}\hat U^2
+\frac{(1+k)\hat U^2}{m k}
+\frac{2(1+k)\hat V^{2k}\hat U^2}{k}
-\frac{n\hat U^2}{2m}
\Big]\,dy
\nonumber
\\[1ex]
&&=
\int\sb{\R^n}
\left[
\frac{\hat V^2}{4m^3}\Big(\frac{1}{k}-\frac{n}{2}\Big)
+
2\Big(2+\frac{1}{k}-\frac n 2\Big)\hat V^{2k}\hat U^2
+
\Big(1+\frac{1}{k}-\frac{n}{2}\Big)\frac{\hat U^2}{m}
\right]\,dy.
\qquad
\nonumber
\qedhere
\end{eqnarray}
\end{proof}

\begin{lemma}\label{lemma-q}
Let
$f\in C^1(\R\setminus\{0\})$
satisfy
$f(\tau)=\abs{\tau}^k+O(\abs{\tau}^K)$,
$\tau\in\R$.

\begin{enumerate}
\item
Assume that
in the assumption~\eqref{ass-fp}
either
$
k\in(0,2/n),
$
or
$
k=2/n$,
$K>4/n$.
Then there is $\omega\sb\ast\in(\omega\sb 2,m)$
such that
$\p\sb\omega Q(\phi\sb\omega)<0$
for
$\omega\in(\omega\sb\ast,m)$.
\item
If in the assumption~\eqref{ass-fp}
one has $k\in(2/n,\,2/(n-2))$ (any $k>2/n$ if $n\le 2$),
then there is $\omega\sb\ast\in(\omega\sb 2,m)$
such that
$\p\sb\omega Q(\phi\sb\omega)>0$
for
$\omega\in(\omega\sb\ast,m)$.
\end{enumerate}
\end{lemma}

Above,
$\omega\sb 2=\sqrt{m^2-\epsilon_2^2}$,
with $\epsilon_2>0$ from
Theorem~\ref{theorem-existence-c1}~\itref{theorem-existence-c1-i}.

\begin{proof}
We recall that
$
v(x,\omega)=\epsilon^{\frac 1 k}
\big(
\hat V(\epsilon x)+\tilde V(\epsilon x,\epsilon)
\big),
$
$
u(x,\omega)=\epsilon^{\frac 1 k+1}
\big(
\hat U(\epsilon x)+\tilde U(\epsilon x,\epsilon)
\big)
$
(cf. Theorem~\ref{theorem-solitary-waves});
\[
Q(\phi\sb\omega)
=\int\sb{\R^n}\abs{\phi\sb\omega(x)}^2\,dx
=\epsilon^{\frac 2 k-n}
\int\sb{\R^n}
\left(
V(\abs{y},\epsilon)^2
+\epsilon^2 U(\abs{y},\epsilon)^2
\right)
\,dy.
\]
Let us evaluate the contribution to the derivative of
$Q(\phi\sb\omega)$
with respect to $\epsilon$:
\begin{eqnarray*}
&&
\p\sb\epsilon Q
=
\Big(\frac{2}{k}-n\Big)
\epsilon^{\frac{2}{k}-n-1}
(\langle V,V\rangle
+\epsilon^2\langle U,U\rangle)
+
\epsilon^{\frac{2}{k}-n}
\p\sb\epsilon
\left(
\langle V,V\rangle
+\epsilon^2\langle U,U\rangle
\right)
\\
&&
=
\Big(\frac{2}{k}-n\Big)
\epsilon^{\frac{2}{k}-n-1}
(\langle V,V\rangle
+\epsilon^2\langle U,U\rangle)
+
\epsilon^{\frac{2}{k}-n}
\left(
2\langle \hat V,\p\sb\epsilon\tilde V\rangle
+2\epsilon\langle\hat U,\hat U\rangle
+O(\epsilon^{4\varkappa-1})
\right).
\end{eqnarray*}
The estimate on the error terms in the right-hand side,
such as
$\langle\tilde V,\p\sb\epsilon\tilde V\rangle=O(\epsilon^{4\varkappa-1})$,
follows from
\eqref{def-V-U}, \eqref{def-V-U-hat},
and $X^1$-bounds on $\tilde W$ and $\p\sb\epsilon\tilde W$
from
Theorem~\ref{theorem-solitary-waves}~\itref{theorem-solitary-waves-v}
and
Theorem~\ref{theorem-existence-c1}~\itref{theorem-existence-c1-i},
respectively.
By Lemma~\ref{lemma-hat-v-tilde-v-p},
in the non-critical case, when
$k\ne 2/n$ and $K>k$,
one has
\begin{eqnarray}
\p\sb\epsilon Q
&=&
\Big(\frac{2}{k}-n\Big)
\epsilon^{\frac 2 k -n -1}
\langle \hat V,\hat V\rangle
+
O(\epsilon^{\frac 2 k -n}\epsilon^{\frac{2K}{k}-3})
+
o(\epsilon^{\frac 2 k -n -1})
\nonumber
\\
&=&
\Big(\frac{2}{k}-n\Big)
\epsilon^{\frac 2 k -n -1}
\langle \hat V,\hat V\rangle
+
o(\epsilon^{\frac 2 k -n -1});
\nonumber
\end{eqnarray}
hence, for $\epsilon>0$ sufficiently small,
the sign of $\p\sb\epsilon Q$
is determined by the sign of $\frac 2 k -n$.
Thus,
if $k\in(0,2/n)$,
one has
$\p\sb\omega Q=-\frac{\epsilon}{\omega}\p\sb\epsilon Q<0$
as long as $\omega<m$ is sufficiently close to $m$.
In the critical case $k=2/n$,
again by Lemma~\ref{lemma-hat-v-tilde-v-p},
\[
\p\sb\epsilon
Q(\omega)
=2\langle \hat V,\p\sb\epsilon\tilde V\rangle
+2\epsilon\langle \hat U,\hat U\rangle
+O(\epsilon^{4\varkappa-1})
\]
\[
=2\epsilon
\int\sb{\R^n}(4m \hat V^{2k}\hat U^2+\hat U^2)\,dy
+2\epsilon
\int\sb{\R^n}\hat U^2\,dy
+O(
\epsilon^{\frac{2K}{k}-3}
+\epsilon^{4\varkappa-1}
+\epsilon^{2\varkappa+1}
)
+o(\epsilon).
\]
If
$
K/k>2,
$
$
\varkappa
=\min\Big(1,\frac{K}{k}-1\Big)=1,
$
then, for $\epsilon>0$ sufficiently small,
the above is dominated by the first term
of order one in $\epsilon$,
hence is strictly positive.
Thus, in this case,
$\p\sb\omega Q=-\frac{\epsilon}{\omega}\p\sb\epsilon Q<0$
as long as $\omega<m$ is sufficiently close to $m$.
This finishes the proof of Lemma~\ref{lemma-q}.
\end{proof}

This concludes the proof of
Theorem~\ref{theorem-existence-c1}~\itref{theorem-existence-c1-ii}.


\appendix

\section{Smoothness of NLS groundstates}
\label{sect-nls-smooth}

We start with the properties
of the profiles of solitary wave solutions
to the nonlinear Schr\"odinger equation.

\begin{lemma}\label{lemma-bl}
Let $n\ge 1$
and $k>0$.
If $n\ge 3$, additionally assume that $k<\frac{2}{n-2}$.
Then there is a unique positive
spherically symmetric
monotonically decaying
solution
$u_k\in H^1(\R^n)\cap C^2(\R^n)$
to the equation
\begin{eqnarray}\label{stationary-kg}
-\frac{u}{2m}=-\frac{\Delta u}{2m}-\abs{u}^{2k}u,
\qquad
u(x)\in\R,
\quad
x\in\R^n.
\end{eqnarray}
For any $\delta<1$ there is $C\sb\delta<\infty$ such that
\[
\abs{u_k(r)}
+
\abs{\p\sb r u_k(r)}
\le
C\sb\delta e^{-\delta r},
\qquad
r\ge 0.
\]
For any $s<\frac{n}{2}+2$
one has $u_k\in H^s(\R^n)$.
As $\abs{x}\to\infty$,
the function
$u_k$ is \emph{strictly} monotonically decreasing.

There are $0<c_{n,k}<C_{n,k}<\infty$
such that
\begin{eqnarray}\label{exp-decay-sharp}
c_{n,k}\langle x\rangle^{-(n-1)/2}e^{-\abs{x}}
\le
\abs{u_k(x)}
\le
C_{n,k}\langle x\rangle^{-(n-1)/2}e^{-\abs{x}},
\qquad
x\in\R^n.
\end{eqnarray}

If $n\ge 3$ and $k\ge\frac{2}{n-2}$,
then \eqref{stationary-kg}
has no $H^1$ solutions.
\end{lemma}


Let us give an extension of Lemma~\ref{lemma-bl},
deriving optimal regularity
of the groundstates of the nonlinear
Schr\"odinger equation in Sobolev spaces.

\medskip

\begin{proof}
The absence of $H^1$-solutions
for $k\geq \frac{2}{n-2}$, $n\geq 3$
is proved in \cite[Section 2.1]{MR695535}
via Pohozhaev's identities.
The uniqueness of a symmetric solution $u>0$
is proved in \cite{MR969899,MR1201323}.
The inclusion
$u_k\in H^1(\R^n)\cap C^2(\R^n)$,
monotonicity,
and the exponential decay of $u_k$
follows from \cite{MR695535} for $n\ge 3$ and $n=1$;
for $n=2$,
the inclusion
$u_k\in H^1(\R^2)\cap C^2(\R^2)$
is proved in \cite{MR734575},
and the exponential decay is proved
following the lines of \cite[Lemma 3.1]{MR3530581}.

The exponential decay of $\p\sb r u$
could be shown as follows.
The groundstate profile $u_k$,
considered as a function of $r=\abs{x}$,
satisfies the equation
\[
-\p\sb r^2 u_k-\frac{n-1}{r}\p\sb r u_k-2m u_k^{2k+1}+u_k
=0,
\qquad
r>0.
\]
Multiplying this by $r^{n-1}$,
one has:
\begin{eqnarray}\label{p-r-r-u}
-
\p\sb r\big(r^{n-1}\p\sb r u_k\big)
-2m r^{n-1}u_k^{2k+1}+r^{n-1}u_k
=0.
\end{eqnarray}
Integrating this relation from zero to some $R>0$
and taking into account the exponential decay of $u_k$,
one concludes that there exists a finite limit
$c=\lim\sb{r\to\infty}r^{n-1}\p\sb r u_k$.
This limit has to be equal to zero
or else there is $r_0>0$ such that
$\abs{r^{n-1}\p\sb r u_k}\ge c/2$
for $r\ge r_0$,
hence $\p\sb r u_k\le -c/(2r^{n-1})$,
$u_k(r)\ge c/(2(n-2) r^{n-2})$ for $r\ge r_0$
for $n\ne 2$
or
$u_k(r)\ge (c\ln r)/2$ for $r\ge r_0$ for $n=2$
or $\lim\sb{r\to\infty}\p\sb r u_k=c>0$ for $n=1$;
so, for any $n\ge 1$,
we arrive at a contradiction with the exponential decay
of $u_k$.
So,
$\lim\sb{r\to\infty}r^{n-1}\p\sb r u_k=0$.
Integrating \eqref{p-r-r-u} from some $R>0$ to infinity,
one has:
\[
R^{n-1}\p\sb r u_k(R)
=
\int_R^\infty
\left(2m u_k^{2k+1}-u_k\right)r^{n-1}\,dr.
\]
Now the exponential decay of $\p\sb r u_k(r)$
follows from the exponential decay of $u_k(r)$.

The strict monotonicity of $u_k$ is proved
as follows.
Assume that
\begin{eqnarray}\label{uv}
\p_r u_k =w_k,
\qquad
\p_r w_k=-\frac{n-1}{r}w_k
-2m\abs{u_k}^{2k}u_k,
\qquad
r>0,
\end{eqnarray}
and that $u_k'(r_0)=0$
(here $u_k$ is considered as a function
of $r=\abs{x}$)
at some $r_0>0$.
Since $u_k(r)$ is monotonically decreasing,
$w_k=\p\sb r u_k\in C^1(\R\sb{+})$
satisfies
$w_k\le 0$.
Once we know that $w_k(r_0)=u_k'(r_0)=0$,
we conclude that
$w_k$ has a local maximum at $r_0$,
so that $\p\sb r w_k(r_0)=0$.
Now from the second equation in \eqref{uv}
one would conclude that $u_k(r_0)=0$,
in contradiction to the strict positivity
of the groundstate $u_k$.

The estimate \eqref{exp-decay-sharp}
follows from Lemma~\ref{lemma-v-exp}.

%
%
%

Let us prove the improved Sobolev regularity
\[
u_k\in H^s(\R^n),
\qquad
\forall s<\frac n 2+2.
\]
Considering $u$ as a function of $r\ge 0$,
we write \eqref{stationary-kg}
in the form
\begin{eqnarray}\label{upp}
u''
=u-2m u^{1+2k}-\frac{n-1}{r}u',
\qquad
r>0.
\end{eqnarray}
Denote
\[
f(r)=\frac{u'(r)}{u(r)};
\]
note that $f$ is non-positive since $u$ is non-increasing.

\begin{lemma}\label{lemma-up-u-bounded}
There is $c_1<\infty$ such that
\begin{eqnarray}\label{up-u-bounded}
\abs{f(r)}
=\Abs{\frac{u'(r)}{u(r)}}
\le c_1\frac{r}{\langle r\rangle},
\qquad r>0.
\end{eqnarray}
\end{lemma}

\begin{proof}
Using \eqref{upp}, we arrive at
\begin{eqnarray}\label{f-p}
f'(r)=\frac{u''}{u}-f^2
=1-2m u^{2k}-\frac{n-1}{r}f-f^2,
\qquad
r>0.
\end{eqnarray}
We already mentioned that $f(r)\le 0$,
$r>0$.
If $f(r)$ were unbounded from below
for $r\ge 1$,
then it would blow up, going to $-\infty$ at some $r_0<\infty$.
Indeed,
fix $a=\sup\sb{r>0}\abs{1-2m u^{2k}}<\infty$,
and, assuming $f\to-\infty$,
consider the smallest $r_1\ge 1$
such that $-f(r)\geq 
4\max\{\frac{n-1}{r},\sqrt{a}\}$ for $r\geq r_1$;
then $\abs{f(r)}$ grows faster than the solution to
$F'=F^2/2-a/2$ with the same initial data $F(r_1)=f(r_1)$,
while this solution
blows up  in the interval $[r_1,r_1-4/f(r_1)]$).
Of course,
the blow-up of $f$ at some $r<\infty$
would contradict $u\in C^2$.
We conclude that $\abs{f}$ remains bounded as $r\to+\infty$.
We also conclude from \eqref{upp}
and from the inclusion $u\in C^2(\R^n)$
(considered as a function of $x\in\R^n$)
that
$u'/r$ remains bounded near $r=0$;
due to $u(0)>0$,
the bound \eqref{up-u-bounded} follows.
\end{proof}

We claim that for $j\ge 2$ there are $C\sb j<\infty$
such that
\begin{eqnarray}\label{uppp}
\abs{u^{(j)}(r)}
\le C\sb j\left(\frac{\langle r\rangle}{r}\right)^{j-2}
u(r),
\qquad
r>0,
\qquad
j\ge 2.
\end{eqnarray}
The proof is by induction.
For $j=2$, the statement follows from \eqref{upp}
and Lemma~\ref{lemma-up-u-bounded}.
Assume that \eqref{uppp} is proved for $j\le l$,
with some $l\in\N$.
To get $u^{(l+1)}$ out of \eqref{upp},
one takes the derivative of the expression
for $u^{(l)}$, 
\begin{eqnarray}\label{ulll}
 u^{(l+1)}
=u^{(l-1)}
-
\left(
2m u^{1+2k}
+(n-1)
\frac{u'}{r}
\right)^{(l-1)}.
\end{eqnarray}
We notice that each of $l-1$ derivatives
of the expression in the brackets,
when acting on $u$,
contributes a factor of $u'/u$
(which is uniformly bounded);
or else it changes one of the factors $u^{(i)}$ to $u^{(i+1)}$
with $i<l$
(worsening the bound by $\langle r\rangle/r$
by the induction assumptions);
or else it acts on $1/r$, contributing another $1/r$;
therefore, after each differentiation,
the resulting estimate deteriorates
by the factor $C\langle r\rangle/r$,
with some $C<\infty$.
This allows to bound \eqref{ulll}
by $(\langle r\rangle/r)^{l-1}$ (times a constant factor),
concluding the induction argument.

The inequality \eqref{uppp} and the interpolation arguments
show that $u\in H^s(\R^n)$
as long as $\abs{x}^{-(s-2)}$ is $L^2$ locally near the origin;
this imposes the restriction $s-2<n/2$.
\end{proof}

\bibliographystyle{sima-doi}
\bibliography{dirac-existence}
\end{document}